\pdfoutput=1  
\documentclass[11pt,reqno]{amsart}
\usepackage[utf8]{inputenc}
\usepackage{amsmath,amssymb,amsthm}
\usepackage{booktabs}
\usepackage{array}
\usepackage{longtable}
\usepackage[table]{xcolor}
\definecolor{hproved}{HTML}{2A78D6}   
\definecolor{hverif}{HTML}{EB6834}    
\definecolor{hband}{HTML}{F2F1EC}     
\definecolor{hext}{HTML}{6B6560}      
\definecolor{hrule}{HTML}{B9B7AE}     
\newcommand{\stproved}{\textcolor{hproved}{\small proved}}
\newcommand{\stverif}{\textcolor{hverif}{\small verified}}
\newcommand{\stext}{\textcolor{hext}{\small external}}
\newlength{\keyw}
\newcommand{\keybox}[1]{%
  \par\medskip\noindent
  \setlength{\keyw}{\linewidth}%
  \addtolength{\keyw}{-4\fboxsep}%
  \addtolength{\keyw}{-4\fboxrule}%
  \fcolorbox{hrule}{hband}{\parbox{\keyw}{\smallskip #1\smallskip}}%
  \par\medskip}
\usepackage{graphicx}
\usepackage[margin=3.1cm]{geometry}
\usepackage[colorlinks=true,linkcolor=blue!55!black,citecolor=blue!55!black,urlcolor=blue!55!black]{hyperref}
\usepackage{caption}
\makeatletter
\renewcommand\section{\@startsection{section}{1}%
  \z@{\linespacing\@plus\linespacing}{.5\linespacing}%
  {\normalfont\large\bfseries\centering}}
\makeatother

\theoremstyle{plain}
\newtheorem{theorem}{Theorem}[section]
\newtheorem{lemma}[theorem]{Lemma}
\newtheorem{corollary}[theorem]{Corollary}
\newtheorem{proposition}[theorem]{Proposition}
\newtheorem{conjecture}[theorem]{Conjecture}
\theoremstyle{definition}
\newtheorem{remark}[theorem]{Remark}

\newtheorem{problem}[theorem]{Problem}

\newtheorem{observation}[theorem]{Observation}

\newcommand{\ZZ}{\mathbb{Z}}

\newcommand{\zt}{\mu_t}
\newcommand{\Sp}{\operatorname{Sp}}
\newcommand{\spc}{\mathrm{sp}}
\newcommand{\Newt}{\operatorname{Newt}}
\newcommand{\sgn}{\operatorname{sgn}}

\begin{document}

\title[Torsion filters at a root-of-unity orbit]
{Schur polynomials twisted by roots of unity\\
and reciprocal pairs: torsion filters,\\
fusion quotients, and total unimodularity\\
at odd order}

\author{Carles Mar\'in}
\address{Independent researcher}
\email{karlesmarin@gmail.com}

\date{August 20, 2026}


\begin{abstract}
Write $\zt$ for the $t$-th roots of unity and $z^{\pm1}$ for $r$ free reciprocal pairs. We study
$\Phi_{t,r}(\beta)=s_{\lambda}(\zt,z^{\pm1})$, $\beta=\lambda+\delta$, and the question the
companion paper left open after $r=1$: when does it vanish?

We factor the evaluation into classical branching followed by a torsion filter, and the shape
depends on the parity of $t$: the point lies in the
orthogonal group with determinant $(-1)^{t+1}$. For odd $t$ it sits in the
identity component: an ordinary restriction $SO_{2R'+1}\downarrow SO_{2m'+1}\times SO_{2r}$,
the filter an odd orthogonal character at a principal element of order $h+1$. For even $t$
in the other: a twining, a virtual expansion, and a torsion element regular but
not principal; there we prove the filter, with its sign.

One description covers both: the filter is nonzero exactly when the shifted point is regular
semisimple. Both are minimal-level fusion projections: the even of type $C$, the odd's tensor sector of type $B$. Affine folding accounts for $0,\pm1$; what it does not survives as conjectures.

The highest surviving weight is the dominant
vertex of the numerator's Newton polytope minus the denominator's --- the latter proved, the
former conditional on a single-orbit property --- and the class there is conjecturally primitive,
$\pm$ the generator of the rank-one quotient. For odd $t$ and one $\Lambda$ that numerator is a signed transversal count in $\{0,\pm1\}$ by the
equal-rank character formula, leaving one division. We invert it in closed form, along an
arithmetic progression; the quotient is $\pm\epsilon_t\det M$ for an
explicit $0/{\pm}1$ matrix --- an interval matrix up to signs, hence totally unimodular, which
settles (L1). The fibre count is a permanent, odd only when $1$, so at a dominant index a multi-hit fibre sums to
zero. Two extremal statements remain. What is unproved is measured, in both parities.
\end{abstract}

\maketitle

\begin{center}
\fcolorbox{hrule}{hband}{%
\parbox{0.82\linewidth}{\small\smallskip
This paper belongs to an open-ended series on the evaluation of Schur polynomials at a full
root-of-unity orbit together with free reciprocal pairs. It continues the companion paper
\cite{PaperI}, which settled the case of a single free pair. The entries of the series refer to one
another by role rather than by number, so that a later entry does not renumber the earlier ones.
\smallskip}}
\end{center}

\medskip

\begin{figure}[h]
\centering
\includegraphics[width=\textwidth]{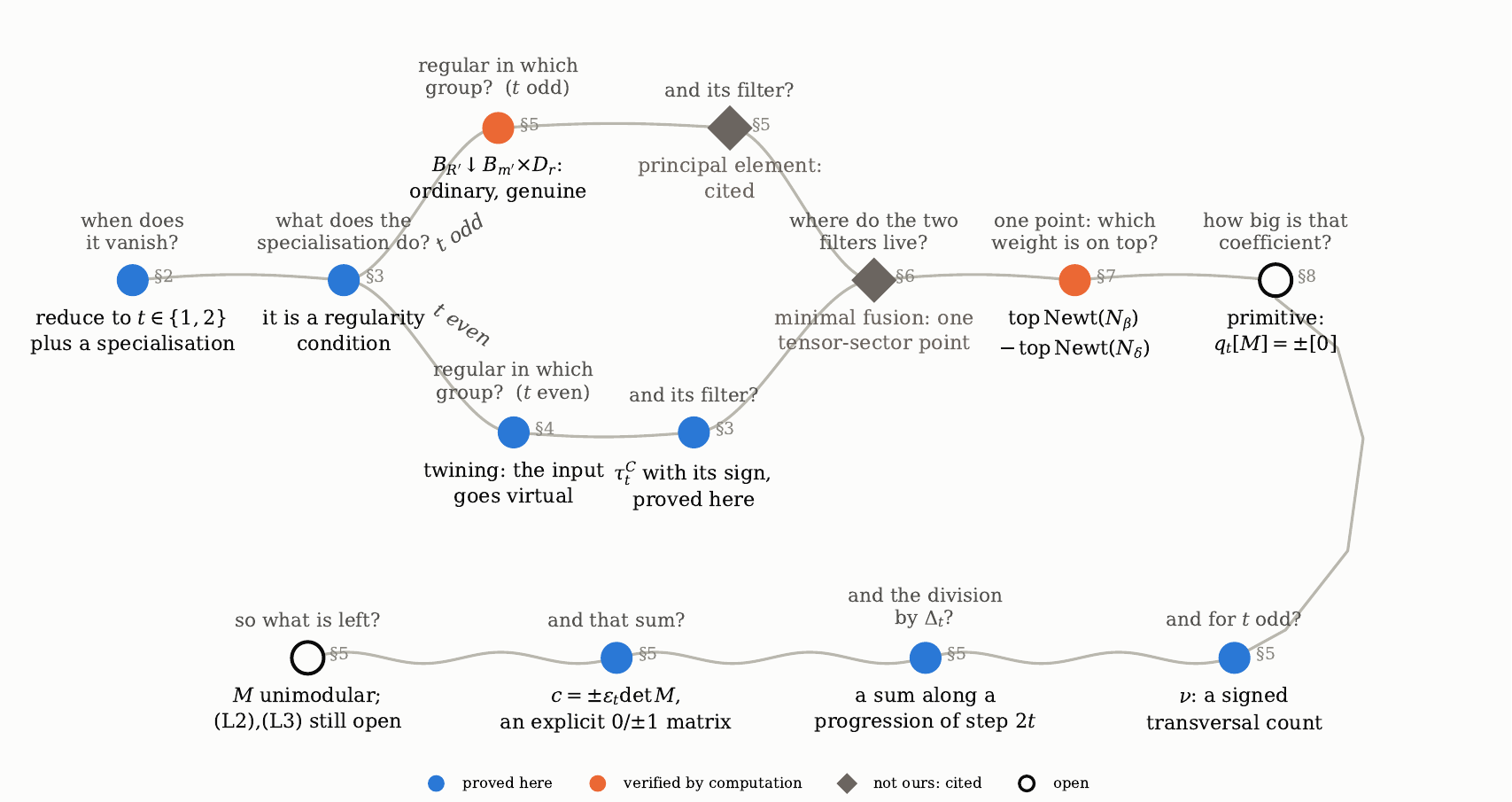}
\caption{The paper as one thread, with each question answered by the pearl below it and each question
picking up a word from the answer before. The thread \emph{forks}: once the filter is known to be a
regularity condition, the next question is \emph{regular in which group}, and the two parities answer
it differently --- an ordinary restriction $B\to B\times D$ on a genuine character above, a twining
$\rightsquigarrow C\times C$ on a virtual one below. They close again at the fusion quotient, where
the two filters turn out to be the same kind of object, and from there the chain is single-stranded
to the end. Pearls carry the colour of their status, because some are not ours: the affine folding is
standard at roots of unity, and the odd filter is a corollary of the principal-element theorem of
Nadimpalli, Pattanayak and Prasad, so both are drawn as cited. The main thread ends on an open pearl,
and that is not a flourish --- the factorisation is complete on both branches, and the extremal
coefficient is the part none of it explains. The two parities do not stop at the same distance from
it, so the thread grows a tail on the odd side. The tail wraps to the row below and is read
\emph{from right to left}: the numerator of that coefficient is a signed transversal count and is
proved (Proposition \ref{prop:transversal}); the division that is left inverts in closed form, as a
sum along an arithmetic progression (Proposition \ref{prop:divided}); that sum, regrouped, is a
determinant (Proposition \ref{prop:determinant}); and the unimodularity of that matrix is now proved
too (Lemma \ref{lem:c1p}, Corollary \ref{cor:unimodular}), so what is finally left is (L2) and (L3).
The general pearl stays open; the tail says how much of it is not.}
\label{fig:thread}
\end{figure}

\section{Introduction}

\subsection{The object}
Fix $t\ge 1$ and $r\ge 1$, put $N=t+2r$, let $\zeta=e^{2\pi i/t}$ and
$\zt=\{1,\zeta,\dots,\zeta^{t-1}\}$. For a strictly decreasing $\beta\in\ZZ^{N}$ \emph{normalised by} $\beta_N=0$ write
$\delta=(N-1,N-2,\dots,0)$ and $\lambda=\beta-\delta$, so that $\lambda$ is a partition; and set
\begin{equation}\label{eq:def}
\Phi_{t,r}(\beta)\;=\;\frac{\det\bigl(x_i^{\beta_j}\bigr)_{i,j=1}^{N}}
                          {\det\bigl(x_i^{\delta_j}\bigr)_{i,j=1}^{N}},
\qquad
x=(\underbrace{1,\zeta,\dots,\zeta^{t-1}}_{\zt},\;z_1,z_1^{-1},\dots,z_r,z_r^{-1}).
\end{equation}
This is the Schur polynomial $s_\lambda$ evaluated at a full root-of-unity orbit together with $r$
free reciprocal pairs. The normalisation $\beta_N=0$ is not cosmetic. Without it $\lambda$ is only a
weakly decreasing integer vector, possibly with $\lambda_N<0$, and the quotient is a Laurent rather
than a polynomial character; and replacing $\beta$ by $\beta+c(1,\dots,1)$ multiplies the numerator
by $\bigl(\prod_ix_i\bigr)^c=(-1)^{c(t+1)}$, which leaves the vanishing locus alone but flips signs
at even $t$. Since this paper pins signs, the normalisation has to be fixed once and kept. The companion paper \cite{PaperI} treated $r=1$ and gave a closed form; the vanishing locus for
$r\ge2$ was left open there, and is the subject of this paper.

\subsection{What is already known, and whose it is}
Two mechanisms meet in \eqref{eq:def} and both are classical. Evaluating a Schur function on a
root-of-unity orbit factors it through the $t$-quotient of $\lambda$, and it vanishes as soon as some
residue class of $\beta$ modulo $t$ is empty; this is Littlewood's and, in the form we use,
Ayyer--Kumari's \cite{LR34,AK22,Kum24} \stext{} --- and it has an antecedent that its own
rediscoverers are careful to name. Ayyer and Kumari record that Lecouvey \cite{Lec09a} \stext, in
what they call an \emph{almost unnoticed} work, had already generalised the classical character
factorisations, which they then found independently \cite{AK22}; and that \cite{Lec09b} \stext, in
their words \emph{another little-known work}, carried them to the universal characters of types
$B$, $C$ and $D$, where they were found independently again, by Albion \cite{Alb23}. What Lecouvey
proves there is worth stating, because it is the shape of this paper one level up. Writing
$p_\ell\circ s^{\mathfrak g}_\lambda=\sum_\mu a^{\ell,\mathfrak g}_{\lambda,\mu}s^{\mathfrak g}_\mu$
for the power-sum plethysm, in the stable range $n>\max(\ell\,l(\lambda),l(\lambda'))$ the
coefficients \emph{are branching coefficients}: his Theorem 4.5.1 reads
\[
a^{\ell,so}_{\lambda,\mu}
=\varepsilon(\mu)\,\bigl[V^{so_{2n+1}}(\lambda):V^{L_{\ell,\mu}}(\gamma_{\ell,\mu})\bigr],
\]
a signed multiplicity for restriction to a Levi $L_{\ell,\mu}$ that his \S4 algorithm produces
together with $\gamma_{\ell,\mu}$; at $\ell=2$ the Levi is $gl_n$ and his (23)--(24) expand the
multiplicity into Littlewood--Richardson coefficients, and the $q$-analogue, in which these become
parabolic Kazhdan--Lusztig polynomials, is \cite{Lec09a}. Four things separate that from
\eqref{eq:def}, and they are the four to keep in view: his is a plethysm and ours an evaluation; his
branching is to a Levi and ours to an equal-rank pair, which is not one; his coefficients are sums
of Littlewood--Richardson numbers and ours a filter with values in $\{0,\pm1\}$; and he works in a
stable range where we assume none.
The nearest current literature on classical characters
at roots of unity is \cite{Alb23,AK26,Albion} \stext, which revisits those specialisations systematically
but does not consider an alphabet mixing a full orbit with free reciprocal pairs. Evaluating on reciprocal pairs factors it into symplectic and orthogonal
characters; this is Ciucu--Krattenthaler and Ayyer--Behrend \cite{CK09,AB19} \stext, and the rectangular case is older
than either of those papers: Ayyer and Kumari record that specialising one half of the variables to
the reciprocals of the other half was considered for rectangular shapes by Okada \cite{Oka98}
\stext\ and by Ciucu--Krattenthaler \cite{CK09}, and generalised to self-dual shapes by
Ayyer--Behrend \cite{AB19}. We cite Okada on their authority; his paper is not one we have read.
What is peculiar to \eqref{eq:def}
is that both are present at once, and they meet at the fixed points of $x\mapsto x^{-1}$ inside
$\zt$: at $+1$ always, and at $-1$ exactly when $t$ is even.

We take no credit for either mechanism, nor for the two facts about them that we use most. The first
is that a character of a simple module evaluated at a suitable element of finite order is $0$ or
$\pm1$. For the Coxeter element that is Kostant's theorem \cite{Kostant,Prasad}. For a
\emph{principal} element of any order the governing statement is Theorem 4.1 of Nadimpalli,
Pattanayak and Prasad \cite{NPP} --- which is more general than a bound: it computes the value in
terms of the dimension of a representation of another group, and collapses to a unit only when that
group is a torus, which is our situation --- with the accompanying constant classified in general by
Polo \cite{Polo}. Which of these covers us
turns out to depend on the parity of $t$, and that is not a footnote but half of this paper: the odd
torsion element \emph{is} principal, so its filter is a corollary of \cite{NPP} and we cite it; the
even one is regular but not principal, no available theorem reaches it, and we prove the case we need
directly. The second fact is that a weight on a wall of the relevant affine Weyl group is killed
while a regular weight folds into the fundamental alcove with the sign of the folding element
(Andersen--Stroppel \cite{AndersenStroppel} \stext, in the form we quote it).

A third classical mechanism enters halfway through, and only on one branch. The pair
$SO_{2R'+1}\downarrow SO_{2m'+1}\times SO_{2r}$ that the odd side produces has \emph{equal rank}, and
equal-rank pairs have a character formula built for them: Gross, Kostant, Ramond and Sternberg
\cite{GKRS} \stext, with Kostant's cubic Dirac operator behind it \cite{Kostant99} and Landweber's
extension to loop groups \cite{Landweber}; the branching rule itself is Koike--Terada
\cite{KoikeTerada} \stext. Specialising that formula at our torsion element is ours; the formula is
not. The same division of labour applies to the operator language in which we then invert what it
leaves over: divided differences in Demazure's sense, in the form Harada, Landweber and Sjamaar give
them in equivariant K-theory \cite{HLS} \stext, together with the maximal-rank sequel of Landweber
and Sjamaar \cite{LS13} \stext, which is the setting our pair belongs to and in which our numerator
is a multiplet in their sense. What we add there is a coefficientwise inverse and a determinant, not
a theory.

\subsection{What this paper does}
The paper is a chain of six questions, each one picking up a word from the answer before it, and a
seventh that only one of the two branches can ask. Figure \ref{fig:thread} is that chain drawn ---
with the change the subject forces, which is that it \emph{forks} in the middle, closes again, and
then grows a short tail on the odd side.

\emph{When does it vanish?} \S\ref{sec:red} reduces every $t$ to $t\in\{1,2\}$ together with a
\textbf{specialisation}, and shows that the difficulty moves rather than shrinks.

\emph{What does the specialisation do?} \S\ref{sec:filter} computes it in closed form and then, in
Lemma \ref{lem:regular}, says what it is: three species of wall are one condition, and the filter is
nonzero exactly when the shifted torsion point is \textbf{regular} in the group.

\emph{Regular in \emph{which} group?} Here the chain forks, and \S\ref{sec:comp} and \S\ref{sec:odd}
walk the two branches. For even $t$ the evaluation point sits in the non-identity component of an
orthogonal group, the first step is a twining, and what enters the filter is a \emph{virtual}
character of $\Sp_{2R}$. For odd $t$ it sits in the identity component, the step is an ordinary
restriction $SO_{2R'+1}\downarrow SO_{2m'+1}\times SO_{2r}$, and what enters is a genuine one. Two
switches govern the difference, and Proposition \ref{prop:parity} names both: the determinant
$(-1)^{t+1}$, and whether $2$ is invertible modulo $t$. Each branch ends with its own \textbf{filter}.

\emph{And where do the two filters live?} \S\ref{sec:fusion} closes the fork. They are the same kind
of object: the minimal-level fusion projections of their types, leaving the tensor characters a
\textbf{single available point}. (The even alcove is that point; the odd alcove has a second one, and
it is spinorial --- so the odd fusion ring is not rank one, only its tensor sector is. The
qualification is not decoration: it is the whole content of \S\ref{sec:fusion} in the odd case.) That
is why both take the values $0,\pm1$: the target is a rank-one ring, and
every surviving weight folds to the same place, namely the vacuum --- checked directly in both
parities, where the survivors reach exactly one alcove point and it is the one $\eta=0$ reaches
\stverif.

\emph{If everything folds to one point, which weight is on top?} \S\ref{sec:top} answers with the
dominant vertex of the Newton polytope minus that of the denominator, the latter proved outright in
Proposition \ref{prop:newtden}; what remains conditional is that no cancellation reaches the top, and
that is isolated as Conjecture \ref{conj:A}. What comes out is a single extremal
\textbf{coefficient}.

\emph{And how big is that coefficient?} Always $\pm1$, on everything we can compute, in both
parities --- and nothing in the factorisation explains it. \S\ref{sec:unit} states it as Conjecture
\ref{conj:H}: the extremal class --- virtual on the even branch, a genuine multiplicity space on the
odd one --- is \textbf{primitive} in the rank-one quotient.

\emph{And how close is primitive to proved?} On the even side, not close. On the odd side, for one
$\Lambda$ at a time, one division. The pair
$SO_{2R'+1}\downarrow SO_{2m'+1}\times SO_{2r}$ has equal rank, so the character formula written for
such pairs applies; specialising it at the torsion point turns the numerator of that coefficient
into a signed count of \textbf{transversals} of the residue classes of the doubled shifted exponents
$2(\Lambda+\rho)$, with values in $\{0,\pm1\}$ by Proposition \ref{prop:transversal}. What is left
\emph{there} is to divide by the spinor denominator without creating a $2$ --- that is (L1).

\emph{And what \emph{is} that division?} \S\ref{sec:invert} answers it, and the answer changes the
category of the question. The denominator is $\psi^t(\Delta_1)$, the Adams image of a rank-one Weyl
denominator, so multiplying by it is a product of $r$ commuting centred difference operators; and on
functions of finite support that product inverts in closed form, as a sum along an arithmetic
\textbf{progression} of step $2t$ (Proposition \ref{prop:divided}). (L1) stops being a statement
about an algorithm --- a back-substitution that never doubles --- and becomes a statement about a
signed count on one \emph{fibre} of that sum --- one arithmetic progression, all of whose terms
land on the same $X$.

\emph{Then why does a fibre with more than one hit vanish?} It need not, for (L1) --- a fibre may
sum to $\pm1$ without vanishing and (L1) is still satisfied --- but on everything we compute it
vanishes, which is a strictly stronger statement and is Theorem \ref{thm:cancel}. What makes it
vanish is not a property of $M$ but the dominance with which the $X$ are enumerated, through Lemma
\ref{lem:laminar}; \S\ref{sec:whatcancels} keeps the mechanisms that failed first, because the
order in which they failed is what located the right one; Remark \ref{rem:nocircuit} records the
one that looks likeliest of all and does not work. Four mechanisms
are tested and fail: a product of independent binary choices, an involution exchanging one index,
pure Weyl antisymmetry, and affine folding --- the last killed by the observation that a shift by
$2t$ is a product of two affine reflections and so carries sign $+1$. What does close is a narrower
move, an exchange inside a single folded class with the chirality coupled to the folding sign; it
accounts for every cancelling pair. The \textbf{cancellation} itself is no longer open: it is Theorem
\ref{thm:cancel}, proved by a different route, through the permanent of Lemma \ref{lem:permanent} and
the parity $\operatorname{per}N\equiv\det M\pmod 2$. What stays open is only the \emph{bijective}
reading: assembling those moves into a fixed-point-free sign-reversing involution.

\emph{Is there a reading in which it is not a cancellation at all?} Yes, and it comes from noticing
that every attempt above sums over \emph{subsets} and then explains the sign separately. An element
of $W^1$ is not a subset but a signed permutation, and a signed sum over perfect matchings is a
determinant: $c=\pm\epsilon_t\det M$ for an explicit $0/{\pm}1$ matrix (Proposition
\ref{prop:determinant}). (L1) becomes total \textbf{unimodularity} of $M$ --- the first form of the
question with an existing body of technique behind it, and the first we can test exhaustively rather
than sample. At $t=3$, $r=2$ the nonzero-determinant content of (L1) is carried by three matrices.

Passing from one $\Lambda$ to Conjecture \ref{conj:H} still needs the two extremal statements (L2)
and (L3), and we do not have them either. \S\ref{sec:map} lays the four conjectures of this paper on
one page and says which implies which, because three of them circle the same difficulty from
different sides. That is the tail of Figure \ref{fig:thread}.

There the chain stops, and the rest of the paper is what one does when a chain stops: look at
the shapes that nearly break it, and try the same computation from the other side.
\S\ref{sec:top16} isolates the stratum of candidates that fail the closed form and asks what
separates them; \S\ref{sec:residue} measures what the classical criterion leaves behind, with one
constraint the locus satisfies for free; \S\ref{sec:square} computes the same coefficient the other
way round and proves that the two orders of processing commute, and at $t=4$ \S\ref{sec:comb} makes
the branching side completely combinatorial; and \S\ref{sec:where}
finds that the change of coordinates does not simplify the cancellation so much as remove it, which
is where we would look next. \S\ref{sec:answers} says what all of this settles of the questions the
companion paper left; \S\ref{sec:verif} is the verification table, \S\ref{sec:attr} separates what we
use from what we claim, \S\ref{sec:open} states the problems we are leaving open, and
\S\ref{sec:closing} says in one page what the whole paper turns on: three formalisms, each producing
the same parity.

Before any of that, here is the odd branch in one statement. It is assembled from results proved in
\S\ref{sec:odd} and nothing in it is new to the paper; what is new is that it is \emph{one}
statement, and a reader who wants the odd case and no more can stop after it.

\begin{theorem}[odd-order coefficient theorem]\label{thm:odd}
Let $t=2m'+1\ge3$ and $r\ge1$, fix a single $\Lambda$ and a \emph{dominant} regular coefficient
index $X=2(\mu+\rho_{D_r})$, $X_1>\dots>X_{r-1}>|X_r|$, and let $M=M(\Lambda,X)$ be
the explicit integer matrix of Proposition \ref{prop:determinant} and $N=N(\Lambda,X)$ its unsigned
companion of Lemma \ref{lem:permanent}. Then:
\begin{enumerate}
\item[\rm(i)] \emph{the coefficient is a determinant}:
$c(\Lambda,X)=\kappa_{t,r}\,\epsilon_t\,\det M$, with $\epsilon_t$ the filter constant and
$\kappa_{t,r}\in\{\pm1\}$ depending on $(t,r)$ alone;
\item[\rm(ii)] \emph{the matrix is totally unimodular}: after ordering the rows and changing signs by
row and by column, $M$ becomes a $0/1$ matrix whose ones are consecutive in every column; hence $M$
is totally unimodular and $\det M\in\{0,\pm1\}$, which settles \emph{(L1)};
\item[\rm(iii)] \emph{the fibre count is a permanent}: $|I_X|=\operatorname{per}N$, and this number
is odd only when it equals $1$;
\item[\rm(iv)] \emph{hence the multi-hit cancellation}:
\[
|I_X|>1\ \Longrightarrow\ \sum_{Y\in I_X}\tilde\nu(Y)=0,
\qquad\text{and}\qquad c(\Lambda,X)\in\{0,\pm1\};
\]
\item[\rm(v)] \emph{and all of it is one condition on a graph}: sending each interval column
$\mathbf1_{[a,b]}$ of $M$ to the oriented edge $e_a-e_{b+1}$ turns every block $B$, with $n$ rows,
into a multigraph $G_B$ whose non-loop part is contained in $K_{2,n-1}+\alpha\omega$, the zero
columns appearing as loops, and whose cycles are its dependencies; and
\[
c(\Lambda,X)\neq0\iff|I_X|=1\iff\text{every }G_B\text{ is a forest.}
\]
\end{enumerate}
\end{theorem}

\noindent
Part (i) is Proposition \ref{prop:determinant}; (ii) is Lemma \ref{lem:c1p} with Corollary
\ref{cor:unimodular}; (iii) is Lemma \ref{lem:permanent} with Lemmas \ref{lem:laminar} and
\ref{lem:sdr}; (iv) is Theorem \ref{thm:cancel}; (v) is Proposition \ref{prop:graph} with Corollary
\ref{cor:forest}. What the theorem does \emph{not} contain is the passage from one $\Lambda$ to the
whole character, which is (L2) and (L3) and stays open.

\section{\texorpdfstring{The reduction to $t\in\{1,2\}$}{The reduction to t in \{1,2\}}}\label{sec:red}

\subsection{Both parities at once}
The involution $x\mapsto x^{-1}$ acts on $\zt$. Its fixed points are $+1$, and also $-1$ when $t$ is
even; every other element is matched with its inverse. Hence, writing
\[
t=2m+2 \quad(t \text{ even}),\qquad t=2m'+1 \quad(t \text{ odd}),
\]
we have the decompositions
\begin{equation}\label{eq:orbit}
\zt=\{1,-1\}\ \sqcup\ \{\xi^{\pm1},\dots,\xi^{\pm m}\}
\qquad\text{and}\qquad
\zt=\{1\}\ \sqcup\ \{\xi^{\pm1},\dots,\xi^{\pm m'}\},
\end{equation}
with $\xi=\zeta$. The right-hand blocks are reciprocal pairs, indistinguishable in \eqref{eq:def}
from the free ones except that they are pinned at roots of unity. Reading \eqref{eq:def} with the
pinned pairs counted among the free ones gives at once:

\begin{proposition}\label{prop:red}
Let $R=r+m$ and $R'=r+m'$. For the same $\beta$,
\[
\Phi_{t,r}\;=\;\Phi_{2,R}\Big|_{\,y=(\xi,\xi^2,\dots,\xi^{m})}\quad(t\text{ even}),
\qquad
\Phi_{t,r}\;=\;\Phi_{1,R'}\Big|_{\,y=(\xi,\xi^2,\dots,\xi^{m'})}\quad(t\text{ odd}),
\]
where $y$ denotes the first $m$ (resp.\ $m'$) free pairs of the right-hand object and the remaining
$r$ are left free. In both cases $N$ is unchanged: $2+2R=t+2r$ and $1+2R'=t+2r$.
\end{proposition}

Proposition~\ref{prop:red} is a regrouping of one and the same alphabet, not a computation, and we
record it as such: it is an accounting identity whose content is entirely in what it makes possible.
It was verified as an implementation control on $6435$ forms for even $t$ and on $260$ forms at
$t=3$, in both cases with no exception \stverif, against a decoy comparing each specialisation with
the object of a neighbouring $\beta$ (which disagrees in $258$ of $259$ consecutive pairs).

\begin{remark}
The two halves of \eqref{eq:orbit} are not symmetric, and the asymmetry is the single structural
difference between the parities: for even $t$ the second fixed point $-1$ is present, for odd $t$ it
is not. Everything below that distinguishes the two cases descends from this.
\end{remark}

\subsection{What the reduction does not do}
The reduction does not make the problem easier by itself. On the populations of \S\ref{sec:residue},
the object $\Phi_{2,R}$ vanishes for only a small fraction of the $\beta$ for which $\Phi_{t,r}$
does: at $t=3$, of the six vanishing forms, none is inherited and all six are created by the
specialisation. The reduction relocates the problem rather than solving it, and the located problem
is the specialisation.

\noindent
One convention before the two branches separate. The cases $t=1$ and $t=2$ are the rank-zero base
cases: there is no frozen block, $\Phi_{t,r}$ is the unspecialised object, and the reduction says
nothing. From here on, statements involving $C_m$ are to be read for even $t\ge4$ and statements
involving $B_{m'}$ for odd $t\ge3$; the rank-zero groups $C_0$ and $B_0$ are trivial and we do not
name them again.

\section{The torsion filter}\label{sec:filter}

\subsection{The functional}
\emph{Throughout this section $t$ is even}, $t=2m+2$. The odd filter is a different functional on a
different group and is treated in \S\ref{sec:odd}; the two agree nowhere except trivially, and
\eqref{eq:tau} below must not be read at odd $t$ --- see Remark \ref{rem:notodd}.

For a dominant weight $\eta$ of $\Sp_{2m}$ (that is, $\eta_1\ge\dots\ge\eta_m\ge0$) put
\begin{equation}\label{eq:tau}
\tau^C_t(\eta)\;=\;\spc_\eta(\xi,\xi^2,\dots,\xi^{m}),
\qquad
a_j=\eta_j+m-j+1 \quad (1\le j\le m),
\end{equation}
where $\spc_\eta$ is the irreducible symplectic character. Since the eigenvalues of the corresponding
group element are $\xi^{\pm1},\dots,\xi^{\pm m}$, all distinct and different from $\pm1$, the element
is regular of order $t$; its order exceeds the Coxeter number $h=2m$ of $C_m$ by two. That
$\tau^C_t(\eta)\in\{0,\pm1\}$ is \emph{not} quoted here from anywhere: it is proved outright in Lemma
\ref{lem:T} below, by the bialternant. Kostant's theorem \cite{Kostant,Prasad} is the conceptual
ancestor of that species of conclusion and not a statement that covers our element --- see Remark
\ref{rem:kostant}. What we need, and what the following makes explicit, is \emph{which} value.

\begin{lemma}[the filter, with its sign]\label{lem:T}
Let $t=2m+2$ be even and let $a=(a_1,\dots,a_m)$ be as in \eqref{eq:tau}. For each $j$ set
$c_j=a_j \bmod t$. Then
\[
\tau^C_t(\eta)\neq0
\iff
\text{no } c_j\in\{0,\tfrac t2\}
\ \text{ and the classes }\ \min(c_j,t-c_j)\ \text{ are pairwise distinct,}
\]
in which case they are a permutation of $\{1,2,\dots,m\}$ and, writing
$\varepsilon_j=+1$ if $c_j\le m$ and $-1$ otherwise, and $\sigma$ for the permutation
$j\mapsto \min(c_j,t-c_j)$ read against the decreasing order $m,m-1,\dots,1$,
\[
\boxed{\ \tau^C_t(\eta)\;=\;\sgn(\sigma)\prod_{j=1}^{m}\varepsilon_j\ }
\]
\end{lemma}

\begin{proof}
By the type-$C$ bialternant,
$\spc_\eta(x)=\det\bigl(x_i^{a_j}-x_i^{-a_j}\bigr)\big/\det\bigl(x_i^{d_j}-x_i^{-d_j}\bigr)$
with $d_j=m-j+1$. At $x_i=\xi^i$ the $(i,j)$ entry is
$\xi^{ia_j}-\xi^{-ia_j}=2\sqrt{-1}\,\sin(2\pi i a_j/t)$ --- where the $i$ in the exponent and in
the sine is the row index, not the imaginary unit --- so column $j$ depends on $a_j$ only modulo $t$
and is
odd in $a_j$; it vanishes identically when $a_j\equiv0$ or $a_j\equiv t/2$. Since $t=2m+2$, the
classes modulo sign are $0,1,\dots,m,t/2$, so exactly $m$ of them survive those two conditions. If two columns share a class
they are proportional and the determinant vanishes. Otherwise the $m$ classes are pairwise distinct,
hence exhaust $\{1,\dots,m\}$, and the columns are -- up to a permutation and a sign per column --
the columns of the denominator, whose matrix $\bigl(\sin(2\pi ik/t)\bigr)_{i,k=1}^{m}$ is a discrete
sine transform and is invertible. The quotient is therefore the stated signed permutation factor.
\end{proof}

\begin{corollary}[the even filter lives on even weights]\label{cor:evensize}
Let $t=2m+2$ be even. If $\tau^C_t(\eta)\neq0$ then $|\eta|$ is even.
\end{corollary}

\begin{proof}
By Lemma \ref{lem:T} the folded classes of $a=\eta+\rho_{C_m}$ are a permutation $\sigma$ of
$\{1,\dots,m\}$, so $a_j\equiv\pm\sigma(j)\pmod t$. Since $t$ is \emph{even}, the fold is invisible
modulo $2$ and this gives $a_j\equiv\sigma(j)\pmod 2$; summing,
$\sum_j a_j\equiv\sum_j\sigma(j)=\tfrac{m(m+1)}2\pmod2$. As
$\sum_j a_j=|\eta|+\tfrac{m(m+1)}2$, the two shifts cancel and $|\eta|\equiv0$.
\end{proof}

\noindent
The whole proof is ``since $t$ is even'', which is the same fact about $2$ and $t$ that Remark
\ref{rem:BC} turns on --- there $2$ fails to be invertible modulo $t$, here $2$ divides $t$, and the
two are one statement --- and the corollary is the frozen-block counterpart of the selection rule of
Corollary \ref{cor:selection}. Checked over the entire support at $t=4,6,8,10,12$ --- $13$, $36$, $64$, $81$
and $48$ weights, with the step $a_j\equiv\sigma(j)$ verified coordinate by coordinate --- and
against the decoy that has to fail, namely the same statement on the odd branch, where $t$ is odd and
the argument evaporates: there $|\eta|$ is even on $9$ of $17$, $25$ of $49$ and $40$ of $80$ of the
support, which is chance \stverif.

\begin{remark}[the type-$A$ ancestor]
Lemma~\ref{lem:T} is the type-$C$ counterpart of a statement the companion paper already uses:
Littlewood's evaluation \cite{LR34} of a Schur function on a full root-of-unity orbit, which is
$(-1)^{\binom t2}\sgn(\sigma)$ when the $t$-core of $\lambda$ is empty and $0$ otherwise. Both are
``a sign when a residue condition holds, zero when it fails''; what changes with the type is the
residue condition and the sign. Read side by side, the two evaluate the two halves of the alphabet
of \eqref{eq:def}: Littlewood the frozen block seen from type $A$, Lemma~\ref{lem:T} the frozen
block seen from the symplectic side, which is the reading the free pairs force.
\end{remark}

\begin{remark}[what is ours here and what is not]\label{rem:kostant}
Lemma~\ref{lem:T} is proved above from the type-$C$ bialternant and owes nothing to what follows,
which is why we can be precise about its standing. Its \emph{shape} --- on a wall $\Rightarrow$ zero,
off the walls $\Rightarrow$ fold with a sign --- is the standard mechanism of representation theory at
roots of unity, stated for quantum groups in exactly this form by Andersen and Stroppel
\cite{AndersenStroppel}. Its \emph{species} of conclusion, a character value in $\{0,\pm1\}$ at an
element of finite order, is the species of Kostant's theorem \cite{Kostant,Prasad}. But that theorem
is about the Coxeter element, and ours is not: the eigenvalues here are $\xi^{\pm1},\dots,\xi^{\pm m}$
with $\xi$ of order $t=h+2$, where $h=2m$ is the Coxeter number of $C_m$.

The precise home of the statement is the programme of Nadimpalli, Pattanayak and Prasad \cite{NPP},
who study exactly this: the character of an irreducible at a torsion element. Their Theorem 4.1
concerns the \emph{principal} element $C_d=\rho^\vee(e^{2\pi i/d})$ and says that
$\Theta_\lambda(C_d)\neq0$ if and only if the centraliser of $(\lambda+\rho)(e^{2\pi i /d})$ in the
dual group has the smallest possible dimension, the deciding set of roots being
$\Phi_{\lambda,d}=\{\alpha: d\mid\langle\lambda+\rho,\alpha^\vee\rangle\}$; the associated constant
is $1$ in types $A$ and $B$, and the classification of the root systems involved is completed by
Polo \cite{Polo}. Their Question 8.1, which they attribute to Prasad and leave open, asks whether the
same dichotomy governs an \emph{arbitrary} torsion element.

Our two elements answer that question in opposite ways, and the reason is the parity of $t$
\stproved{} --- the two lines below settle both halves, and \S\ref{sec:verif} records the same
statement as proved. For odd $t$ the torsion element is regular \emph{and} conjugate to the principal element of
order $t=h+1$; since $t>h$ forces $\Phi_{0,t}=\varnothing$ --- the set just defined at $\lambda=0$,
which is the principal element --- Theorem 4.1 of \cite{NPP} applies verbatim and
\emph{proves} the odd filter, values in $\{0,\pm1\}$ included --- see \S\ref{sec:odd}. For even $t$
the element is regular but is \emph{not} conjugate to any principal element once $m\ge2$ --- and
that is a two-line argument rather than a measurement. A principal element takes the same value
$q=\xi^k$ on every simple root. Writing the toral coordinates of a Weyl conjugate of our element as a
signed permutation $b_1,\dots,b_m$ of $\{\pm1,\dots,\pm m\}$, the simple roots $e_i-e_{i+1}$ and
$2e_m$ of $C_m$ give
\[
b_i-b_{i+1}\equiv k,\qquad 2b_m\equiv k \pmod t .
\]
Since $t=2m+2$ is even, the second congruence forces $k$ even; the first then forces all the $b_i$ to
have the same parity. But $|b_1|,\dots,|b_m|$ run over $1,\dots,m$, which for $m\ge2$ contains both
an odd and an even number. Contradiction. And there
the centraliser dichotomy is no longer equivalent to vanishing: it misses the wall $2a_i\equiv0$,
which exists as a separate condition only because $2$ is not invertible modulo an even $t$. What that
wall is --- the regularity wall of the long root of $C_m$, degraded to a short root on passing to the
dual --- is Remark \ref{rem:BC}. The smallest instance is $\Sp_4$ at $t=6$ with
$\eta=(1,0)$: both centralisers are maximal tori and the character is nevertheless $0$. So for
even $t$ we prove the case we use because no available theorem covers it; for odd $t$ we cite.
\end{remark}

\subsection{The filter is regularity in the group itself}
Lemma \ref{lem:T} lists three species of wall, and Proposition \ref{prop:oddfilter} two. Both lists
are one statement, and it is the statement that explains the parity.

\begin{lemma}[the filter is a regularity condition]\label{lem:regular}
Define, in the two cases,
\[
x^{C}_{\eta,t}=\bigl(\xi^{\,\eta_i+m-i+1}\bigr)_{i=1}^{m}\in\Sp_{2m},
\qquad
x^{B}_{\eta,t}=\bigl(\xi^{\,2\eta_i+2(m'-i)+1}\bigr)_{i=1}^{m'}\in SO_{2m'+1}.
\]
Then
\[
\tau^C_t(\eta)\neq0\iff x^{C}_{\eta,t}\ \text{is regular semisimple in }\Sp_{2m},
\]
\[
\tau^B_t(\eta)\neq0\iff x^{B}_{\eta,t}\ \text{is regular semisimple in }SO_{2m'+1}.
\]
\end{lemma}

\begin{proof}
Both exponent vectors are $\eta+\rho$ read in the coordinates the group provides: in type $C$ that is
literal, $\rho_{C_m}=(m,\dots,1)$ being integral; in type $B$ the half-sum
$\rho_{B_{m'}}=(m'-\tfrac12,\dots,\tfrac12)$ is not, and one must use $2(\eta+\rho)$, which is legitimate
because $t$ is odd there and $2$ is invertible modulo $t$. This is not a convention: reading the odd
case with $\eta+\rho$ instead of $2(\eta+\rho)$ misses zeros, as Remark \ref{rem:notodd} shows.

Now evaluate the roots. In type $C_m$ they are $e_i-e_j$, $e_i+e_j$ and $2e_i$, taking the values
$\xi^{a_i-a_j}$, $\xi^{a_i+a_j}$ and $\xi^{2a_i}$; in type $B_{m'}$ they are $e_i\pm e_j$ and $e_i$,
taking $\xi^{A_i\mp A_j}$ and $\xi^{A_i}$ with $A_i=2(\eta_i+\rho_i)$. So the element fails to be
regular exactly when $a_i\equiv\pm a_j$ or $2a_i\equiv0$ in type $C$, and when $A_i\equiv\pm A_j$ or
$A_i\equiv0$ in type $B$, modulo $t$ --- which is verbatim the list of walls of Lemma \ref{lem:T} and
of Proposition \ref{prop:oddfilter} respectively. The forward implication needs no more than that:
$a_i\equiv\pm a_j$ makes two columns of the Weyl numerator equal or opposite, and the vanishing of
the long-root value makes a column identically zero. That direction is standard Weyl-numerator
theory, and is recorded for regular torsion elements in \cite[Rem.~8.2]{NPP}. The converse --- that
regularity suffices --- is the content of Lemma \ref{lem:T} in the even case and of \cite{NPP} in the
odd one, and it is the half that is not formal.
\end{proof}

\noindent
Verified independently of the whole apparatus above --- exact cyclotomic arithmetic, no character
library --- on $24\,809$ weights for $3\le t\le12$, with no exception in either parity \stverif; a
decoy using only the roots of $A_{\mathrm{rk}-1}$ agrees on $9\,562$ of them.

\begin{corollary}[the filter is a function on a finite set]\label{cor:periodic}
$\tau_t(\eta)$ depends only on the residues of the shifted vector modulo $t$. In particular it is
unchanged by $\eta_j\mapsto\eta_j+t$ whenever that leaves a partition.
\end{corollary}

\begin{proof}
Whether the element is regular is a condition on the residues, by Lemma \ref{lem:regular}, and that
half is uniform in the parity. For the value: when $t$ is even the sign given by Lemma \ref{lem:T} is
written entirely in terms of the $c_j$; when $t$ is odd the same holds in the coordinates
$A_i=2(\eta_i+\rho_i)$ of Proposition \ref{prop:oddfilter}, because the Weyl numerator there is
$\det\bigl(\omega^{iA_j}-\omega^{-iA_j}\bigr)$ with $\omega$ a primitive $t$-th root, and each entry
depends on $A_j$ only modulo $t$.
\end{proof}

\noindent
Checked on $3\le t\le6$: the filter is constant on each residue class --- $3$, $4$, $25$ and $36$
classes with more than one representative, all constant --- while the same test at modulus $t-1$
fails on all but one class \stverif. Small as it is, this is what makes the step size of
\S\ref{sec:where} not a coincidence: on the orbit side a move between surviving shapes is a
$t$-ribbon, and on the branching side survival is periodic with the same period $t$.

Once the filter is a function on $(\mathbb Z/t)^n$ one can ask what symmetries that finite set has,
and the answer identifies the two types with each other.

Two loci have to be kept apart, because they are the two species of wall of Lemma \ref{lem:T} and
they differ exactly where the paper's dichotomy lives. Write
\begin{equation}\label{eq:twoloci}
\begin{aligned}
\mathcal R^B_{t,n}&=\bigl\{x\in(\ZZ/t)^n:\ x_i\not\equiv0,\ x_i\not\equiv\pm x_j\ (i\neq j)\bigr\},\\
\mathcal R^C_{t,n}&=\bigl\{x\in(\ZZ/t)^n:\ 2x_i\not\equiv0,\ x_i\not\equiv\pm x_j\ (i\neq j)\bigr\}
\end{aligned}
\end{equation}
for the two regular loci of Lemma \ref{lem:regular} in exponent coordinates. The type-$C$ locus
carries the extra condition $x_i\neq t/2$, which is the long-root regularity wall of Remark
\ref{rem:BC} and is
one of the theses of this paper; it is empty of content exactly when $2$ is invertible.

We also need a name for the type-$C$ rule read at an odd modulus. That is \emph{not} the odd filter
of this paper --- Remark \ref{rem:notodd} --- so it gets its own symbol: for $t=2n+1$ put
\begin{equation}\label{eq:auxC}
\widetilde\tau^{\,C}_{t,n}(\eta)\ :=\ \spc_\eta(\zeta,\dots,\zeta^{\,n}),
\end{equation}
the auxiliary type-$C$ reading, used only for comparison. Lemma \ref{lem:regular} does not cover it
--- its type-$C$ half is the even filter --- so it needs its own line, which the same argument
supplies.

\begin{lemma}[the auxiliary reading is regularity in $\mathcal R^C$]\label{lem:auxC}
For $t=2n+1$ odd, $\widetilde\tau^{\,C}_{t,n}(\eta)\neq0$ if and only if
$\eta+\rho_{C_n}\in\mathcal R^C_{t,n}$.
\end{lemma}

\begin{proof}
The symplectic character is the bialternant in the exponents $a=\eta+\rho_{C_n}$ evaluated at
$\xi^{a_j}$ with $\xi$ a primitive $t$-th root, so each column depends on $a_j$ only modulo $t$.
Two columns coincide as soon as $a_i\equiv\pm a_j$, and a column vanishes as soon as $2a_i\equiv0$;
these are exactly the conditions cut out by $\mathcal R^C_{t,n}$. Conversely, if $a$ is regular then
its $n$ folded classes are distinct and nonzero, and since $t=2n+1$ there are exactly $n$ nonzero
classes modulo sign: a regular vector exhausts them, so the matrix is the base one up to a signed
permutation of columns and the ratio is $\pm1$. Note that at odd $t$ the condition $2a_i\equiv0$ is
equivalent to $a_i\equiv0$, so $\mathcal R^C_{t,n}=\mathcal R^B_{t,n}$ here and the statement is
consistent with Proposition \ref{prop:galois}(ii).
\end{proof}

\begin{proposition}[the two loci are stable under the units, and $2$ is what compares
them]\label{prop:galois}
Fix a rank $n$.
\begin{enumerate}
\item[\rm(i)] $\mathcal R^B_{t,n}$ and $\mathcal R^C_{t,n}$ are each stable under $x\mapsto kx$ for
every $k\in(\ZZ/t)^\times$.
\item[\rm(ii)] If $t$ is odd then $\mathcal R^B_{t,n}=\mathcal R^C_{t,n}$, because $2$ is
invertible; and for every $\eta$ of rank $n$
\begin{equation}\label{eq:BCtrans}
\tau^B_t(\eta)\neq0
\quad\Longleftrightarrow\quad
\widetilde\tau^{\,C}_{t,n}\Bigl(\eta+\tfrac{t-1}{2}(1,\dots,1)\Bigr)\neq0 ,
\end{equation}
so that the two \emph{nonvanishing loci} are translates of one another by $\tfrac{t-1}{2}$ in every
coordinate.
\item[\rm(iii)] For even $t$ the argument of (ii) is unavailable, and the conclusion can already fail
at rank $2$: at $(t,n)=(6,2)$ the two loci have $16$ and $8$ elements, so no bijection of any kind
relates them, let alone a translation. It fails in every even case we tested.
\end{enumerate}
\end{proposition}

\begin{proof}
(i) For $k$ invertible, $x\mapsto kx$ is an additive bijection of $\ZZ/t$ commuting with
negation. It therefore preserves each of the conditions $x_i=0$, $2x_i=0$, $x_i=x_j$, $x_i=-x_j$ in
both directions --- for the second because $2(kx_i)=k(2x_i)$ and $k$ is invertible --- and hence
preserves each of $\mathcal R^B_{t,n}$ and $\mathcal R^C_{t,n}$.

(ii) The two exponent vectors of Lemma \ref{lem:regular} are $A=2(\eta+\rho_{B_n})$ and
$a=\eta+\rho_{C_n}$, and the two $\rho$'s differ by the constant vector $\tfrac12$:
$\rho_{B_n}=\rho_{C_n}-\tfrac12(1,\dots,1)$. Since $t$ is odd, $2$ is invertible; write
$s=2^{-1}=\tfrac{t+1}{2}$. Then
\[
sA=\eta+\rho_{B_n}=\eta+\rho_{C_n}-\tfrac12(1,\dots,1)
 =\Bigl(\eta+\tfrac{t-1}{2}(1,\dots,1)\Bigr)+\rho_{C_n}
\qquad\text{in }\mathbb Z/t,
\]
because $-\tfrac12\equiv-s\equiv\tfrac{t-1}{2}$. So $sA$ \emph{is} the type-$C$ exponent vector of
the translated weight, and by (i) with $k=s$ we have $A\in\mathcal R^B_{t,n}$ if and only if
$sA\in\mathcal R^B_{t,n}=\mathcal R^C_{t,n}$. Lemma \ref{lem:regular} turns the left-hand side into
$\tau^B_t(\eta)\neq0$ and Lemma \ref{lem:auxC} turns the right-hand side into
$\widetilde\tau^{\,C}_{t,n}\bigl(\eta+\tfrac{t-1}{2}\mathbf 1\bigr)\neq0$.

(iii) $2$ is not invertible modulo an even $t$, so the computation in (ii) is unavailable; that alone
would prove nothing, since some other element might serve. The witness at rank $2$ rules it out
outright: at $(t,n)=(6,2)$ the two loci have $16$ and $8$ elements, and there is no bijection between
sets of different size. Beyond that witness we report measurement, not proof. Writing
$|\mathcal R^B_{t,n}|/|\mathcal R^C_{t,n}|$, the even counts are $36/24$ at $(8,2)$, $64/48$ at
$(10,2)$, $100/80$ at $(12,2)$, $144/120$ at $(14,2)$, and at rank $3$, $120/48$, $336/192$,
$720/480$, $1320/960$ at $(8,3)$, $(10,3)$, $(12,3)$, $(14,3)$ \stverif. In every one the two
cardinalities differ, so no bijection exists there either. We do not claim this for all even $t$ and
all $n\ge2$; a general formula for the two cardinalities would settle it, and we do not have one.

The counts are in exponent coordinates, which is where \eqref{eq:twoloci} lives. In weight
coordinates the type-$B$ side is pulled back along $\eta\mapsto2(\eta+\rho)$, which for even $t$ is
not a bijection, so the numbers there are different --- at $(6,2)$ they are $0$ and $8$. Both
comparisons give the same verdict; we quote the first because it is the one the definition names.
\end{proof}

\begin{remark}[the converse is false, and it is false one rank below]\label{rem:rank1}
This remark is about the two rules read in \emph{weight} coordinates, and the distinction matters, so
we name the sets. Write
\begin{equation}\label{eq:pullbacks}
\mathcal P^B_{t,n}=\{\eta:\ A_B(\eta)\in\mathcal R^B_{t,n}\},
\qquad
\mathcal P^C_{t,n}=\{\eta:\ a_C(\eta)\in\mathcal R^C_{t,n}\},
\end{equation}
the pullbacks of \eqref{eq:twoloci} along $A_B(\eta)=2(\eta+\rho_{B_n})$ and
$a_C(\eta)=\eta+\rho_{C_n}$. For odd $t$ the second map is a bijection of $(\ZZ/t)^n$ and so is the
first, so \eqref{eq:BCtrans} passes between the two pictures unchanged; for even $t$ the first is
not, and the two pictures say different things. It is $\mathcal P^B$ and $\mathcal P^C$ that
Figure \ref{fig:galois} draws.

Part (ii) is a mechanism, not a characterisation, and rank $1$ shows the difference \emph{in the
weight picture}. At $(t,n)=(6,1)$, $(10,1)$ and $(14,1)$ --- that is, at even $t\equiv2\pmod4$ ---
the sets $\mathcal P^B_{t,1}$ and $\mathcal P^C_{t,1}$ \emph{are} translates of one another, by
$\tfrac{t+2}{4}$ and by $\tfrac{t+2}{4}+\tfrac t2$, even though $2$ is not invertible there
\stverif. The reason is visible once stated: for those $t$ the half-modulus $t/2$ is odd, and $2$ is
invertible modulo it. So a translation can exist without the Galois element that produces it in
(ii), and only the rank-$2$ witness of (iii) shows that the general statement fails. Nothing of the
kind happens for the exponent loci themselves, where already
$|\mathcal R^B_{t,1}|=t-1$ and $|\mathcal R^C_{t,1}|=t-2$ for every even $t$. The rank-$1$
phenomenon is a fact about the pullbacks and about nothing else.
\end{remark}

\noindent
Two measurements, both designed to be able to fail. First, the translation of \eqref{eq:BCtrans} was
checked against the two rules directly, over the whole of $(\ZZ/t)^n$ for odd $3\le t\le11$ and
$n\le3$: the two \emph{supports} agree in all twelve non-degenerate cases, and \emph{no other}
translation achieves that \stverif{} --- had some other shift worked as well, the statement would
carry no information. It is a statement about supports and not about values; the witness at the end
of Remark \ref{rem:galois} shows the values differ. Second, part (i) was checked against its own
converse: over $3\le t\le12$ and $n\le3$, all $118$ units preserve the locus and \emph{none} of the
$58$ non-units does \stverif. So invertibility is not a convenience of the proof, it is the property
being used.

\begin{remark}[what this is, in one word]\label{rem:galois}
It is Galois, and it is the cyclotomic action in coordinates.
$\operatorname{Gal}(\mathbb Q(\zeta_t)/\mathbb Q)\cong(\ZZ/t)^\times$ acts by $\zeta_t\mapsto\zeta_t^k$, and on
residues that action is multiplication by a unit, which preserves $0$, equality and opposition.
That is the whole of part (i); \emph{no conjugacy theorem is needed for Proposition
\ref{prop:galois}}. Two theorems one might reach for do not apply and should not be cited here:
Springer's conjugacy of $\zeta$-regular elements is inside the reflection group, and the uniqueness
statement of \cite[Thm.~12.1]{NPP} is for orders dividing the Coxeter number, whereas our orders are
$t=h+1$ and $t=h+2$.

What the proposition does buy is specific to the two rules in play. The passage from one type to the
other is the single Galois element $\sigma_2$, and \eqref{eq:BCtrans} says the two
\emph{nonvanishing loci} are related by nothing else. It is the loci and not the values: the two
rules do not agree pointwise even where both survive. At $t=5$, $n=2$, $\eta=(0,0)$ one has
$\tau^B_5(0,0)=1$, while the translate $(2,2)$ prescribed by \eqref{eq:BCtrans} gives
$\widetilde\tau^{\,C}_{5,2}(2,2)=-1$ \stverif. The element $\sigma_2$ exists exactly when $t$ is
odd, which is why the two types see the same walls there and different ones otherwise --- the sixth
row of Proposition \ref{prop:parity} restated as a group-theoretic fact rather than an inspection of
walls.

It also makes Remark \ref{rem:notodd} quantitative. Reading the odd case with $\eta+\rho$ instead of
$2(\eta+\rho)$ does not scramble the answer; it decides survival correctly, at the wrong point,
displaced by $\tfrac{t-1}{2}$ in every coordinate --- and it can still return the wrong sign there,
by Proposition \ref{prop:galoissign}. That is why the wrong rule can reproduce the right number of
survivors --- $49$ against $49$ in Figure \ref{fig:filter} --- while disagreeing on which they are:
a translate of a set has the same size.
\end{remark}

The locus is invariant under the units; the \emph{value} is not, and that has to be said with a
witness rather than glossed. At $t=4$, $m=1$, the shifted residue $a=1$ gives $\tau=+1$ while
$a=3$ gives $\tau=-1$, and $3$ is a unit \stverif. What survives the action is a sign character.

\begin{proposition}[signed Galois covariance of the filter]\label{prop:galoissign}
Fix $t$ and let $n$ be the number of folded residue classes carried by the filter --- $n=(t-1)/2$ in
type $B$, $n=t/2-1$ in type $C$, the class $t/2$ being fixed by every unit. For
$k\in(\ZZ/t)^\times$, multiplication by $k$ followed by folding is a signed permutation $w_k$ of
$\{1,\dots,n\}$; put
\[
\gamma_t(k):=\det w_k=\sgn(\sigma_k)\prod_{j=1}^{n}\varepsilon_j(k)\ \in\{\pm1\}.
\]
Then $\gamma_t:(\ZZ/t)^\times\to\{\pm1\}$ is a character, and for every residue vector $a$
\begin{equation}\label{eq:galoissign}
\tau_t(k\cdot a)=\gamma_t(k)\,\tau_t(a).
\end{equation}
Moreover $\prod_j\varepsilon_j(k)$ is the Jacobi symbol $\left(\tfrac kt\right)$ for odd $t$, so
$\gamma_t(k)=\sgn(\sigma_k)\left(\tfrac kt\right)$.
\end{proposition}

\begin{proof}
Write $\delta(a):=\det w_a$ for the determinant of the signed permutation that the folded classes of
$a$ define, and $\delta(a):=0$ when those classes are not a permutation of $\{1,\dots,n\}$. Both
$\tau_t$ and $\delta$ are alternating in $a$, both depend only on $a$ modulo $t$, and both vanish
exactly off the regular locus. Those three facts alone would not give proportionality; what does is
that we are at minimal level, where \emph{the number of coordinates equals the number of available
folded classes}. Hence a regular $a$ is a signed permutation of the full list
$(1,2,\dots,n)$, and there is exactly one such list: the regular locus is a \emph{single free orbit}
of the signed permutation group, and on the numerator side the columns of a regular $a$ are
literally a signed permutation of the columns of the base point. An alternating function on a single
free orbit is determined by its value at one point, so $\tau_t=\epsilon_{t}\,\delta$ for a single
sign $\epsilon_t$ depending on $t$ and the type alone.

Replacing $a$ by $ka$ composes the signed permutation attached to $a$ with $w_k$, so
$\delta(ka)=\det(w_k)\,\delta(a)$ by multiplicativity of $\det$; the vanishing case is stable by
Proposition \ref{prop:galois}(i). Multiplying by $\epsilon_t$ gives \eqref{eq:galoissign}, and
multiplicativity of $\det$ gives $\gamma_t(k_1k_2)=\gamma_t(k_1)\gamma_t(k_2)$.

The two factors of $\gamma_t$ are classical, and we claim neither. That
$\prod_j\varepsilon_j(k)=(-1)^{\mu}$ with $\mu=\#\{j\le n:\ kj\bmod t>t/2\}$ equals
$\left(\tfrac kt\right)$ is Gauss's lemma in Jacobi's form. That $\sgn(\sigma_k)=
\left(\tfrac kt\right)^{(t+1)/2}$ --- equivalently $\left(\tfrac kt\right)$ for $t\equiv1$ and $1$
for $t\equiv3$ modulo $4$ --- is a theorem of Pan \cite{Pan06}; the observation that
$\sigma_{k_1k_2}=\sigma_{k_1}\sigma_{k_2}$ and $\sigma_{-k}=\sigma_k$, so that $\sgn(\sigma_\cdot)$
is a real even character, is Pan's as well.
\end{proof}

\noindent
The product of the two can be said in one line, and saying it names a field.

\begin{corollary}[$\gamma_t$ is trivial or quadratic, and of which field]\label{cor:galoisquad}
Let $n$ be the number of folded classes the filter carries --- $n=(t-1)/2$ in type $B$, $t$ odd, and
$n=t/2-1$ in type $C$, $t$ even --- and put
\[
D_t=\det\bigl(\zeta^{ij}-\zeta^{-ij}\bigr)_{1\le i,j\le n}.
\]
Then, in \emph{both} branches, $\sigma_k(D_t)=\gamma_t(k)\,D_t$ and $D_t^{\,2}=(-t)^n$.
Consequently
\begin{equation}\label{eq:galoisquad}
\gamma_t=
\begin{cases}
1 & n\ \text{even},\\[2pt]
\chi_{\mathbb Q(\sqrt{-t})} & n\ \text{odd},
\end{cases}
\end{equation}
and for odd $t$ this reads $\gamma_t(k)=\left(\tfrac kt\right)^{n}$. Translating the parity of $n$
into a congruence on $t$ gives one statement covering both parities:
\[
\gamma_t\ \text{is trivial for}\ t\equiv1,2\pmod4,
\qquad
\gamma_t\ \text{is quadratic for}\ t\equiv0,3\pmod4 .
\]
\end{corollary}

\begin{proof}
Applying $\sigma_k$ to $D_t$ replaces the row indexed by $i$ by the row indexed by $ki$; folding
$ki$ into $\{1,\dots,n\}$ costs the sign $\varepsilon$ because the entries are odd under
$\zeta\mapsto\zeta^{-1}$, so the rows are permuted by $\sigma_k$ and signed by the
$\varepsilon_j(k)$. The determinant therefore picks up exactly $\det w_k=\gamma_t(k)$. For the
square, $D_t=(2\sqrt{-1})^n\det\bigl(\sin(2\pi ij/t)\bigr)$, the $i$ inside the sine being the row
index; and the orthogonality of the discrete sine
transform gives $\bigl(\sin(2\pi ij/t)\bigr)^2=\tfrac t4 I$, whence
$D_t^2=(-4)^n(t/4)^n=(-t)^n$. So $D_t=\pm\sqrt{(-t)^n}$ generates $\mathbb Q(\sqrt{-t})$ when $n$ is
odd and lies in $\mathbb Q$ when $n$ is even, and $\gamma_t$ is the character by which the Galois
group acts on it. Comparison with the proposition gives \eqref{eq:galoisquad}, consistently:
$(t+3)/2=n+2$ and the symbol squares to $1$.
\end{proof}

\noindent
Verified independently of the proof, in both branches: $D_t^2=(-t)^n$ exactly in $21$ of $21$ moduli
$3\le t\le23$, $\sigma_k(D_t)=\gamma_t(k)D_t$ on all $170$ pairs (unit, modulus) there, and
$\gamma_t$ trivial exactly when $n$ is even in $21$ of $21$; the resulting split by $t$ modulo $4$
is uniform, trivial at $1,2$ and quadratic at $0,3$ \stverif.

\begin{remark}[whose the mechanism is]\label{rem:galoisattrib}
Signed Galois covariance of Weyl-type data is standard, and we do not claim it. In the modular-data
setting Fuchs, Schellekens and Schweigert attach to a Galois transformation the sign of the affine
Weyl element that returns the scaled weight to the alcove, and write
$S_{\widehat w(a),b}=\sgn(\widehat w)S_{a,b}$ \cite{FSSGalois}; Proposition
\ref{prop:galoissign} is that phenomenon for our two filters. What is specific here is that at
minimal level the sign collapses to a \emph{weight-independent} character --- there is one
$\gamma_t$, not one sign per weight --- and that Corollary \ref{cor:galoisquad} identifies it, in
both branches, as trivial or as the quadratic character of $\mathbb Q(\sqrt{-t})$ according to the
parity of $n$.

The arithmetic half of that corollary is older still, and saying so places it. Reading a quadratic
character off a matrix of roots of unity whose square is a scalar --- so that $\sqrt{\pm t}$ appears
and the Galois group acts on it by a sign --- is Schur's linear-algebra derivation of the sign of the
quadratic Gauss sum \cite{Schur21,Murty}, and thence of quadratic reciprocity. Our $D_t^2=(-t)^n$ is the same first step:
the sine matrix squares to $\tfrac t4 I$ exactly as Schur's $(\zeta^{rs})$ squares to $n$ times a
permutation.

We also checked whether the rest of Schur's argument buys us anything, and it does not, which is
worth recording so that it is not tried twice. What his multiplicity count delivers is the sign of
the \emph{trace} --- the Gauss sum itself; transported here it would give the sign of $D_t$
individually, from the multiplicities of $\pm\sqrt t/2$ as eigenvalues of the sine matrix. We never
need that sign. Everything in this paper uses $D_t$ through a \emph{ratio}: $\gamma_t(k)$ is
$\sigma_k(D_t)/D_t$, and $\epsilon_t$ is, after Corollary \ref{cor:oddsign}, essentially
$D_t(\zeta^2)/D_t(\zeta)$ up to $(-1)^{m'}$. The ratio is exactly what the Galois action hands over
for free, and computing the two signs separately in order to divide them would be strictly more work
for the same answer --- and would still need the Gauss sum sign as an input. So the antecedent is
real and the technique is not a tool for us.
\end{remark}

\noindent
The constant $\epsilon_t$ is then not a loose end either, and identifying it upgrades Proposition
\ref{prop:oddfilter} from an absolute value to a sign.

\begin{corollary}[the odd filter in closed form, sign included]\label{cor:oddsign}
Let $t=2m'+1$ and $A=2(\eta+\rho_{B_{m'}})$, and let $\delta(A)$ be the determinant of the signed
permutation of $\{1,\dots,m'\}$ given by the folded classes of $A$ modulo $t$, read --- as in Lemma
\ref{lem:T} --- \emph{against the decreasing order} $m',m'-1,\dots,1$, and $0$ when the classes are
not a permutation. Then
\begin{equation}\label{eq:oddsign}
\tau^B_t(\eta)=\epsilon_t\,\delta(A),
\qquad
\epsilon_t=\Bigl(\tfrac{-2}{t}\Bigr)^{(t+3)/2}(-1)^{m'(m'-1)/2},
\end{equation}
and $\epsilon_t=+1$ unless $t\equiv5\pmod8$, where it is $-1$.
\end{corollary}

\begin{proof}
By the proposition, $\tau^B_t=\epsilon_t\delta$ with $\epsilon_t=\delta(A_\rho)$, the value at
$\eta=0$, where $\tau^B_t=1$ because the trivial representation has character $1$ at every point.
Now $A_\rho=(t-2,t-4,\dots,1)$, whose $j$-th entry is $t-2j\equiv-2j$, so
$A_\rho\equiv(-2)\cdot(1,2,\dots,m')$. Read against the \emph{increasing} order this gives exactly
$\gamma_t(-2)=\left(\tfrac{-2}{t}\right)^{(t+3)/2}$ by Corollary \ref{cor:galoisquad}; reading it
against the decreasing order of Lemma \ref{lem:T} composes with the reversal of $m'$ letters and so
multiplies by $(-1)^{m'(m'-1)/2}$. Evaluating both factors modulo $8$ leaves $-1$ exactly at
$t\equiv5$.
\end{proof}

\begin{remark}[the reading order is not cosmetic]\label{rem:convention}
Taking the classes against the increasing order instead multiplies $\delta$ by
$(-1)^{m'(m'-1)/2}$, and therefore moves the exceptional residue of $\epsilon_t$ from $t\equiv5$ to
$t\equiv7$ modulo $8$: with the wrong order in hand one has a statement about a convention wearing
the clothes of a statement about $t$. The character $\gamma_t$ is unaffected, being the determinant
of a permutation of $\{1,\dots,m'\}$ onto itself, whose sign does not depend on how the set is
listed.
\end{remark}

\noindent
Checked against the character evaluation, not against the formula: the closed form
\eqref{eq:oddsign} agrees on every weight tested at $t=3,5,7,9$, the constant coming out
$+1,-1,+1,+1$; the closed value of $\epsilon_t$ was fitted on odd $t\le31$ and then \emph{predicted}
correctly on the held-out range $33\le t\le61$; and the two classical inputs hold on all $788$ pairs
(unit, modulus) over odd $t\le61$ \stverif. The constant is easy to miss because $t=5$ is the
smallest modulus at which it is not $+1$.

\noindent
Both halves were measured against a filter computed from the character rather than from the closed
formula, which is the only way the control can fail: the closed formula agrees with the character
evaluation in every case tested, $\gamma_t$ is a character in every case, and \eqref{eq:galoissign}
holds on every pair (unit, weight), including on the subset where $\tau\neq0$, which is the only
place the statement has content \stverif.

\begin{remark}[what it buys, computationally]\label{rem:galoiscost}
One evaluation per orbit, plus the character $\gamma_t$. Since each locus is a union of
$(\ZZ/t)^\times$-orbits, survival has to be decided at one representative of each; and by
Proposition \ref{prop:galoissign} the \emph{value} on the rest of the orbit is then the
representative's value times $\gamma_t$, so nothing is lost. The saving is up to $\varphi(t)$ and
the bound is attained as soon as the action is free: at $t=5,7,11$ and rank $\le3$ the measured
saving is exactly $\varphi(t)=4,6,10$, and at composite $t$ it is smaller at small rank and climbs
to $\varphi(t)$ --- at $t=12$, the factors are $2.50$, $3.64$, $4.00$ at ranks $1,2,3$ \stverif.
The mechanism is a routine use of a classical symmetry; what is specific to us is the explicit sign
$\gamma_t$ for these two filters. We record it because it is the reason the tables in
\S\ref{sec:verif} could be computed over the range they cover.
\end{remark}

\subsection{Verification}
Lemma~\ref{lem:T} was checked against two independent computations of $\tau^C_t$: one from the weight
multiplicities of the irreducible symplectic character (Freudenthal, hence not through the
bialternant) and one from the bialternant itself \stverif.

\begin{center}
\rowcolors{2}{hband}{white}
\begin{tabular}{c c r c c c c c}
\toprule
$t$ & $m$ & weights & routes agree & $(T)$ failures & $|\tau|=1$ & sign exact & decoys err \\
\midrule
\multicolumn{8}{l}{\emph{even} $t$, $m=(t-2)/2$}\\
4 & 1 & 15 & yes & $0+0$ & yes & $8/8$ & $0$ / $4$ \\
6 & 2 & 66 & yes & $0+0$ & yes & $16/16$ & $8$ / $16$ \\
8 & 3 & 120 & yes & $0+0$ & yes & $18/18$ & $26$ / $21$ \\
10 & 4 & 126 & yes & $0+0$ & yes & $16/16$ & $54$ / $32$ \\
12 & 5 & 126 & yes & $0+0$ & yes & $8/8$ & $48$ / $28$ \\
\midrule
\multicolumn{8}{l}{\emph{odd} $t$, $m'=(t-1)/2$}\\
3 & 1 & 15 & yes & $0+0$ & yes & $10/10$ & $0$ / $5$ \\
5 & 2 & 66 & yes & $0+0$ & yes & $25/25$ & $12$ / $10$ \\
7 & 3 & 120 & yes & $0+0$ & yes & $27/27$ & $36$ / $9$ \\
9 & 4 & 126 & yes & $0+0$ & yes & $16/16$ & $54$ / $8$ \\
\bottomrule
\end{tabular}
\end{center}

\noindent
``$(T)$ failures'' counts the two directions separately (a weight predicted to vanish that does not,
and one predicted not to vanish that does); both are zero throughout, on $780$ weights. The two
decoys are: dropping the folding, i.e.\ requiring only that the $c_j$ be distinct and nonzero; and,
for even $t$, moving the wall from $\{0,t/2\}$ to $\{0\}$, while for odd $t$ \emph{adding} a wall at
$(t+1)/2$, which is not one. Both err on a substantial fraction from rank $2$ on, so the condition of
Lemma~\ref{lem:T} is not equivalent to a weaker one --- with one exception that the table shows and
that we name rather than average away: at rank $1$, that is at $t=3$ and $t=4$, the first decoy errs
$0$ times, because with a single $c_j$ there is nothing for ``distinct'' to say. Those two rows are
untested by it, not confirmed by it. The sign predicted at random agrees on
$27$ of $66$ (even) as chance requires.

\begin{figure}[t]
\centering
\includegraphics[width=\textwidth]{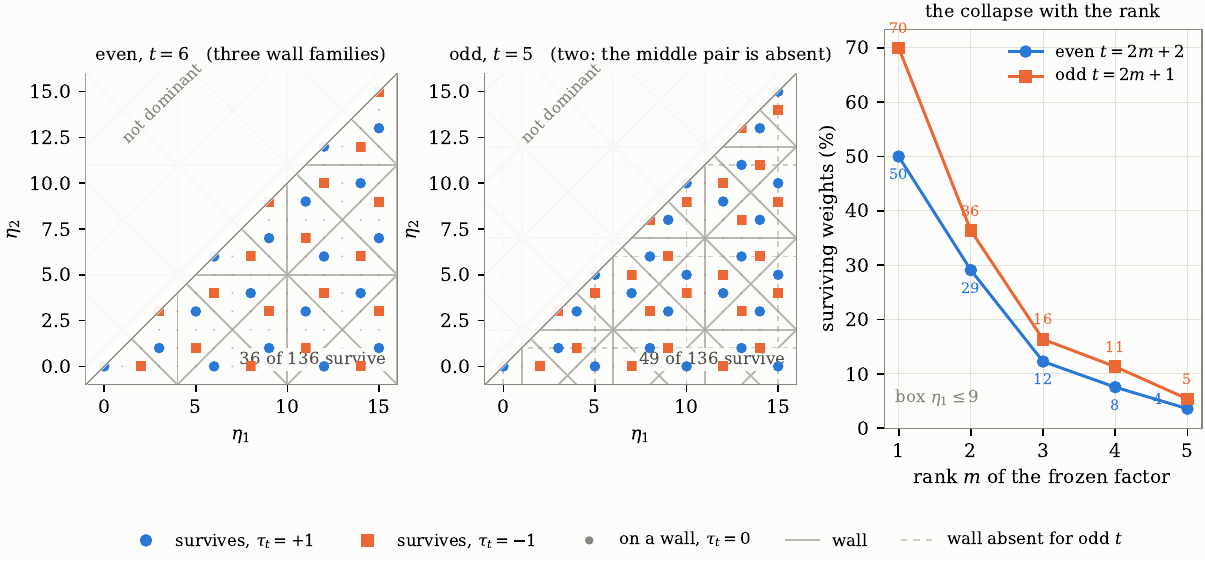}
\caption{The two filters drawn rather than described, and they are \emph{not} the same rule.
\emph{Left:} Lemma \ref{lem:T}, even order. Its walls are lines in the $\eta$-plane: two families
from $a_i\equiv0$, two from $a_i\equiv t/2$ and two from $a_1\equiv\pm a_2$, and a weight dies when it
meets a line. \emph{Centre:} Proposition \ref{prop:oddfilter}, odd order of the same rank, in the
coordinates $A_j=2(\eta_j+\rho_j)$ that type $B$ requires; the middle pair of families does not exist
there and is drawn as ghosts, so that the parity difference is seen and not asserted. Drawing this
panel with the even rule gives the \emph{same} count of
survivors, $49$ of $136$, while disagreeing on $66$ of the $136$ points; the total is exactly the
statistic that fails to notice, which is why the panel is computed from the type-$B$ bialternant
directly rather than from any closed form. \emph{Right:} the survival rate against the rank of the
frozen factor, in a fixed box, from the regularity criterion of Lemma \ref{lem:regular}, which is
valid in both parities. The script checks on every run that the exact odd values stay in
$\{0,\pm1\}$ and that their support agrees with that criterion.}
\label{fig:filter}
\end{figure}

\begin{figure}[htbp]
\centering
\includegraphics[width=\linewidth]{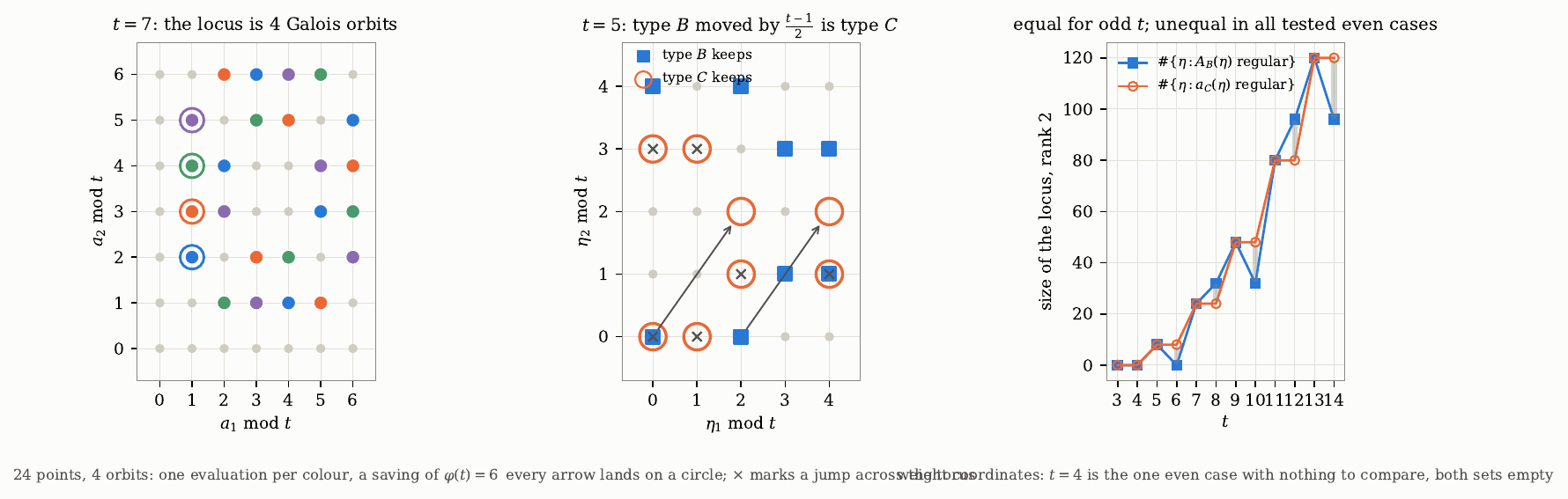}
\caption{The filter seen through its Galois symmetry. \textbf{Left}: at $t=7$ and rank two the
regular locus in exponent coordinates is a union of four $(\mathbb Z/t)^\times$-orbits, one colour
each, never a partial orbit; the ringed point of each colour is the only one that has to be
evaluated. \textbf{Centre}: the two rules the paper uses, at $t=5$ and rank two, in weight
coordinates. Squares are what the type-$B$ rule keeps and circles what the type-$C$ rule keeps ---
different sets, which is the disagreement of Figure \ref{fig:filter} --- and every arrow of the
translation by $\tfrac{t-1}{2}$ lands on a circle. The sets agree; the \emph{values} on them need
not, by Proposition \ref{prop:galoissign}. \textbf{Right}: why this is about parity. The centre and
right panels are drawn in \emph{weight} coordinates, that is, they show the pullbacks
$\mathcal P^B_{t,n}$ and $\mathcal P^C_{t,n}$ of \eqref{eq:pullbacks} and not the exponent loci
\eqref{eq:twoloci}; for odd $t$ the two pictures agree, for even $t$ they do not, and the counts
here are the weight-coordinate ones. Those have equal size for all odd $t$ by Proposition
\ref{prop:galois}(ii) and unequal size in every even case drawn, so no translation can exist in
those. Computed by \texttt{fig\_galois.py}, which refuses to draw if a non-unit preserves the set or
if any translation other than $\tfrac{t-1}{2}$ works.}
\label{fig:galois}
\end{figure}

\begin{remark}[the odd filter is not this one]\label{rem:notodd}
An earlier version of this paper read \eqref{eq:tau} at odd $t$ as well, with $m'$ in place of $m$.
That is wrong, and one weight shows it. At $t=5$, $m'=2$, $\eta=(1,0)$ the symplectic reading gives
$\spc_{(1,0)}(\xi,\xi^2)=\xi+\xi^{-1}+\xi^{2}+\xi^{-2}=-1$, because the five fifth roots of unity sum
to zero; but the odd object of \S\ref{sec:odd} evaluates the \emph{odd orthogonal} character, whose
eigenvalue set includes the fixed point $1$, giving
$\mathrm{o}_{(1,0)}(1,\xi^{\pm1},\xi^{\pm2})=1+(-1)=0$. The two differ by exactly that fixed point,
which the reduction assigns to the torsion block when $t$ is odd and not when it is even. Proposition
\ref{prop:oddfilter} is the correct odd statement, and it has two species of wall where
Lemma \ref{lem:T} has three --- Figure \ref{fig:filter} draws the three and shows the absent family
in ghost. The reason for the missing family is not that $t/2$ fails to be an integer but the one
given in Remark \ref{rem:BC}: the wall is $2a\equiv0$, and it is distinct from $a\equiv0$ only when
$2$ is not invertible modulo $t$.
\end{remark}

\section{The composition}\label{sec:comp}

By the analysis of the companion paper \cite{PaperI} the object with all pairs free is a virtual character of the symplectic group
of the corresponding rank, so we may write, with $R=r+m$,
\begin{equation}\label{eq:branch}
\Phi_{2,R}\;=\;\sum_{\eta,\mu} B_{\eta,\mu}\,\spc_\eta(y)\,\spc_\mu(z),
\qquad B_{\eta,\mu}\in\ZZ,
\end{equation}
the branching of $\Sp_{2R}$ to $\Sp_{2m}\times\Sp_{2r}$. Two things about \eqref{eq:branch} should be
kept apart, because only one of them is ours. That the $B_{\eta,\mu}$ \emph{exist} and are integers
is not a computational finding: $\Sp_{2m}\times\Sp_{2r}$ is a reductive subgroup of maximal rank in
$\Sp_{2R}$, restriction is a ring map on representation rings, and the characters $\spc_\eta\otimes
\spc_\mu$ are a basis. That argument is all we use, and it needs nothing beyond the definitions;
the general Young-diagrammatic machinery for restrictions of classical groups to reductive subgroups
of maximal rank is Koike--Terada \cite{KoikeTerada} \stext, which we cite as the framework and not as
a source for the two particular pairs occurring here. Integrality rather than positivity is all we may assert,
because $\Phi_{2,R}$ is a \emph{virtual} character --- see Remark \ref{rem:virtual}. What is verified
below is our concrete algorithm for computing the $B$, not the decomposition itself. (The rank-one
branching algebra with a standard monomial basis, $\Sp_{2n}\downarrow\Sp_{2n-2}$, is Kim--Yacobi
\cite{KimYacobi}; ours is the branching to a \emph{product}, and we do not claim their construction
covers it.) Specialising $y$ as in Proposition~\ref{prop:red} and applying
Lemma~\ref{lem:T} termwise gives the formula this paper is organised around.

\begin{equation}\label{eq:A}
\Phi_{t,r}=\sum_{\mu} A_\mu\,\spc_\mu(z),
\qquad
A_\mu=\sum_{\eta} B_{\eta,\mu}\,\tau^C_t(\eta)\in\ZZ ,
\end{equation}
with $\tau^C_t(\eta)\in\{0,\pm1\}$ given by Lemma~\ref{lem:T}. Equivalently, the whole object is the
composite
\[
R_{\mathrm{virt}}(\Sp_{2R})\ \xrightarrow{\ \text{branch}\ }\
R(\Sp_{2r})\otimes R(\Sp_{t-2})\ \xrightarrow{\ 1\otimes\mathcal T_t\ }\ R(\Sp_{2r}),
\]
where $\mathcal T_t(\spc_\eta)=\tau^C_t(\eta)$, and $A_\mu$ is the value at the torsion element of
$M_\mu=\sum_\eta B_{\eta,\mu}[\spc_\eta]$. We call $M_\mu$ the \emph{virtual multiplicity class} on
this branch rather than the multiplicity space: the $B_{\eta,\mu}$ take both signs here, so $M_\mu$
is a class in the Grothendieck group and not a module. On the odd branch, for a single $\Lambda$, the
same object is a genuine multiplicity space.

\keybox{\emph{The factorisation in one sentence.} Everything inherited from the unspecialised object
sits in the branching coefficients $B$, everything created by the root of unity sits in the filter
$\mathcal T_t$, and the vanishing of $\Phi_{t,r}$ is the vanishing of the composite on every $\mu$.}

\noindent
Formula \eqref{eq:A} separates the two halves cleanly: everything inherited from the unspecialised
object sits in $B$, everything created by the torsion sits in $\mathcal T_t$, and the vanishing locus
of $\Phi_{t,r}$ is the locus where the composite is zero on every $\mu$.

\subsection{Verification}
Equation \eqref{eq:A} is formal once the coefficients $B_{\eta,\mu}$ of \eqref{eq:branch} are known.
The existence and integrality of \eqref{eq:branch} are theoretical, for the reason given just above:
restriction is a map of representation rings and the characters are a basis. What is checked here is
our concrete peeling algorithm for computing those coefficients, and the composite it feeds. It was
checked on $13$ forms -- nine with $t=4$ (so $\Sp_6\supset\Sp_2\times\Sp_4$) and four with $t=6$
(so $\Sp_8\supset\Sp_4\times\Sp_4$), four of them vanishing \stverif:

\begin{center}
\rowcolors{2}{hband}{white}
\begin{tabular}{l c}
\toprule
control & result \\
\midrule
the peeling reconstructs $\Phi_{2,R}$ exactly & remainder $0$ in $13/13$ \\
$B_{\eta,\mu}\in\ZZ$ & $13/13$ \\
\eqref{eq:A} against $\Phi_{t,r}$ computed independently & identical monomial by monomial, $13/13$ \\
decoy: no filter ($\tau\equiv1$) & disagrees $13/13$ \\
decoy: the $\tau$ of the wrong $t$ & disagrees $10/13$, \emph{ties in $3$} \\
\bottomrule
\end{tabular}
\end{center}

\noindent
We report the three ties rather than the ten disagreements alone: on those three forms the wrong
filter reproduces the object, so they are untested by that decoy, not confirmed by it. We proposed
the level threshold familiar from affine fusion as the explanation and then measured it, and it is
\emph{not} the explanation --- Problem \ref{prob:threshold} says why, and replaces it with what the
measurement does give: the wrong filter reproduces the object exactly when it agrees with the right
one \emph{pointwise on the support of} $B$, in $13$ of $13$ \stverif.

\section{The odd case, and why the two parities differ}\label{sec:odd}

\noindent
This section proves Theorem \ref{thm:odd}, and it is worth saying where each clause is settled:
(i) is Proposition \ref{prop:determinant}, (ii) Lemma \ref{lem:c1p} and Corollary
\ref{cor:unimodular}, (iii) Lemma \ref{lem:permanent} with Lemmas \ref{lem:laminar} and
\ref{lem:sdr}, (iv) Theorem \ref{thm:cancel}, and (v) Proposition \ref{prop:graph} with Corollary
\ref{cor:forest}. Everything before Proposition \ref{prop:transversal} is the reduction that makes
those statements possible.

Proposition \ref{prop:red} is uniform in the parity of $t$, but \eqref{eq:branch} is not: it is built
from $\Phi_{2,R}$, and for odd $t$ the reduction lands on $\Phi_{1,R'}$ instead, with
$R'=r+m'$ and $m'=(t-1)/2$. This section supplies the missing half. It turns out to be not a
translation of the even case but its opposite, and the reason is a single sign.

\subsection{The evaluation point, and its determinant}
The alphabet
\[
p_t=\bigl(1,\zeta,\dots,\zeta^{t-1},\,z_1^{\pm1},\dots,z_r^{\pm1}\bigr)
\]
is stable under $x\mapsto x^{-1}$, so it is a torus element of the \emph{orthogonal} group $O(N)$.
Its determinant is the product of the whole alphabet, and the free pairs contribute $1$:
\begin{equation}\label{eq:det}
\det p_t=\prod_{k=0}^{t-1}\zeta^{k}=\zeta^{\binom t2}=(-1)^{t+1}.
\end{equation}
The orbit block is the permutation matrix of a $t$-cycle, so \eqref{eq:det} is the sign of that
cycle. Verified exactly for $2\le t\le 11$ \stverif.

\begin{remark}[the parity dichotomy is \eqref{eq:det}]\label{rem:virtual}
For odd $t$ we have $\det p_t=+1$: the point lies in the identity component $SO(N)$, the restriction
$GL_N\downarrow SO_N$ is an ordinary restriction of an honest representation, and the multiplicities
in the first step are \emph{nonnegative}. For even $t$ we have $\det p_t=-1$: the point lies in the
other component of $O(N)$, the character there is not that of a restriction --- it is a twining, in
the sense of Jantzen and Kumar--Lusztig--Prasad \cite{Jantzen,KLP} --- and the symplectic expansion
is \emph{virtual}, with coefficients of both signs. Measured: no negative coefficient in $129$ odd
forms, against $12$ of $35$ even forms carrying one, least witness $\beta=(8,7,6,5,4,2,1,0)$
\stverif.

What \eqref{eq:det} controls should be stated carefully, because the tempting stronger sentence is
false. It decides \emph{ordinary against twined branching}, hence whether the object entering the
filter is a genuine character or a virtual one. It does not by itself forbid cancellation in the
final sum: even with $B_{\eta,\mu}\ge0$, the specialisation
$A_\mu=\sum_\eta B_{\eta,\mu}\tau^{B}_t(\eta)$
has signs from the filter and may cancel collectively. Whether it does is measured in
\S\ref{sec:oddverif}, and localised in \S\ref{sec:oddlocal}.
\end{remark}

\begin{proposition}[parity dichotomy]\label{prop:parity}
The parities of $t$ differ by two distinct structural mechanisms, and every structural difference in the
branching-and-filter factorisation of \S\S\ref{sec:filter}--\ref{sec:odd} follows from one of the
two. (A third mechanism, of a different kind, appears in \S\ref{sec:square}: it is combinatorial
rather than group-theoretic, and Corollary \ref{cor:selection} is independent of both switches.)
\begin{center}
\rowcolors{2}{hband}{white}
\small
\begin{tabular}{l l l}
\toprule
 & $t$ odd & $t$ even \\
\midrule
component of $O(N)$ containing $p_t$ & $SO(N)$ & $O(N)\setminus SO(N)$ \\
branching & $B_{R'}\to B_{m'}\times D_r$ & twining $\to C_m\times C_r$ \\
input to the filter & genuine character & possibly virtual \\
torsion block & $B_{m'}$, order $t=h+1$ & $C_m$, order $t=h+2$ \\
is the torsion element principal? & yes & no, for $m\ge2$ \\
$2$ invertible modulo $t$? & yes & no \\
do $B$ and $C$ see the same walls? & yes, via $\sigma_2$ & no: $a_i=t/2$ \\
\bottomrule
\end{tabular}
\end{center}
The first switch is \eqref{eq:det}; the second is invertibility of $2$ modulo $t$, and it is
Remark \ref{rem:BC}. The last row is Proposition \ref{prop:galois}: for odd $t$ the two
\emph{nonvanishing loci} are exchanged by the Galois element $\sigma_2$, which is why the two types
see the same walls; the loci are translates rather than equal, and the values still differ.
\end{proposition}

\subsection{The pipeline}
Two fixed points of $x\mapsto x^{-1}$ are available when $t$ is even, namely $\pm1$, and they are
spent in the Littlewood step. When $t$ is odd there is exactly one, and it goes with the torsion:
$(1,\zeta^{\pm1},\dots,\zeta^{\pm m'})$ are $2m'+1=t$ numbers, a torus element of $SO_t$, while the
$2r$ free entries are a torus element of $SO_{2r}$. Hence the odd analogue of \eqref{eq:branch} is
the branching to the maximal-rank subgroup
\begin{equation}\label{eq:oddbranch}
B_{R'}\ \longrightarrow\ B_{m'}\times D_r,
\qquad\text{that is}\qquad
SO_{2R'+1}\ \downarrow\ SO_{2m'+1}\times SO_{2r},
\end{equation}
again in the range of \cite{KoikeTerada}, and with no twining anywhere: writing $\Phi_{1,R'}=\sum_
\Lambda a^B_\Lambda\,\mathrm{o}_\Lambda(w)$ for the Littlewood restriction and branching each
$\mathrm{o}_\Lambda$ through \eqref{eq:oddbranch}, the composite is
\begin{equation}\label{eq:oddA}
\Phi_{t,r}=\sum_\mu A^D_\mu\,\mathrm{o}^{D_r}_\mu(z),
\qquad
A^D_\mu=\sum_{\Lambda,\eta}a^B_\Lambda\,B^{\mathrm{odd}}_{\Lambda;\eta,\mu}\,\tau^B_t(\eta),
\end{equation}
where $\tau^B_t(\eta)=\mathrm{o}^{B_{m'}}_\eta(1,\zeta^{\pm1},\dots,\zeta^{\pm m'})$ is the odd
torsion filter: the character of $SO_{2m'+1}$ at the element whose eigenvalues are \emph{all} the
$t$-th roots of unity.

One caveat, and it is not a defect of \eqref{eq:oddbranch}. Our $\Phi$ is invariant under each
$z_i\mapsto z_i^{-1}$ separately, hence under $O(2r)$ and not merely $SO(2r)$; so the weights $\mu$
and $\mu^*=(\mu_1,\dots,-\mu_r)$ carry equal coefficients and all statements about a highest weight
are statements about $\mu^+=(\mu_1,\dots,|\mu_r|)$. The disconnectedness does not disappear when $t$
is odd; it moves from the torsion side to the free side, where it is harmless.

\subsection{The odd filter is a theorem of others}
Here the odd case is in better shape than the even one. Its torsion element is regular \emph{and}
conjugate to the principal element of order $t=h+1$, where $h=2m'$ is the Coxeter number of
$B_{m'}$. That identification is immediate and we give it rather than measure it: in type $B_{m'}$
the half-sum of positive coroots is $\rho^\vee=(m',m'-1,\dots,1)$, so the principal element
$\rho^\vee(\zeta)$ has, in the natural $(2m'+1)$-dimensional representation, the eigenvalues
$\zeta^{\pm m'},\dots,\zeta^{\pm1},1$; and since $t=2m'+1$ those exhaust $\zt$ exactly once each.
So our torsion element \emph{is} the principal element of order $t$, up to the ordering of the torus
coordinates. The multiset $\{\alpha(g)\}$ was in addition computed for $t\le11$
\stverif. Since $t>h$ forces $\Phi_{0,t}=\varnothing$, Theorem 4.1 of \cite{NPP} applies verbatim. Its
constant is $1$ here for two independent reasons: in the odd type-$B$ case it is already $1$ in
\cite{NPP}, and in any case $\Phi_{0,t}=\varnothing$ makes the relevant group a torus, so no
classification is needed at all; \cite{Polo} completes that classification in general. It therefore
gives, with no work of ours:

\begin{proposition}[the odd filter, after {\cite{NPP}}]\label{prop:oddfilter}
$\tau^B_t(\eta)\neq0$ if and only if $\Phi_{\eta,t}=\{\alpha:t\mid\langle\eta+\rho,\alpha^\vee
\rangle\}$ is empty, and in that case $|\tau^B_t(\eta)|=1$.
\end{proposition}

\noindent
In the coordinates of $B_{m'}$ the condition reads: with $A_j=2\eta_j+2(m'-j)+1$, no $A_j\equiv0$ and
the classes $\min(A_j,t-A_j)$ pairwise distinct modulo $t$. There are two species of wall, not the
three of Lemma \ref{lem:T}, because the wall $2a\equiv0$ with $a\not\equiv0$ requires an even
modulus. We checked Proposition \ref{prop:oddfilter} before invoking it: the criterion decides
vanishing on $10$ of $10$ weights at $t=3$, $72$ of $72$ at $t=5$ and $406$ of $406$ at $t=7$, with a
decoy at modulus $t+1$ agreeing only $7$, $29$ and $208$ times \stverif.

\begin{remark}[the parity dichotomy, second switch: a root-length discrepancy]\label{rem:BC}
\begin{sloppypar}
Read through Lemma \ref{lem:regular}, the difference between the parities is a statement about the
lengths of roots under duality, and about $2$. The condition given by the \emph{long} root $2e_i$ of
$C_m$ is $2a_i\equiv0$; passing to the dual $C_m^\vee=B_m$ replaces that long root by the
short root $e_i$ and the condition by $a_i\equiv0$. For odd $t$ the two agree, because $2$ is
invertible modulo $t$; for even $t$ the first admits the extra solution $a_i=t/2$ and the second does
not. This is a root-length discrepancy under $B/C$ duality, not a distinction between finite and
affine walls: in the original group the offending condition is the regularity wall of a finite root.
\end{sloppypar}

This is what separates our filter from the criterion of \cite{NPP}, whose centraliser lives in the
dual group. Measured over the same $24\,809$ weights, the dual condition decides vanishing correctly
in every odd case and fails $1\,344$ times in the even ones, always at a weight with some $a_i=t/2$
\stverif; Figure \ref{fig:walls3d} draws one such case and shows the failures sitting on the three
planes $a_i=t/2$. The smallest is $\Sp_4$ at $t=6$ with $\eta=(1,0)$: there $\eta+\rho=(3,1)$ satisfies no
coroot condition, so the dual centraliser is a maximal torus, and yet
$\Theta=\xi+\xi^2+\xi^{-1}+\xi^{-2}=0$.

Two remarks on how to read that. First, it bears only on the converse: it is consistent with, and says
nothing against, the implication ``larger centraliser $\Rightarrow$ zero'' of Question 8.1 of
\cite{NPP}; what it excludes is a formula for the nonvanishing case depending only on the dimension of
the \emph{dual} centraliser. Second, passing to a simply connected cover of the dual does not repair
it: a central isogeny leaves the root system of the centraliser unchanged, so it cannot turn a short
root back into a long one.
\end{remark}

\begin{figure}[!ht]
\centering
\includegraphics[width=0.86\textwidth]{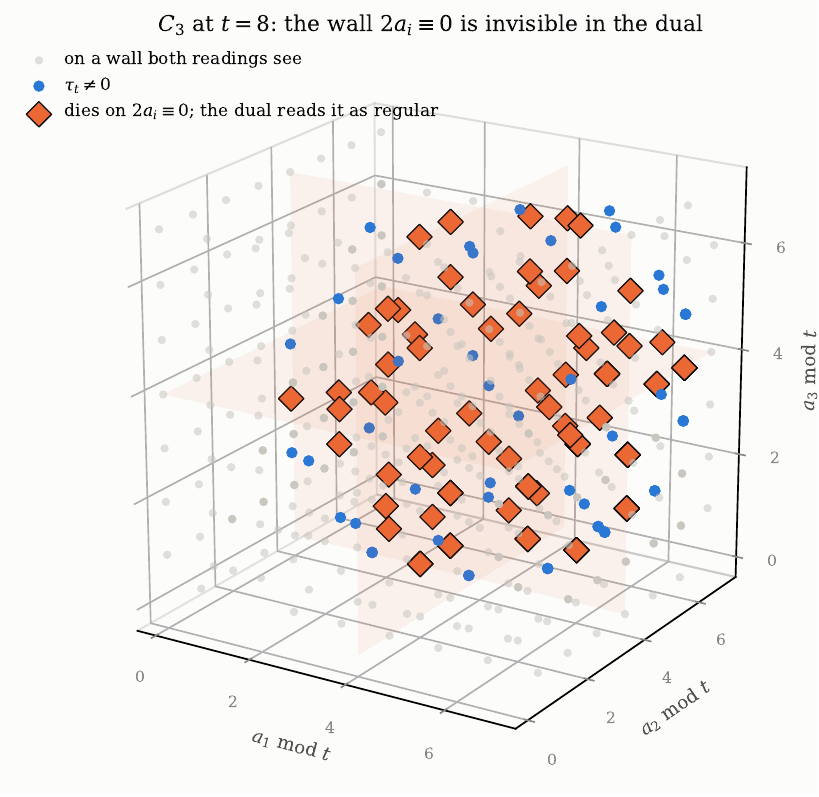}
\caption{The two readings of the filter, drawn in type $C_3$ at $t=8$. Each point is a weight placed
at its shifted coordinate $a=\eta+\rho$ in the cube of residues modulo $t$; the colour is decided by
computing the Weyl numerator exactly and comparing it with the two regularity conditions. Blue:
$\tau^C_t\neq0$. Grey: killed by a wall both the group and its dual see. Orange: killed by
$2a_i\equiv0$, which only the original group sees --- these are exactly the weights where the dual
criterion gives the wrong answer, and they land on the three faint planes $a_i=t/2$ without being put
there. On the $560$ points of this window, Lemma \ref{lem:regular} is right $560$ times, the dual
reading is wrong $96$; and no orange point lies off a plane. When $t$ is odd the orange set is empty,
because $2$ is then invertible modulo $t$. Computed by \texttt{fig\_walls3d.py}.}
\label{fig:walls3d}
\end{figure}

\subsection{Verification}\label{sec:oddverif}
The pipeline \eqref{eq:oddA} was checked end to end on nine forms, at $(t,r)=(3,2)$, $(5,2)$ and
$(3,3)$ \stverif: the composite reproduces $\Phi_{t,r}$ coefficient by coefficient in $9$ of $9$;
the Littlewood peeling in $B_{R'}$ leaves remainder $0$ with integer coefficients in $9$ of $9$; the
chirality $A^D_\mu=A^D_{\mu^*}$ holds in $9$ of $9$; and three decoys --- the filter at a $(t+2)$-th
root, no filter at all, and the torsion placed in $C_{m'}$ instead of $B_{m'}$ --- each disagree in
$9$ of $9$. The third decoy is the one that matters: it is the reading in which the fixed point $1$
is left out of the torsion block, and the data reject it.

\subsection{The odd conjecture localises}\label{sec:oddlocal}
On the odd side the statement can be taken apart, and the pieces are of different kinds. Write, for a
single $\Lambda$,
\begin{equation}\label{eq:cLmu}
c(\Lambda,\mu)\;=\;\sum_\eta B^{\mathrm{odd}}_{\Lambda;\eta,\mu}\,\tau^B_t(\eta),
\qquad\text{so that}\qquad
A^D_\mu=\sum_\Lambda a^B_\Lambda\,c(\Lambda,\mu).
\end{equation}
The inner sum involves no Littlewood step: it is the branching \eqref{eq:oddbranch} followed by the
fusion projection of Theorem \ref{thm:fusion}, applied to one irreducible. Then
$|A^D_{\mu^+_{\max}}|=1$ would follow from three statements, and we separate them because only the
last is combinatorial:
\begin{enumerate}
\item[(L1)] $c(\Lambda,\mu)\in\{0,\pm1\}$ for \emph{every} pair $(\Lambda,\mu)$ --- that is, the
multiplicity space of a single $\Lambda$ is already primitive after projection --- genuine here,
since for odd $t$ and a single $\Lambda$ the character is not virtual;
\item[(L2)] at $\mu=\mu^+_{\max}$, exactly one $\Lambda$ has $c(\Lambda,\mu)\neq0$;
\item[(L3)] that $\Lambda$ occurs in the Littlewood restriction with multiplicity $a^B_\Lambda=1$.
\end{enumerate}
\noindent
(L1) is proved in full below --- its numerator first, in Proposition \ref{prop:transversal}, and
then the division, in \S\ref{sec:invert}; (L2) and (L3) are not, and by the end of this
subsection they will have been sharpened into two named statements, Conjectures \ref{conj:tie} and
\ref{conj:L3}, each about a smaller object than the one it came from.

\noindent
All three hold on everything we have computed \stverif. Over the $123$ nonvanishing forms of
Observation \ref{obs:oddunit}, together with a third configuration at $(t,r)=(3,3)$: $c$ takes only
the values $0,\pm1$ on all $1\,735$ pairs $(\Lambda,\mu)$ that occur; exactly one $\Lambda$
contributes at the top weight in $82$ of $82$ forms at $(3,2)$, $41$ of $41$ at $(5,2)$ and $9$ of
$9$ at $(3,3)$; and the reconstruction \eqref{eq:cLmu} returns $A^D$ in every one. Of those pairs,
$999$ have surviving $\eta$ whose contributions cancel to zero --- so the cancellation the conjecture
is about happens \emph{inside a single $\Lambda$}, which is a far smaller object than the whole
composite. That is where we would attack it.

One caveat about that population, because it makes (L3) look better tested than it was. On
$\beta_i\le9$ \emph{every} $a^B_\Lambda$ equals $1$ --- all $476$ of them --- so (L3) cannot fail
there for any choice of $\Lambda$ whatever, and the agreement measures the box rather than the
statement. Multiplicity-freeness is not a general fact: at $t=3$, $r=2$ and $\beta_i\le13$ only
$47\%$ of the $70\,685$ coefficients are $1$, and they run up to $23$, the smallest witness being
$\beta=(10,7,4,3,2,1,0)$ with $a^B_{(2,0,0)}=2$ \stverif. Re-measured where it can fail --- $852$
shapes at $t=3$, $r=2$, $\beta_i\le12$, of which $78\%$ carry some $a^B_\Lambda>1$ --- (L3) holds in
$678$ of the $678$ shapes with a unique contributing $\Lambda$, and the contributor has
$a^B_\Lambda=1$ even on shapes whose largest available multiplicity is $13$. A second configuration
behaves the same: at $t=5$, $r=2$, $\beta_i\le12$ there are $368$ shapes, $52\%$ of them carrying
some $a^B_\Lambda>1$ and the largest available multiplicity being $4$, and (L3) holds in $323$ of the
$323$ shapes with a unique contributor \stverif. The remaining shapes --- $174$ and $45$ --- are the
ties of Remark \ref{rem:tie}.

\begin{remark}[why this is the odd side's advantage]
Nothing of the kind is available for even $t$. There $a_\Lambda$ takes both signs, so (L3) has no
meaning as stated and the sum \eqref{eq:cLmu} is not a sum of nonnegative multiples of $0,\pm1$. The
localisation is a consequence of the first switch of Proposition \ref{prop:parity}, and it is the
concrete sense in which the odd half of Conjecture \ref{conj:H} is the easier one.
\end{remark}

\begin{remark}[a route to (L1) that does not work, and why we record it]\label{rem:L1route}
Since $c(\Lambda,\cdot)$ collects the coefficients of $\mathrm{o}_\Lambda$ evaluated at the point
$p_t$ itself, the natural attack is Littlewood's: expand the Weyl numerator of $B_{R'}$ by Laplace
along the $m'$ rows carrying the roots of unity. Each term is a minor on the frozen block times its
complementary minor on the free one, and the frozen minor is a $B_{m'}$ Weyl numerator at the torsion
element --- so by the type-$B$ Weyl-numerator criterion underlying Proposition \ref{prop:oddfilter},
namely that two columns coincide as soon as $A_i\equiv\pm A_j$ and one vanishes as soon as
$A_i\equiv0$, it vanishes unless the corresponding $m'$-subset $S$ of the shifted exponents is
regular modulo $t$, and is $\pm$ the denominator otherwise. Since the exponents
are distinct, distinct $S$ give distinct complements and hence distinct weights: the map $S\mapsto\mu$
is injective, which we checked on $67$ shapes with no exception. That would give $c\in\{0,\pm1\}$ with
a closed formula.

It is wrong, and the smallest shape says so. Take $\Lambda=0$ at $t=3$, $R'=3$: the shifted exponents
are $(5,3,1)$, the regular singletons are $\{5\}$ and $\{1\}$, and they predict weights $(0,0)$ and
$(1,1)$. But $\Lambda=0$ is the trivial representation, whose restriction meets only $\mu=(0,0)$. The
predicted support is too big on $66$ of $67$ shapes \stverif. The step that fails is the
normalisation: we assumed the \emph{denominator} $N_\rho(p_t)$ reduces to a single Laplace term, and
$\Lambda=0$ is precisely the case showing it does not. Taken in product form the denominator carries
cross factors between the two blocks --- one $z_j^{\,t}-1$ per free variable --- which do not
distribute over the sum. A correct proof has to divide by them, and we do not have it. We record the
route because it isolates what a proof must handle, and because the injectivity it establishes is
probably still a step of one.
\end{remark}

\noindent
That obstruction turns out to have a name, and naming it is more useful than removing it.

\subsection{The obstruction is an equal-rank pair}\label{sec:oddgkrs}

\begin{proposition}[the cross factor is the relative denominator of an equal-rank pair]
\label{prop:crossden}
Let $t=2m'+1$ be odd and $R'=m'+r$ with $r\ge2$, so that $B_{R'}\supset B_{m'}\times D_r$ is an
inclusion of \emph{equal rank} with $D_r$ a genuine semisimple factor --- simple for $r\ge3$, while
$D_2\cong A_1\times A_1$ is semisimple but not simple, and $r=2$ is a case we use throughout. (At $r=1$, $D_1$ is the torus
$SO_2$ and every statement below is to be read with $W(D_1)$ trivial and the two half-spin
characters the two characters $z^{\pm\mu}$; nothing breaks, but the chirality language degenerates,
so we assume $r\ge2$ throughout the rest of \S\ref{sec:odd}.) The positive roots of $B_{R'}$ not
belonging to the subgroup are exactly
$e_i\pm f_j$ $(1\le i\le m',\,1\le j\le r)$ and $f_j$ $(1\le j\le r)$, where $e$ and $f$ are the
coordinates of the frozen and free blocks. Evaluated with the frozen block at the torsion element
$y_i=\zeta^i$ and the free block at $z$, their product is
\begin{equation}\label{eq:crossden}
\prod_{j=1}^{r}\Bigl[(1-z_j)\prod_{i=1}^{m'}(1-\zeta^{i}z_j)(1-\zeta^{-i}z_j)\Bigr]
=\prod_{j=1}^{r}\bigl(1-z_j^{\,t}\bigr),
\end{equation}
which is the relative spinor denominator of the pair after the same specialisation, up to a monomial
and a sign that are both explicit:
\begin{equation}\label{eq:crosssign}
\prod_{j=1}^{r}\bigl(1-z_j^{\,t}\bigr)
\;=\;(-1)^{r}\Bigl(\prod_{j=1}^{r}z_j^{\,t/2}\Bigr)\,\Delta_t,
\qquad
\Delta_t=\prod_{j=1}^{r}\bigl(z_j^{t/2}-z_j^{-t/2}\bigr).
\end{equation}
\end{proposition}

\begin{proof}
The listed roots are the complement of $\{e_i\pm e_j,\,e_i\}\cup\{f_i\pm f_j\}$ in the positive roots
of $B_{R'}$, which is the assertion about which roots occur. For the product, fix $j$. The exponents
appearing are $\zeta^{0}$ and $\zeta^{\pm i}$ for $1\le i\le m'$, and since $t=2m'+1$ these are the
$t$ distinct $t$-th roots of unity, each once. Hence the bracket is
$\prod_{k=0}^{t-1}(1-\zeta^{k}z_j)=1-z_j^{\,t}$. For \eqref{eq:crosssign}, one coordinate at a time:
$z^{t/2}\bigl(z^{t/2}-z^{-t/2}\bigr)=z^{t}-1$, so $1-z^{t}=-z^{t/2}\bigl(z^{t/2}-z^{-t/2}\bigr)$, and
taking the product over $j$ collects $(-1)^r$.
\end{proof}

\begin{remark}[what this says about (L1)]\label{rem:gkrs}
The factor that defeats the Laplace expansion of Remark \ref{rem:L1route} is therefore not a
normalisation accident. It is the relative denominator of an equal-rank pair, and equal-rank pairs
have a character formula built for them: Gross, Kostant, Ramond and Sternberg \cite{GKRS} show that
the restricted character of an irreducible of $G$, multiplied by the difference of the two half-spin
characters of $\mathfrak g/\mathfrak h$, equals an alternating multiplicity-free sum of characters
of $H$ indexed by a set of Weyl-coset representatives; the multiplet in the numerator is the one
computed by Kostant's cubic Dirac operator \cite{Kostant99}, with the loop-group extension in
\cite{Landweber}. After the frozen specialisation, that relative spinor denominator specialises, up
to an overall monomial and sign, to $\prod_j(z_j^{t/2}-z_j^{-t/2})$, which is \eqref{eq:crossden}.

Carrying it out. Two vectors are needed and they must not be confused, because $\rho_{B_{R'}}$ is
half-integral. Write
\begin{equation}\label{eq:uV}
u:=\Lambda+\rho_{B_{R'}},
\qquad
V:=2u=2(\Lambda+\rho_{B_{R'}})\in\ZZ^{R'} .
\end{equation}
The Weyl-coset combinatorics below is that of $u$; every statement about residues modulo $t$ is
about $V$, for the reason given in Lemma \ref{lem:regular} --- in type $B$ the vector one may reduce
modulo $t$ is $2(\eta+\rho)$, not $\eta+\rho$. In particular the frozen block of a representative is
$A_w=V|_{S(w)}=2(\eta_w+\rho_{B_{m'}})$, which is the vector Lemma \ref{lem:T} and Corollary
\ref{cor:oddsign} take. Let $W^1$ be the set of $w$ with $w(u)$
strictly $H$-dominant, so that $\eta_w$ and $\mu_w$ are its two blocks shifted back by the
respective $\rho$'s. This set does not depend on $\Lambda$, and is the usual set of minimal coset
representatives: $u$ is $G$-regular, hence $H$-regular, so each $W_H$-orbit inside $W_G\cdot u$
contains exactly one strictly $H$-dominant element and $|W^1|=|W_G|/|W_H|$ \stverif. Applying the fusion projection $q^B_t$ of Theorem \ref{thm:fusion} to the
$B_{m'}$ factor kills every $w$ with $\tau^B_t(\eta_w)=0$ and leaves
\begin{multline}\label{eq:gkrsfilt}
\Bigl(\sum_\mu c(\Lambda,\mu)\,\chi^{D_r}_\mu\Bigr)\cdot\Delta_t
\;=\;s_r\sum_{w\in W^1}(-1)^{\ell(w)}\tau^B_t(\eta_w)\,\chi^{D_r}_{\mu_w}
\;=:\;s_r\sum_\mu\nu(\Lambda,\mu)\,\chi^{D_r}_\mu ,\\
\Delta_t=\prod_{j=1}^{r}\bigl(z_j^{t/2}-z_j^{-t/2}\bigr).
\end{multline}
\emph{The sign $s_r$ is named and carried from here on}, because it is not the same object as the
one in Proposition \ref{prop:determinant} and an anonymous $\pm$ in both places let an earlier
version of this paper identify them. Half of it is now a theorem rather than a measurement:
Proposition \ref{prop:crossden} turns the cross factor into $\Delta_t$ at the cost of exactly
$(-1)^r$ by \eqref{eq:crosssign}, so the \emph{cyclotomic} part of $s_r$ is $(-1)^r$. What is not
settled analytically is the residual convention factor --- the order in which \cite{GKRS} takes the
difference $S^+-S^-$ of the half-spin characters --- and the measurement says that factor is $+1$
here. Checked on $245$ dominant $\Lambda$ at $(t,r)=(3,2),(5,2),(3,3),(7,2)$: \eqref{eq:gkrsfilt}
holds in every case with
\[
s_r=(-1)^r\quad\text{throughout, that is } +1 \text{ at } r=2 \text{ and } -1 \text{ at } r=3,
\]
so $s_r$ is constant within each $r$ and \emph{not} constant across them, while a decoy dilating by
$t+2$ and a decoy dropping the sign $\prod_j\varepsilon_j$ hold in none \stverif. So the reframing is
a computation and not an analogy, and the sign in it is now $(-1)^r$ times a convention we have
measured and not derived.

What it settles. GKRS is multiplicity-free with signs and $\tau^B_t\in\{0,\pm1\}$ by Proposition
\ref{prop:oddfilter}, so the only way the numerator of \eqref{eq:gkrsfilt} could leave $\{0,\pm1\}$
is two representatives sharing a $\mu_w$ --- the projection $q^B_t$ forgets $\eta_w$, and
multiplicity-freeness is a statement about the \emph{pair}. That cannot happen.

\subsection{The numerator is a signed transversal count}\label{sec:oddnum}

\begin{lemma}[the coset representative is remembered by its $D_r$ block]\label{lem:muinj}
For $B_{R'}\supset B_{m'}\times D_r$ the map $w\mapsto\mu_w$ is injective on $W^1$. Consequently
$\nu(\Lambda,\mu)\in\{0,\pm1\}$ for every $\Lambda$ and $\mu$.
\end{lemma}

\begin{proof}
Recall $u=\Lambda+\rho_{B_{R'}}$ from \eqref{eq:uV}, whose coordinates are strictly decreasing and
strictly positive. Strict $H$-dominance forces the first block of $w(u)$ to be $m'$ of those taken with
positive signs in decreasing order, and the second block to be the complementary $r$ coordinates in
decreasing order of absolute value with all signs positive except possibly the last. So $w$ is
exactly the datum of a subset $S\subset\{1,\dots,R'\}$ with $|S|=m'$ together with a $D_r$
chirality, and $|W^1|=2\binom{R'}{m'}$, which is
$|W(B_{R'})|/|W(B_{m'})||W(D_r)|$ as it must be. Now $\mu_w+\rho_{D_r}$ is the complementary block,
so its absolute values recover $S^{\mathrm c}$ and hence $S$; and its last coordinate is
$\pm u_{\min S^{\mathrm c}}\neq0$, so its sign recovers the chirality. Hence $\mu_w$ determines $w$.
The consequence is immediate: distinct $w$ contribute to distinct $\mu$, so each $\nu(\Lambda,\mu)$
is a single term $(-1)^{\ell(w)}\tau^B_t(\eta_w)\in\{0,\pm1\}$.
\end{proof}

\noindent
As a control: $|W^1|=2\binom{R'}{m'}$
and $w\mapsto\mu_w$ injective on all of $W^1$ in $147$ of $147$ weights over five $(t,r)$, while the
decoy $w\mapsto\eta_w$ is injective in $0$ of $147$ \stverif.

Combining the lemma with the closed form \eqref{eq:oddsign} removes representation theory from the
numerator altogether. Partition $\{1,\dots,R'\}$ by the folded class of $V_i$ modulo $t$ --- the
\emph{doubled} shifted exponent of \eqref{eq:uV}, never $u_i$, which is half-integral; because
$t=2m'+1$ there are exactly $m'$ nonzero folded classes, namely $1,\dots,m'$, and one class $0$.
Write $n_j=\#\{i:\ V_i\ \text{folds to class}\ j\}$ for $1\le j\le m'$, and write $\delta_S$ for the
determinant $\delta\bigl(V|_S\bigr)$ of Corollary \ref{cor:oddsign}, read --- as there --- against
the decreasing order $m',m'-1,\dots,1$.

\begin{proposition}[the numerator is a signed transversal count]\label{prop:transversal}
With the notation above:
\begin{enumerate}
\item[\rm(i)] $\tau^B_t(\eta_w)\neq0$ if and only if $S(w)$ is a \emph{transversal}: it contains
exactly one index from each nonzero class and none from class $0$.
\item[\rm(ii)] Consequently $\displaystyle|\mathrm{supp}\,\nu(\Lambda,\cdot)|
=2\prod_{j=1}^{m'}n_j$, counted among \emph{unfolded} $D_r$ weights, where the two chiralities of a
transversal are two distinct points.
\item[\rm(iii)] In particular $\nu\equiv0$ if and only if some $n_j=0$ --- some nonzero residue
class is missed by $V=2(\Lambda+\rho_{B_{R'}})$.
\item[\rm(iv)] And the value is explicit: at the point indexed by $S$ and a chirality it is
\[
\epsilon_t\,\det(w)\,\delta_S\ \in\ \{\pm1\},
\qquad
\epsilon_t=\Bigl(\tfrac{-2}{t}\Bigr)^{(t+3)/2}(-1)^{m'(m'-1)/2},
\]
the constant and the reading order being those of Corollary \ref{cor:oddsign}.
\item[\rm(v)] Let $S_{\min}$ be the transversal taking, in each nonzero class, the index with the
\emph{largest} label --- equivalently the smallest value of $V$. Then the alternant index of
$(S_{\min},+)$ \emph{dominates coordinatewise} the index of every other point of
$\mathrm{supp}\,\nu(\Lambda,\cdot)$, strictly in at least one coordinate. In particular it is the
unique maximiser of $\langle w_0,\cdot\rangle$ for \emph{every} $w_0$ with all coordinates positive.
\end{enumerate}
\end{proposition}

\begin{proof}
By Corollary \ref{cor:oddsign}, $\tau^B_t(\eta_w)\neq0$ exactly when the $m'$ folded classes of
$V|_{S(w)}$ are a permutation of $\{1,\dots,m'\}$. Since those are \emph{all} the nonzero classes
available, that says precisely that $S(w)$ meets each of them once and avoids class $0$, which is
(i). For (ii), a transversal is a choice of one index from each of the $m'$ nonzero classes, giving
$\prod_j n_j$ subsets, and each carries two chiralities; by Lemma \ref{lem:muinj} distinct $w$ land
on distinct $\mu_w$, so no cancellation occurs and the support has exactly that size. (iii) is
immediate, and (iv) restates \eqref{eq:oddsign} with $\det(w)=(-1)^{\ell(w)}$.

For (v) we argue by exchange. Index alternants, as throughout, by the doubled vector
$2(\mu+\rho_{D_r})$, so that the point of $\nu$ attached to $(S,\text{chirality})$ has index
$V|_{S^{\mathrm c}}$ sorted decreasingly, with the last entry carrying the chirality. Let $S$ be a
transversal differing from $S_{\min}$ in one class, say $S$ takes the index $i$ there and $S_{\min}$
takes $i'$. The coordinates of $V$ are strictly decreasing, so the largest label carries the
smallest value and $V_{i'}<V_i$. Passing from $S$ to $S_{\min}$ therefore removes $V_{i'}$ from the
complement and inserts $V_i>V_{i'}$: the complement's multiset gains a strictly larger element in
place of a smaller one, so every order statistic of the sorted complement weakly increases and at
least one strictly. Any transversal reaches $S_{\min}$ by finitely many such single-class exchanges,
so the sorted complement of $S_{\min}$ dominates coordinatewise that of every other transversal.
Finally the two chiralities of a fixed $S$ differ only in the sign of the last entry, which is
nonzero because no $V_i$ vanishes, so the positive one dominates the negative one coordinatewise as
well. A functional with all coordinates positive is therefore strictly maximised there, and
uniquely.
\end{proof}

\begin{remark}[the chirality, and why the half-spin difference is the right divisor]\label{rem:chirality}
Two of the objects in \eqref{eq:gkrsfilt} behave oppositely under the outer automorphism $\sigma$ of
$D_r$ that negates the last coordinate, and the mismatch is the point rather than an accident.
The numerator is \emph{anti}-invariant: the two chiralities of one transversal are two coset
representatives differing by a single sign change, so their lengths differ in parity and
$\nu(\Lambda,\mu^*)=-\nu(\Lambda,\mu)$. The divisor is anti-invariant too, since $\sigma$ flips the
sign of the last factor of $\prod_j(z_j^{t/2}-z_j^{-t/2})$. Applying $\sigma$ to
\eqref{eq:gkrsfilt} therefore gives $c^\sigma\Delta_t=c\,\Delta_t$, and $\Delta_t$ is not a zero
divisor, so
\[
c(\Lambda,\mu^*)=c(\Lambda,\mu):
\]
the \emph{quotient} is invariant. Verified on $654$ chirality pairs over four $(t,r)$, with no
exception \stverif.

That is the structural reason the half-spin difference, and not some other multiple of it, is what
divides here: it is exactly anti-invariant, so quotienting by it lands the object in the invariants,
which is where $A^D$ and the statements (L1)--(L3) live. It also fixes a bookkeeping point. Part
(ii) counts an unfolded support; $\mu^+_{\max}$, (L2) and Remark \ref{rem:tie} live on the folded
weights $\mu^+=(\mu_1,\dots,|\mu_r|)$, where a chirality pair is one point. The counts differ by the
factor $2$ and the two must not be compared without folding first.
\end{remark}

\begin{remark}[coordinatewise is not the root order]\label{rem:notdominance}
Coordinatewise domination is strictly weaker than dominance in the root order of $D_r$, and the gap
is exactly the chirality. If $x$ and $x^*$ differ only in the sign of the last coordinate then
$x-x^*=2x_r\,e_r$, and $e_r$ is not a nonnegative combination of the positive roots $e_i\pm e_j$ of
$D_r$: one has $2e_r=(e_{r-1}+e_r)-(e_{r-1}-e_r)$, with a negative coefficient. So $(\mu_1,\dots,
\mu_{r-1},a)$ and $(\mu_1,\dots,\mu_{r-1},-a)$ are two dominant $D_r$ weights that part (ii) counts
separately and that the root order leaves \emph{incomparable}, however positive the functional
preferring the first. What (v) gives is the unique coordinatewise maximum, which is what the
back-substitution of \eqref{eq:gkrsfilt} consumes; a dominance-order statement would need the two
chiralities folded to the $O(2r)$-invariant datum $\mu^+$ first, and we do not prove one.
Coordinatewise domination was checked against the honest support in $1246$ of $1246$ pairs over five
$(t,r)$, strict in every one \stverif.
\end{remark}

\noindent
Checked as a control on $301$ weights over $(t,r)=(3,2),(5,2),(7,2),(3,3),(5,3)$: (ii), (iii), the
support point by point, and the values point by point all hold in $301$ of $301$, while a decoy
counting classes modulo $t+2$ holds in $68$ and a decoy dropping the chirality factor $2$ in $24$
\stverif. Part (v) likewise: the top is at $S_{\min}$ with positive chirality, and unique, in $277$
of $277$ shapes with $\nu\neq0$, while the decoy taking the \emph{smallest} label in each class
gives the top in $32$ and the decoy with negative chirality in $0$ \stverif.

Part (v) is the one with a practical edge: the highest weight of the numerator is read off
$V=2(\Lambda+\rho)$ directly --- fold its residues, take the largest label in each class --- with no
search over $W^1$ and no character evaluated. That is the half of (L2) which was a search turned
into a formula.

\begin{remark}[the tie is the cancellation]\label{rem:tie}
Part (v) prices each $\Lambda$ separately, so it is natural to ask whether the composite obeys the
same rule: is $\mu^+_{\max}(\beta)$ the largest of the tops of the individual $c(\Lambda,\cdot)$,
over the $\Lambda$ with $a^B_\Lambda\neq0$? The answer is a clean biconditional, and it locates the
one thing that can go wrong.

Over the $132$ shapes at $(t,r)=(3,2),(5,2),(3,3)$: the prediction is correct in $105$ and wrong in
$27$, and it is correct \emph{exactly} on the $105$ where the largest top is attained by a single
$\Lambda$ --- the two agree in $132$ of $132$ \stverif. When two $\Lambda$ share the top, their
contributions cancel there and $\mu^+_{\max}$ drops strictly below the prediction; the smallest
witness is $\beta=(8,7,6,3,2,1,0)$ at $t=3$, $r=2$, where $\Lambda=(2,2,2)$ and $\Lambda=(2,2,0)$
both predict $(2,2)$ and the true value is $(2,1)$. A separate statistic, over all $132$ and not
only over the $105$: the $\Lambda$ realising the maximum is the one contributing at $\mu^+_{\max}$
in $117$ of $132$. Two decoys: predicting with the opposite transversal $S_{\max}$ is right in $39$
of $132$, and the older guess that the dominance-largest $\Lambda$ is the contributor in $104$ of
$132$ \stverif.

So (L2) splits in two. Where the top is attained once, it is not a search at all: the contributing
$\Lambda$ is the argmax of a vector read off $V=2(\Lambda+\rho)$. Where it is attained twice, the
question is whether the tie always cancels, and that --- not the extremal coefficient --- is what a
proof of (L2) has to settle.
\end{remark}

\noindent
Both halves deserve a name, because both are now precise statements about small objects rather than
about $A^D$. Write $T(\Lambda)$ for the top weight of $c(\Lambda,\cdot)$ delivered by Proposition
\ref{prop:transversal}(v), $c_{\mathrm{top}}(\Lambda)=c(\Lambda,T(\Lambda))\in\{\pm1\}$ for its value
there, and
\[
M(\beta)=\max\{\,T(\Lambda)\;:\;a^B_\Lambda\neq0\,\}
\]
for the largest of them, the maximum taken in the order used by \eqref{eq:mumax}.

\begin{conjecture}[the tie cancels]\label{conj:tie}
If $M(\beta)$ is attained by more than one $\Lambda$, then
\[
\sum_{\Lambda\,:\,T(\Lambda)=M(\beta)} a^B_\Lambda\,c_{\mathrm{top}}(\Lambda)\;=\;0 .
\]
\end{conjecture}

\noindent
This is exactly the content of the biconditional above: it holds in the $27$ of $132$ shapes where
the tie occurs, and in those shapes $\mu^+_{\max}$ drops strictly below $M(\beta)$ \stverif.

\emph{What it gives, and what it does not.} Together with Proposition \ref{prop:transversal}(v) it
gives (L2) \emph{on the shapes where $M(\beta)$ is attained once} --- there $\mu^+_{\max}=M(\beta)$,
the contributing $\Lambda$ is the argmax, and there is nothing left to prove. It does \emph{not} give
(L2) in general, and an earlier version of this paper claimed that it did. The gap is visible in an
abstract two-line model: let $c_1,c_2$ share the top $M$ with $c_1(M)=+1$ and $c_2(M)=-1$, so
Conjecture \ref{conj:tie} holds and $M$ cancels; nothing above prevents $c_1(m)=c_2(m)=+1$ at the
next weight $m<M$, and then the true $\mu^+_{\max}$ is $m$ with \emph{two} contributing $\Lambda$, so
(L2) fails. Closing that would need the statement iterated down the extremal layers --- cancel the
tied top, truncate, take the maximum of what is left, and demand the same thing again until a layer
survives --- and we do not have the iterated statement. We therefore keep (L2) as the conjecture it
is, and Conjecture \ref{conj:tie} as what it is: the exact account of the non-tie case, and the
reason the naive prediction fails in the other $27$.

\begin{conjecture}[extremal multiplicity one]\label{conj:L3}
If $M(\beta)$ is attained by a unique $\Lambda^*$, then $a^B_{\Lambda^*}=1$.
\end{conjecture}

\noindent
This is (L3) \emph{where $M(\beta)$ is attained once}, which is where it is not vacuous. It holds in
$678$ of the $678$ such shapes at $t=3$,
$r=2$, $\beta_i\le12$ --- a population in which $78\%$ of the shapes carry some $a^B_\Lambda>1$, up
to $13$ --- and in $323$ of the $323$ at $t=5$, $r=2$, $\beta_i\le12$, where $52\%$ do, up to $4$
\stverif. It is not the whole of (L3), for the same reason Conjecture \ref{conj:tie} is not the whole
of (L2): if $M(\beta)$ is a tie that cancels, the true top is lower, and Conjecture \ref{conj:L3}
says nothing about the multiplicity of whoever contributes there. Neither conjecture is a statement about the filter: both are
about the Littlewood restriction coefficients $a^B_\Lambda$ and about which $\Lambda$ is extremal, so
the natural setting for an attack is a branching rule fine enough to see multiplicities outside the
stable range. Jang and Kwon give one --- a formula for the branching $GL_n\downarrow O_n$
generalising Littlewood's, in terms of Littlewood--Richardson tableaux with flag conditions that
vanish in the stable range \cite{JangKwon} --- and it is precisely outside that range that our
$a^B_\Lambda$ exceed $1$.

Figure \ref{fig:transversal} is the proposition drawn: the residues of $V=2(\Lambda+\rho)$ sorted into
their folded classes, a transversal picked out of them, and --- on the right --- a shape where one
class is unhit and the numerator is therefore empty.

\begin{figure}[htbp]
\centering
\includegraphics[width=\linewidth]{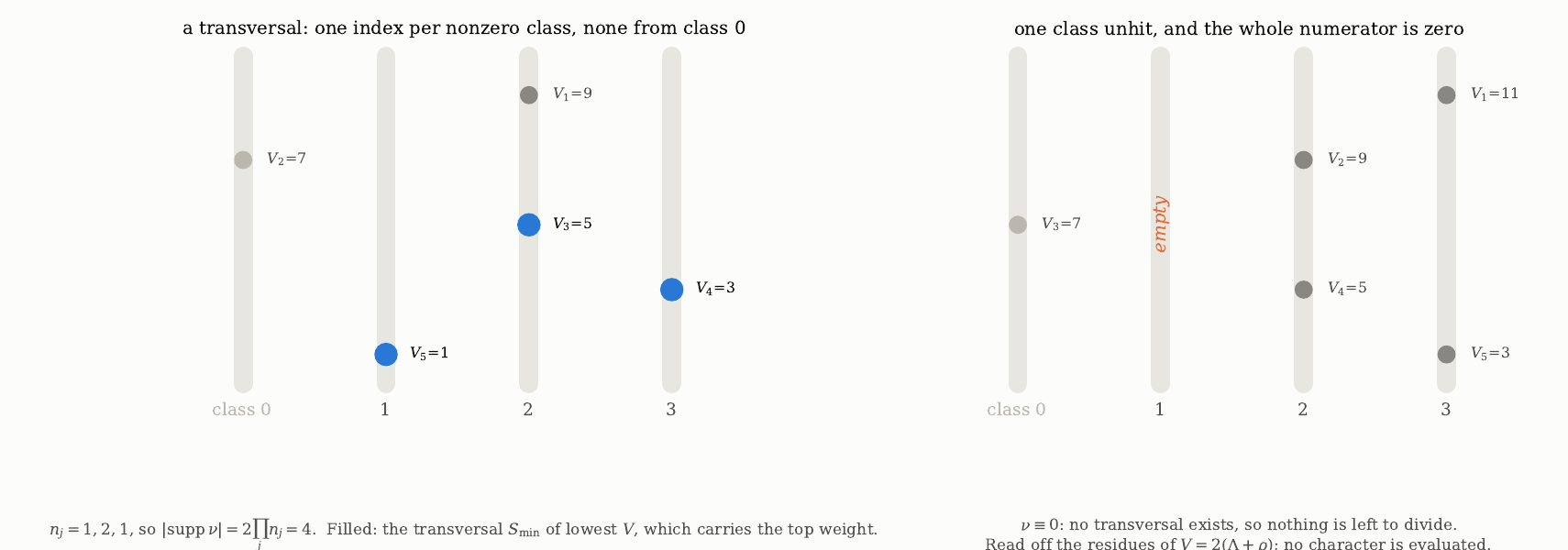}
\caption{Proposition \ref{prop:transversal}, drawn, at $t=7$ and $r=2$, so $m'=3$ and $R'=5$: four
folded classes $0,1,2,3$ and the five doubled shifted exponents $V_i=2(\Lambda+\rho)_i$ to place
among them.
\textbf{Left}, $\Lambda=0$: a summand of $\nu$ is a choice of one index from each \emph{nonzero}
class and none from class $0$. Here $n=(1,2,1)$, so there are $\prod_j n_j=2$ transversals and
$|\mathrm{supp}\,\nu|=4$ once the two chiralities are counted, which is part (ii). Filled is
$S_{\min}$, the transversal of lowest $V$; by part (v) it is the one carrying the top weight, and
in the picture it is simply the lowest dot in each column. \textbf{Right}, $\Lambda=(1,1,1,1,1)$:
class $1$ is unhit, no transversal exists, and $\nu\equiv0$ --- part (iii), a vanishing criterion
read off the residues with no character evaluated. Computed by \texttt{fig\_transversal.py}, which
refuses to draw unless $|\mathrm{supp}\,\nu|$ from an honest enumeration of $W^1$ agrees with
$2\prod_j n_j$ on both shapes, the highlighted set is admissible while a non-transversal is not, the
right-hand numerator really is empty, and $S_{\min}$ really attains the maximum.}
\label{fig:transversal}
\end{figure}

\subsection{The same reading at even $t$}\label{sec:evennum}

\begin{proposition}[the same reading for even $t$, and the parity inside the numerator]
\label{prop:eventransversal}
Let $t=2m+2$ be even, $R=m+r$, and let $V^{\mathrm{ev}}=\Lambda+\rho_{C_R}$ be the shifted exponent
vector of $\Sp_{2R}$, whose classes modulo $t$ are folded as in Lemma \ref{lem:T}. Then:
\begin{enumerate}
\item[\rm(i)] $\Sp_{2R}\supset\Sp_{2m}\times\Sp_{2r}$ has equal rank, its complementary positive
roots are exactly $e_i\pm f_j$ --- with no $f_j$ alone --- and freezing $y_i=\xi^i$ gives the
specialised relative denominator $\prod_{j}(1-z_j^{\,t})/(1-z_j^{2})$.
\item[\rm(ii)] $W^1$ is the set of subsets $S\subset\{1,\dots,R\}$ with $|S|=m$, so
$|W^1|=\binom Rm$ and there is \emph{no} chirality; and $w\mapsto\mu_w$ is injective.
\item[\rm(iii)] $\tau^C_t(\eta_w)\neq0$ if and only if $S$ is a transversal of the classes
$1,\dots,m$ avoiding \emph{both} forbidden classes $0$ and $t/2$; hence
$|\mathrm{supp}\,\nu|=\prod_{j=1}^{m}n_j$, with no factor $2$, and $\nu\equiv0$ exactly when some
$n_j=0$.
\end{enumerate}
\end{proposition}

\begin{proof}
(i) The positive roots of $C_R$ are $e_i\pm e_j$ and $2e_i$; those of the subgroup are the same
within each block, so the complement is $\{e_i\pm f_j\}$ and nothing else --- there is no $2e_i$
crossing the blocks, which is the single difference from type $B$, where $f_j$ itself is
complementary. Freezing $y_i=\xi^i$ for $1\le i\le m$, the exponents $\xi^{\pm i}$ run over every
$t$-th root of unity except $\xi^0=1$ and $\xi^{t/2}=-1$, since $t=2m+2$; hence per free variable
$\prod_{i=1}^m(1-\xi^iz)(1-\xi^{-i}z)=\prod_{k=0}^{t-1}(1-\xi^kz)/\bigl((1-z)(1+z)\bigr)
=(1-z^{\,t})/(1-z^{2})$.

(ii) Both blocks are of type $C$, so both Weyl groups contain all sign changes and strict
$H$-dominance forces every entry of $w(V^{\mathrm{ev}})$ positive; $w$ is therefore the datum of the
subset $S$ alone, and $|W(C_R)|/|W(C_m)||W(C_r)|=2^RR!/(2^mm!\,2^rr!)=\binom Rm$ as it must be. The
complementary block recovers $S^{\mathrm c}$ and hence $S$, so $w\mapsto\mu_w$ is injective.

(iii) By Lemma \ref{lem:T} the frozen block survives exactly when the folded classes of
$V^{\mathrm{ev}}|_S$ are a permutation of $\{1,\dots,m\}$. Modulo $t=2m+2$ the folded classes are
$0,1,\dots,m,t/2$, so there are $m$ admissible ones and two forbidden; $S$ must meet each admissible
class once and neither forbidden one. Counting and the vanishing criterion follow as in
Proposition \ref{prop:transversal}, without the chirality factor by (ii).
\end{proof}

\noindent
Checked on $112$ weights over $(t,r)=(4,2),(6,2),(8,2),(6,3)$: the support, the values and the
vanishing criterion all agree with the transversal reading in $112$ of $112$, $|W^1|=\binom Rm$ in
$112$ of $112$, and the coordinatewise top is at $S_{\min}$ in $101$ of the $101$ shapes with
$\nu\neq0$; a decoy using the odd rule --- one forbidden class instead of two --- is right in only
$72$ \stverif. The identity itself was checked separately, below.

So the numerator carries its own parity dichotomy, and it is not the one the rest of the paper turns
on:
\[
\begin{array}{lcc}
 & t\ \text{odd} & t\ \text{even}\\[2pt]
\text{forbidden residue classes} & 0 & 0,\ t/2\\
\text{chirality} & 2 & 1\\
|\mathrm{supp}\,\nu| & 2\prod_j n_j & \prod_j n_j\\
\text{divisor per coordinate} & z^{t/2}-z^{-t/2} &
\bigl(z^{t/2}-z^{-t/2}\bigr)/\bigl(z-z^{-1}\bigr)
\end{array}
\]
This is a dichotomy \emph{inside} the equal-rank numerator, and it owes nothing to the virtuality of
$a_\Lambda$; the two forbidden classes against one are Remark \ref{rem:BC} said in the language of
transversals.

\begin{remark}[what the even side still does not have]\label{rem:eventransversal}
Before the deficit, the credit: the identity itself goes through as well. Writing $x=\mu+\rho_{C_r}$ for the alternant index, the
even denominator acts on it by the $(t/2)^r$ shifts with exponents $t/2-1,t/2-3,\dots,-(t/2-1)$ in
each coordinate, straightened by $W(C_r)$; and
\[
\Bigl(\sum_\mu c(\Lambda,\mu)\,\chi^{C_r}_\mu\Bigr)\cdot\prod_{j=1}^{r}\frac{1-z_j^{\,t}}{1-z_j^{2}}
\;=\;\sum_{w\in W^1}(-1)^{\ell(w)}\tau^C_t(\eta_w)\,\chi^{C_r}_{\mu_w}
\]
holds on all $266$ weights tested at $(t,r)=(4,2),(6,2),(8,2),(6,3),(10,2)$, with the overall sign
$+1$ throughout and no exception; a decoy using the odd denominator, one using the exponents of
$t+2$, and one straightening by $W(D_r)$ hold in $0$, $0$ and $5$ of the $222$ nontrivial cases
\stverif.

That comparison is worth making explicit, because it is a difference between the parities that has
nothing to do with virtuality. The odd denominator
$\prod_j(z_j^{t/2}-z_j^{-t/2})$ contributes \emph{two} terms per coordinate, with opposite signs: it
is a difference of half-spin characters, and dividing by it is a finite-difference problem. The even
one contributes $t/2$ terms per coordinate, \emph{all with sign} $+1$: it is
$\prod_j(z_j^{t/2}-z_j^{-t/2})/(z_j-z_j^{-1})$, a symmetric sum rather than a difference. Inverting
the second is a strictly larger problem than inverting the first, and it grows with $t$.

What the even side does \emph{not} have is the reduction. (L1) is a statement about a single
$\Lambda$, and for a single $\Lambda$ the even character is genuine, so the numerator behaves; but
the even conjecture is not a per-$\Lambda$ statement, because $a_\Lambda$ takes both signs there and
the localisation (L1)--(L3) has no meaning as stated. So the even side gains the numerator and the
identity, and still has two obstructions where the odd side has one.
\end{remark}

Two things follow that are worth saying out loud. First, \emph{the numerator of} (L1) \emph{is
combinatorics}: a signed count of transversals of a partition of $R'$ residues into $m'+1$ classes,
with no character and no branching anywhere in it. Second, (iii) is a vanishing criterion that costs
no computation --- read the residues of $V=2(\Lambda+\rho)$ modulo $t$, and if any class $1,\dots,m'$ is
unhit, the whole numerator is zero. The lemma also explains the injective
subset-to-weight correspondence that the failed Laplace route of Remark \ref{rem:L1route} had
already found: it is the same subset/complement bookkeeping, seen from the side where it is a
statement about Weyl cosets rather than about minors.

What remains. Exactly the division, and nothing else. It is not the na\"\i ve one: reading $c$ off a
single corner of the $2^r$ shifts works in only $121$ of the $245$ cases. The shifted supports
collide in $131$ of them; in the $114$ where they do not the corner reading is automatic, and it
survives in seven more by accident. It is, however, a terminating back-substitution from the top weight downwards,
using \eqref{eq:gkrsfilt} alone and no branching: on all $245$ cases that recursion has zero
remainder, reproduces $c$ exactly, and never produces a coefficient of absolute value greater than
$1$ \stverif. So (L1) is now the single statement that dividing a $\{0,\pm1\}$ numerator by
$\prod_j(z_j^{t/2}-z_j^{-t/2})$ cannot create a coefficient of absolute value $2$. Figure
\ref{fig:division} is that statement drawn: the numerator is a handful of points and the quotient is
a large signed staircase whose finite difference it is, and the whole difficulty is that the
staircase never doubles back on itself. We do not have
that proof, and claim nothing from it --- though Proposition \ref{prop:divided} below replaces the
recursion by a closed formula and says precisely what is left. The observation that GKRS applies
here, the argument of Lemma \ref{lem:muinj} and the suggestion to attack the division with
divided-difference operators are all due to a correspondent, and we are grateful for them.
\end{remark}

\subsection{Inverting the division}\label{sec:invert}

\begin{remark}[the denominator is an Adams image]\label{rem:adams}
One structural fact about the operator that remains, recorded because it may be the handle. The
specialised denominator is
\[
\prod_{j=1}^{r}\bigl(z_j^{t/2}-z_j^{-t/2}\bigr)=\psi^t\Bigl(\prod_{j=1}^{r}
\bigl(z_j^{1/2}-z_j^{-1/2}\bigr)\Bigr),
\]
the $t$-th Adams image of the unspecialised half-spin difference of $D_r$ --- the relative
denominator of $B_r\supset D_r$. So the division in \eqref{eq:gkrsfilt} is not division by an
arbitrary polynomial: it is the inversion of the $t$-th Adams image of the relative
spinor-denominator factor that the GKRS/Dirac formalism supplies. When this remark was written that was the only description of the operator we had which was not a
product formula, and whether it made the back-substitution transparent was open. It does: the rest
of this subsection is what came of taking it seriously.
\end{remark}

\noindent
The back-substitution is not merely terminating: it has a closed form, and the form is a divided
difference.

\begin{proposition}[the division inverts in closed form]\label{prop:divided}
Index the characters of $D_r$ by the doubled shifted weights $X=2(\mu+\rho_{D_r})$, and extend any
function on dominant regular $X$ to all of $\ZZ^r$ by $W(D_r)$-antisymmetry, writing $\tilde\nu$ for
the extension of $\nu$ and $0$ on the walls; write $\tilde c$ for the same extension of $c$, so that
the displayed formulas below hold at every $X\in\ZZ^r$ and not only on the dominant chamber, and
$c$ is recovered from $\tilde c$ by straightening. Then:
\begin{enumerate}
\item[(i)] $\Delta_t=\prod_j(z_j^{t/2}-z_j^{-t/2})$ is the Weyl denominator of the root system
$A_1^{\,r}$ with roots $\alpha_j=t f_j$, and equals $\psi^t(\Delta_1)$. Its Weyl group is the group
of sign changes, which is abelian, so multiplication by $\Delta_t$ is a product of $r$ commuting
rank-one operators $D_{t,j}=T_{-te_j}-T_{+te_j}$, where $T_v$ is translation by $v$, acting on
functions by $(T_vf)(X)=f(X+v)$.
\item[(ii)] In these coordinates multiplication by $\Delta_t$ is
$a_X\mapsto\sum_{\varepsilon\in\{\pm1\}^r}(\prod_j\varepsilon_j)\,a_{X+t\varepsilon}$, and it is
inverted by
\begin{equation}\label{eq:divided}
\tilde c(X)\;=\;s_r\!\!\sum_{k\in(2\ZZ_{\ge0}+1)^r}\!\!\tilde\nu\bigl(X+tk\bigr),
\end{equation}
a sum with finitely many nonzero terms, with $s_r$ the sign of \eqref{eq:gkrsfilt}. The operator
identity itself carries no sign; $s_r$ enters only because \eqref{eq:gkrsfilt} defines $\nu$ up to
it. Being a global sign it changes no support, no fibre and no cancellation below, and we keep it
visible rather than absorbed so that \eqref{eq:divided} can be composed with \eqref{eq:gkrsfilt}
without a second convention.
\item[(iii)] Call $\{X+tk:k\ \text{odd}\}$ the \emph{fibre} of $X$: it is the set of points whose
contributions \eqref{eq:divided} collects at $X$. Consequently (L1) is equivalent to
\[
\sum_{Y\in\text{fibre of }X}\tilde\nu(Y)\ \in\ \{0,\pm1\}\qquad\text{for every }X,
\]
and in particular it holds whenever every fibre meets $\mathrm{supp}\,\tilde\nu$ at most once.
\end{enumerate}
\end{proposition}

\begin{proof}
(i) is the definition: the positive roots of the system are the $\alpha_j$, pairwise orthogonal, so
$\prod_j(1-e^{-\alpha_j})$ is its denominator and $\Delta_t$ is that product times the monomial
$\prod_j z_j^{t/2}$; the factorisation of $D_t=\prod_j D_{t,j}$ is the expansion of the product, and
$\psi^t(\Delta_1)=\Delta_t$ because the Adams operation raises each $z_j$ to the $t$-th power.

(ii) The displayed action on alternants is the expansion of the product, straightened by $W(D_r)$;
it factors as $\prod_j(T_{-te_j}-T_{+te_j})$, where $T_v$ is translation by $v$, because the sum over
$\varepsilon$ factors coordinatewise. Each factor is inverted by $\sum_{k\ge0}T_{+t(2k+1)e_j}$, since
\[
(T_{-te_j}-T_{+te_j})\sum_{k\ge0}T_{t(2k+1)e_j}
=\sum_{k\ge0}T_{2kte_j}-\sum_{k\ge1}T_{2kte_j}=\mathrm{id},
\]
the telescoping being legitimate because $\tilde\nu$ has finite support. Taking the product over $j$
gives \eqref{eq:divided}. Nothing in this last step is new --- inverting a centred difference
operator on functions of finite support is elementary --- and we claim nothing for it; what the
proposition contributes is the identification of the specialised denominator with such an operator,
and part (iii). Finiteness of the sum is the same finiteness of support; that $c$ vanishes
below the support of $\nu$ --- so that \eqref{eq:divided} defines a Laurent polynomial and not a
formal series --- is the statement that each full progression sums to zero, which is exactly the
divisibility of $\nu$ by $\Delta_t$ already known from \eqref{eq:gkrsfilt}.

(iii) is \eqref{eq:divided} read backwards: $c(X)$ \emph{is} the fibre sum, so bounding $c$ and
bounding that sum are the same statement. The clause after it is the case that costs nothing: by
Lemma \ref{lem:muinj}, $\tilde\nu$ takes values in $\{0,\pm1\}$, so a fibre meeting the support once
contributes $\pm1$ and one missing it contributes $0$. Only a fibre meeting the support at least
twice can produce $|c(X)|\ge2$, and then only if the values fail to cancel far enough.

We state (iii) as the bound and not as ``meets once, or sums to zero''. That second form is
\emph{strictly stronger} than (L1) --- a fibre with three hits of signs $+1,+1,-1$ sums to $+1$ and
so satisfies (L1) while violating it --- and it is Theorem \ref{thm:cancel}, where it belongs.
An earlier version of this paper called the two equivalent, which was simply false; the proof above
never used the stronger form.
\end{proof}

\begin{remark}[what to call it, and what not to]\label{rem:name}
We are deliberately not calling $G_t=D_t^{-1}$ the operator $\partial_{w_0}$ of \cite{HLS}. It is
related to it --- both are division by a Weyl denominator after antisymmetrisation, and \cite{HLS}
is the reference for that formalism in general --- but $\partial_{w_0}$ is defined as a composite of
isobaric divided differences, whereas $G_t$ is the inverse of a centred difference operator on
functions of finite support, and the two definitions do not coincide term by term. What is true is
that the mechanism is the Demazure/Weyl one of $A_1^{\,r}$ \emph{transported by the Adams operation}:
on the image of $\psi^t$ one has $\delta^{(t)}_\alpha(\psi^t u)=\psi^t(\delta_\alpha u)$, and the
denominator passes from $1-e^{-\alpha}$ to $1-e^{-t\alpha}$, which is where the step $2t$ of the
progressions comes from. A fair descriptive name is therefore \emph{the Adams-transformed, or
$t$-dilated, Demazure/Weyl operator}, and \eqref{eq:divided} is its coefficientwise inverse; that
name is ours and is not a bibliographic attribution. We have not found a settled term for it in the
literature, and would rather leave the description than borrow a label.

One caveat belongs with the name. Calling $\{tf_j\}$ a root system is geometrically accurate, but as
a root datum on the \emph{same} character lattice it needs care: for $t>1$ the operator
$\delta_{t\alpha}$ is not an integral endomorphism of the whole character ring, because divisibility
by $1-e^{-t\alpha}$ is automatic only on the appropriate congruence sublattice --- in particular on
the image of $\psi^t$ --- and not for an arbitrary monomial of the original lattice. In our
coordinates this is the same fact as the finiteness clause in Proposition \ref{prop:divided}(ii):
the progression sum is defined pointwise for any finitely supported $\tilde\nu$, but it has finite
support precisely when each full progression sums to zero, that is, precisely when $\nu$ is
divisible. This entire framing, including the objection to the loose use of $\partial_{w_0}$ and the
lattice caveat, we owe to a correspondent.

The same correspondent located the continuation of \cite{HLS} that its introduction announces.
Landweber and Sjamaar \cite{LS13} take $H\subset G$ closed of \emph{maximal rank} and a $G$-space
$X$, study $K^*_H(X)$ as a module over $K^*_G(X)$, prove a relative duality --- a nonsingular
bilinear pairing --- and construct multiplets analogous to those of \cite{GKRS}, with induction
developed through twisted $\mathrm{Spin}^c$ structures. It is the maximal-rank theory the earlier
paper promised, and it is exactly the setting of our pair $B_{R'}\supset B_{m'}\times D_r$. Whether
our coefficient-level inverse $G_t$ is the specialisation of their relative twisted induction after
the torsion projection is a question we can state but not answer; we record it as Problem
\ref{prob:LS}.
\end{remark}

\begin{remark}[our numerator is a GKRS multiplet, and where their theorem stops]\label{rem:LSgap}
More than the setting matches. Landweber and Sjamaar write the relative antisymmetriser in the
opposite order, $J^{\mathrm{op}}_M=\sum_{w\in W^H}\det(w)\,w^{-1}$, write $\mathsf{D}_M$ for the
twisted $\mathrm{Spin}^c$ Dirac operator of $M=G/H$, prove $J_G=J_HJ^{\mathrm{op}}_M$
and deduce the operator identity
\begin{equation}\label{eq:LS44}
j_H^*\bigl(e(\mathsf{D}_M)^*\bigr)\,\partial_G=\partial_H\circ J^{\mathrm{op}}_M
\end{equation}
\cite[(4.4)]{LS13}; and for the multiplet $a_w=\partial_H(w^{-1}a)$ they prove
$\sum_{w\in W^H}\det(w)f^*(a_w)=0$ \cite[Thm.~4.2.2]{LS13}, where $f^*$ forgets equivariance. Their
proof is one line: $f^*j_H^*(e(\mathsf{D}_M)^*)=\prod_{\alpha\in R^+_M}(1-1)=0$.

That vanishing is visible in our data. Forgetting equivariance means evaluating a character at the
identity, so their theorem predicts $\sum_\mu\nu(\Lambda,\mu)\dim\chi^{D_r}_\mu=0$, and it holds in
$256$ of $256$ shapes, on supports of size $2$ to $8$ and with dimensions up to $3\,640$; the decoy
that shuffles the signs of $\nu$ on the same support gives $0$ in only $63$ \stverif. So our
$\nu$ \emph{is} a GKRS multiplet in their sense, and \eqref{eq:LS44} is the identity our
factorisation should be read against: the sign sits outside and the filter inside a single relative
operator, which is why no one factor of ours flips uniformly.

But their theorem is not our conjecture, and it is worth being exact about the distance. Three
statements, of which the first is theirs, the second is ours, and the third is the naive hybrid.
Call a fibre \emph{multi-hit} when it meets the support twice or more; (b) is read, as Theorem
\ref{thm:cancel} states it, at a dominant regular coefficient index:
\[
\begin{array}{lll}
\text{(a)} & \textstyle\sum_{\text{all }\mu}\nu(\Lambda,\mu)\dim\chi^{D_r}_\mu=0
 & \text{\cite[Thm.~4.2.2]{LS13}; $256/256$ here}\\[2pt]
\text{(b)} & \textstyle\sum_{\mu\in\text{fibre}}\nu(\Lambda,\mu)=0\ \text{on each multi-hit fibre}
 & \text{Theorem \ref{thm:cancel}; $333/333$ here}\\[2pt]
\text{(c)} & \textstyle\sum_{\mu\in\text{fibre}}\nu(\Lambda,\mu)\dim\chi^{D_r}_\mu=0
 & \text{false: $90$ of $333$ \stverif}
\end{array}
\]
Neither (a) nor (b) implies the other, and (c) --- the fibrewise version of their weighting --- is
simply not true, which is the cleanest evidence that (b) is a genuine refinement and not a corollary
waiting to be extracted. What would settle it is a coefficientwise refinement of \eqref{eq:LS44} for
the $\psi^t$-dilated operator; that is Problem \ref{prob:LS}, and the suggestion is a
correspondent's.
\end{remark}

\noindent
Formula \eqref{eq:divided} agrees with the terminating back-substitution on all $256$ shapes at
$(t,r)=(3,2),(5,2),(7,2),(3,3)$, multiplying back by $\Delta_t$ returns $\nu$ in $256$ of $256$, and
the largest coefficient produced is $1$; the decoy using the even progression $k\in2\ZZ_{>0}$ agrees
in $0$ of $256$, the decoy dropping the straightening sign in $57$, and the decoy keeping only the
term $k=(1,\dots,1)$ in $102$ \stverif. That last number is the same phenomenon as the failure of the
corner reading: the surviving term is $k=(1,\dots,1)$ for only $320$ of the $1\,016$ progressions
that carry one, and $25$ distinct $k$ occur.

\subsection{What cancels, and what does not explain it}\label{sec:whatcancels}

\begin{remark}[what \eqref{eq:divided} moves, and what it does not]\label{rem:divided}
The gain is a change of category. Before, (L1) was a statement about an algorithm --- a recursion
with zero remainder --- and one could only say that it had not misbehaved. Now it is a statement
about a signed count along an arithmetic progression of step $2t$ in a set described by Proposition
\ref{prop:transversal}: no characters, no branching, no division.

What the measurement says about the remaining half is sharper than we expected, and we record it
because it is where a proof would have to bite. Over the same $256$ shapes there are $1\,349$
progressions meeting $\mathrm{supp}\,\tilde\nu$: $1\,016$ meet it once, $322$ meet it twice and $11$
meet it four times, and \emph{every} one of the latter $333$ sums to exactly zero. Never three
terms; never a partial cancellation \stverif. So the missing statement is not an inequality but an
identity: a progression that meets the support more than once contributes nothing.

Three mechanisms for that identity were tested and are false or insufficient. The contributing $k$
do not form a product of independent binary choices --- a box --- failing in $51$ of the $333$, the
smallest witness being $\{(1,5),(3,3)\}$ at $t=3$, which moves diagonally. An involution exchanging
one index between the transversal $S$ and its complement reverses the sign in only $166$ of $266$
pairs, and the contributing set is connected by such exchanges in only $197$ of $333$. Pure
$W(D_r)$-antisymmetry --- the progression striking the same canonical point twice with opposite
straightening signs --- is genuine but accounts for only $43$ of the $333$ \stverif. The other $290$
pair two genuinely distinct support points, and in every witness we examined the two differ by
exactly $2t$ in a single coordinate with opposite values of $\nu$; the smallest is $t=3$,
$\Lambda=(2,1,0)$, $X=(6,-4)$, with $\tilde\nu(9,-1)=-1$ and $\tilde\nu(9,5)=+1$.

A shift by $2t$ reversing a sign looks like an affine reflection at level $t$, and $\tau^B_t$ is a
fusion projection (Theorem \ref{thm:fusion}), so the natural guess is that the identity is affine
folding. That guess is wrong for a reason worth recording, which a correspondent supplied in two
lines: in the affine Weyl group of $A_1$ with walls at $0$ and $t$ one has $s_0(x)=-x$ and
$s_1(x)=2t-x$, so $s_1s_0(x)=x+2t$. The translation by $2t$ is a product of \emph{two} reflections;
its affine sign is $+1$. The visible displacement therefore cannot itself supply the $-1$, and any
affine mechanism has to act on a richer object than the weight --- on the pair $(w,k)$ --- with the
sign coming from the finite straightening. We also tested the folding directly in both places it
could live. On the free block it fails
pointwise: $\tilde\nu(X+2te_j)=-\tilde\nu(X)$ is false off the support --- in the same witness
$\tilde\nu(9,11)=0$ and not $-1$. On the frozen block, where the folding really does live --- $\delta(A)$
\emph{is} the determinant of the element of $W(B_{m'})\ltimes t\ZZ^{m'}$ returning $A$ to the alcove
--- it explains only part of the cancellation. Writing the contribution of a term as
$\det(w)\cdot\tau^B_t(\eta_w)$, with $\det(w)$ the honest Weyl determinant (the straightening sign
times the shuffle sign, each of which separately depends on a convention) and with
$\tau^B_t(\eta_w)=\delta(A_w)$ the folding sign of its frozen block, every cancelling pair reverses \emph{exactly one} of the two factors --- $313$ of
$313$ --- but it is $\det$ in $251$ of them and $\delta$ in only $62$ \stverif. A \emph{cancelling
pair} here is an unordered pair of terms of one fibre landing on two \emph{distinct} support points
whose contributions are opposite, and the count is not one of those above: $281$ of the $322$
two-term fibres give one such pair each --- the other $41$ have both terms on the same canonical
point --- and the $11$ four-term fibres give $32$ between them, two, three or four each, so
$281+32=313$ \stverif. Affine folding is
therefore a part of the mechanism and not the mechanism. A useful by-product of the same
measurement: in the $322$ progressions with two terms the pair always cancels --- $281$ of $281$,
the other $41$ being those whose two terms land on the same canonical point, where the cancellation
is the $W(D_r)$-antisymmetry already accounted for --- and every non-cancelling pair lives in one of
the $11$ progressions with four terms, where the matching is a different one.

What does close, and is the same correspondent's suggestion, is a much narrower move than the
exchanges we had been testing: replace, inside a \emph{single} folded class, the chosen index by
another index of that class. Applied literally --- same class, doubled shifted exponents differing by
exactly $2t$, chirality unchanged --- it accounts for $175$ of the $313$ cancelling pairs. Coupling
the chirality to the folding sign closes it completely:
\[
\begin{array}{ll}
\text{same class, same folding sign, chirality unchanged} & 208\\
\text{same class, \emph{opposite} folding sign, chirality reversed} & 58\\
\text{chirality reversed alone, transversal unchanged} & 47
\end{array}
\]
which is $313$ of $313$ \stverif. The second row is the one that matters for how the paper is
organised: it pairs the two chiralities of \emph{different} transversals, so folding to $\mu^+$
before dividing would identify exactly the two points between which part of the sign lives. That is
why the numerator is kept on unfolded $D_r$ weights throughout, and only the quotient descends to
$\mu^+$. Where the $-1$ comes from is also now measured: among the pairs realised by the literal
toggle it is the Weyl determinant in $167$ of $175$ and the folding sign in only $8$ \stverif{} ---
the correspondent's prediction, that the sign would come from the representative and not from the
jump. What is still missing is that these moves assemble into a \emph{fixed-point-free
sign-reversing involution} on the incidence set
$I_X=\{(w,k):X+tk=2(\mu_w+\rho_{D_r}),\ k\ \text{odd}\}$, which is the shape the statement should
finally take.

What was wrong with our earlier attempts was not the candidate but the decomposition. Each of them
sums over the \emph{subsets} $S$ and then tries to account for the sign separately, in factors ---
straightening, shuffle, folding --- of which two depend on a convention. An element of $W^1$ is not a
subset: it is a signed permutation of $V$, distributing the $R'$ values between $m'$ frozen slots and
$r$ free ones. Summed that way the whole thing is a signed sum over perfect matchings, and a signed
sum over perfect matchings is a determinant.
\end{remark}

\subsection{The quotient is a determinant}\label{sec:det}

\begin{proposition}[the quotient is a determinant]\label{prop:determinant}
Fix $\Lambda$ and $X$, and let $M=M(\Lambda,X)$ be the $R'\times R'$ integer matrix with rows indexed
by $V_1>\dots>V_{R'}$ and columns by the slots:
\begin{itemize}
\item for $1\le p\le m'$, the $p$-th \emph{frozen} slot asks for the folded class $m'-p+1$, in the
decreasing order of Lemma \ref{lem:T}; its entry in row $i$ is the folding sign $\varepsilon$ when
$V_i$ folds to that class, and $0$ otherwise;
\item for $1\le j\le r$, the $j$-th \emph{free} slot asks for $X_j$; its entry in row $i$ is
$\sum\varepsilon$ over those $\varepsilon\in\{\pm1\}$ with $\varepsilon V_i=X_j+tk$ for some odd
$k\ge1$, and $0$ if there is none.
\end{itemize}
Then
\begin{equation}\label{eq:detsign}
\tilde c(\Lambda,X)\;=\;\kappa_{t,r}\,\epsilon_t\,\det M(\Lambda,X),
\end{equation}
for every regular $X\in\ZZ^r$, with $\tilde c$ the $W(D_r)$-antisymmetric extension of Proposition
\ref{prop:divided}; for \emph{dominant} $X=2(\mu+\rho_{D_r})$ the left-hand side is $c(\Lambda,\mu)$
itself, which is the only case Theorem \ref{thm:odd} uses. We keep the tilde wherever $X$ is not
assumed dominant, because Remark \ref{rem:domneeded} turns on exactly that difference.
with $\epsilon_t$ the constant of Corollary \ref{cor:oddsign} and $\kappa_{t,r}\in\{\pm1\}$ a sign
depending on $(t,r)$ alone. \emph{$\kappa_{t,r}$ is not $s_r$}, and the two must not be written with
the same symbol: $s_r$ normalises $\nu$ against $\Delta_t$ in \eqref{eq:gkrsfilt}, while
$\kappa_{t,r}$ normalises the determinant against the closed form of Proposition
\ref{prop:transversal}, which already carries $\epsilon_t$ and never passes through $\Delta_t$.
What the proof below gives is the \emph{existence} of such a $\kappa_{t,r}$ and its independence of
$(\Lambda,X)$; its numerical value is determined experimentally, and the measurement is below. No
part of the statement marked as proved rests on that computation.
\end{proposition}

\begin{proof}
Expand $\det M$ over permutations, and then expand each free entry by its defining sum over
$\varepsilon$ --- both signs can realise the same slot, so the elementary summands are indexed by a
permutation \emph{together with} a choice of $\varepsilon$ at each free slot, and it is those that
correspond to the data $(w,k)$. A permutation contributing a nonzero term assigns to each frozen
slot a row whose folded class is the class that slot asks for, and to each free slot a row and a sign
realising $X_j+tk$ with $k$ odd; the rows so assigned to the frozen slots form a set $S$ with
$|S|=m'$ whose folded classes are exactly $1,\dots,m'$, which is the admissibility of Proposition
\ref{prop:transversal}(iii), and the product of the frozen entries is $\prod\varepsilon$. The sign of
the permutation is the shuffle sign times the sign of the class arrangement, so the two together are
$\delta(A)$ times the shuffle, which is the coefficient of $\nu$ at the corresponding weight up to
$\epsilon_t$; the free entries record precisely which term $\tilde\nu(X+tk)$ of \eqref{eq:divided}
that weight contributes, with the straightening sign supplied by the permutation. Summing over
permutations is therefore summing \eqref{eq:divided} over $k$ and over the transversals at once.

That correspondence is bijective and sign-preserving up to one discrepancy, and it is the discrepancy
that is $\kappa_{t,r}$: the columns of $M$ are listed in a fixed order --- the frozen slots by
decreasing class, then the free slots by index --- while \eqref{eq:divided} lists the same data in
the Weyl-coset order. The two orderings differ by a permutation that depends on the two conventions
alone, hence on $(t,r)$ alone and not on $w$, $k$, $\Lambda$ or $X$. Its sign is $\kappa_{t,r}$, and
that it is a single sign independent of $(\Lambda,X)$ is what is proved here; which sign it is, is
not.
\end{proof}

\noindent
Verified over twelve configurations, $3\le t\le11$ and $2\le r\le4$, on every pair $(\Lambda,X)$ at
which $M$ has no identically zero column --- we call those the \emph{live} pairs --- the others give $\det M=0=c$, checked on a sample of
$526$ of them, and counting them would inflate the score with agreements that cost nothing. On those
$37\,330$ pairs the identity holds in $37\,330$, and $\kappa_{t,r}=+1$ in every one of the twelve
\stverif{} --- so on the range tested $c=\epsilon_t\det M$ on the nose, including the value
$\epsilon_5=-1$ that Corollary \ref{cor:oddsign} predicts and $\epsilon_9=\epsilon_{11}=+1$. That
$\kappa_{t,r}\equiv+1$ is a measurement and not a proof, and it says nothing about $s_r$, which is
$(-1)^r$: the two are normalised against different objects, as the statement of Proposition
\ref{prop:determinant} spells out. Further, $|\det M|\in\{0,1\}$ throughout; and every entry
of $M$ lies in $\{0,\pm1\}$. That last fact is \emph{not} because each slot is realised by at most one
sign: both signs realise the same free slot in $377$ of the $315\,250$ free entries examined, always
in a row whose $V_i$ is divisible by $t$, and there the two contributions cancel and the entry is $0$
\stverif. An entry $\pm2$ could not arise in any case, since the two signs enter with opposite
signs; we record the count because the absence of a $\pm2$ proves nothing about uniqueness, and an
earlier version of this sentence said that it did. Of the two decoys only the first discriminates.
Asking the free slots for an even $k$ agrees in $0$ of the $11\,012$ places where
there is anything to agree about; dropping the folding signs agrees in $11\,012$ of $11\,819$
\stverif{} --- that is, it \emph{ties} on $93\%$ of them, so by the rule of \S\ref{sec:verif} the
folding signs are untested by it there and confirmed by it nowhere. We report the number rather
than the difference for that reason, and we do not have the explanation: a decoy that could fail
would have to change which permutations survive and not only their signs, and we have not built
one. Figure
\ref{fig:determinant} draws the two readings of the same sum side by
side, on the smallest shape where the distinction between them has content.

\begin{theorem}[unfolded cancellation]\label{thm:cancel}
For every $\Lambda$ and every \emph{dominant} regular coefficient index $X=2(\mu+\rho_{D_r})$,
$X_1>\dots>X_{r-1}>|X_r|$,
\[
\sum_{k\in(2\ZZ_{\ge0}+1)^r}\tilde\nu\bigl(X+tk\bigr)=
\begin{cases}
\tilde\nu(Y) & \text{if the fibre of $X$ meets $\mathrm{supp}\,\tilde\nu$ exactly once, at $Y$},\\[2pt]
0 & \text{if it meets it not at all, or more than once.}
\end{cases}
\]
\end{theorem}

\begin{proof}
By Lemma \ref{lem:permanent}, $|I_X|=\operatorname{per}N$, a product of block permanents.
If some entry of $N$ equals $2$ then $|I_X|$ is even by Lemma \ref{lem:doubled}. Otherwise
every block of $N$ is a $0/1$ matrix, laminar by Lemma \ref{lem:laminar}, so by Lemma
\ref{lem:sdr} its permanent is odd only if it equals $1$; hence $|I_X|$ is either even or
equal to $1$. If $|I_X|=1$ the sum is its single term, and if $|I_X|=0$ it is empty. If $|I_X|$
is even then $\det M$ is even by Lemma \ref{lem:permanent} and the congruence
$N\equiv M\pmod 2$ --- a doubled entry has exactly two realisations and the two signs cancel,
so $N_{i,j}\equiv M_{i,j}$ in every entry --- and $\operatorname{per}\equiv\det\pmod2$;
with Corollary \ref{cor:unimodular} this forces $\det M=0$, and the sum vanishes by
Proposition \ref{prop:determinant}.
\end{proof}

\noindent
This is the statement (L1) should be given, and it is stronger than (L1): the bound
$|c(\Lambda,\mu)|\le1$ is the corollary obtained by combining it with $\tilde\nu\in\{0,\pm1\}$, which
is Lemma \ref{lem:muinj}. Stating it as an identity rather than as an inequality is a correspondent's
recommendation, and the data are unambiguous about which is the true shape: the multiple
intersections cancel \emph{exactly}, never partially, in $333$ of $333$ \stverif. Note also that
Theorem \ref{thm:cancel} lives on unfolded $D_r$ weights by necessity, not by convenience; the
folded statement is its corollary, obtained after the division and not before.

\begin{lemma}[$M$ is sign-equivalent to an interval matrix]\label{lem:c1p}
Let $t=2m'+1\ge3$ and let $M=M(\Lambda,X)$ be the matrix of Proposition \ref{prop:determinant}.
For each row $i$ let $c_i\in\{0,1,\dots,m'\}$ and $\sigma_i\in\{\pm1\}$ be the unique pair with
$\sigma_iV_i\equiv c_i\pmod t$, and set $w_i:=\sigma_iV_i$; for each column $j$ set $\tau_j:=+1$ if
$j$ is frozen, and otherwise let $b_j\in\{0,\dots,m'\}$ and $\tau_j\in\{\pm1\}$ be the unique pair
with $X_j\equiv\tau_jb_j\pmod t$, with $\tau_j:=+1$ when $b_j=0$. Let $\pi$ order the rows by
$(c_i,w_i)$ increasing lexicographically. Then
\[
\widehat M_{\pi(i),j}:=\sigma_i\,\tau_j\,M_{i,j}
\]
is a $0/1$ matrix whose ones are consecutive in every column. Consequently $M$ is totally
unimodular: a $0/1$ matrix with the consecutive-ones property is totally unimodular, which is
Fulkerson--Gross \cite[\S8]{FG65}, where it is deduced from the standard sufficient conditions
for total unimodularity; they add that it \emph{can also} be proved directly by induction on the
number of rows, without giving that argument there.
\end{lemma}

\begin{proof}
\emph{Step 0: the parity condition is automatic.} $V_i=2(\Lambda+\rho_{B_{R'}})_i$ is odd and
$X_j$ is even, since $X$ sits one step of $\Delta_t$ below the support of $\nu$, which is odd. Hence
$\varepsilon V_i-X_j$ is odd; as $t$ is odd, the quotient $k$ is odd whenever it is an integer. So
the free entry of Proposition \ref{prop:determinant} is
\begin{equation}\label{eq:freeentry}
M_{i,j}=\sum_{\varepsilon\in\{\pm1\}}\varepsilon\,
[\varepsilon V_i>X_j]\,[\varepsilon V_i\equiv X_j\ (\mathrm{mod}\ t)] .
\end{equation}

\emph{Step 1: folding is a coordinate.} Since $t$ is odd, $v\mapsto-v$ on $\ZZ/t$ has the single
fixed point $0$ and orbits $\{c,-c\}$, $c=1,\dots,m'$; so $(c_i,\sigma_i)$ is well defined, with
$\sigma_i=+1$ when $c_i=0$, and $V_i=\sigma_iw_i$ with $w_i\equiv c_i$ and $w_i$ odd.

\emph{Step 2: frozen columns are class blocks.} The $p$-th frozen entry is $\sigma_i$ exactly when
$c_i=m'-p+1$, so after signing the row by $\sigma_i$ it is the indicator of that class, a
consecutive block in the order $(c_i,w_i)$.

\emph{Step 3: a free column lives in one class.} The congruence in \eqref{eq:freeentry} reads
$c_i\equiv\pm b_j \pmod t$, and $0\le c_i,b_j\le m'$ with $2m'<t$ forces $c_i=b_j$. If $b_j\neq0$
then $c_i\not\equiv-c_i$, so exactly one $\varepsilon$ survives, namely $\varepsilon=\sigma_i\tau_j$.

\emph{Step 4: $b_j\neq0$ gives a half-line.} The single surviving term is
$\sigma_i\tau_j[\tau_jw_i>X_j]$, so $\widehat M_{\pi(i),j}=[\tau_jw_i>X_j]$, whose support is
$\{w_i>X_j\}$ or $\{w_i<-X_j\}$ inside the block of class $b_j$: a suffix or a prefix of that block.

\emph{Step 5: $b_j=0$ gives a nested difference, hence again a half-line.} Here only rows with
$c_i=0$ contribute, for them $\sigma_i=+1$ and $w_i=V_i>0$, and \emph{both} signs pass the
congruence, so $M_{i,j}=[V_i>X_j]-[V_i<-X_j]$. The two sets are nested: $V_i<-X_j$ forces
$X_j<-V_i<V_i$. Hence $M_{i,j}=[V_i>|X_j|]$, with no tie since $V_i$ is odd and $X_j$ even --- a
half-line again, and $\tau_j=+1$.

\emph{Step 6.} $\widehat M$ is a $0/1$ matrix with the consecutive-ones property in columns, i.e.\ an
interval matrix, hence totally unimodular; and $M$ is obtained from $\widehat M$ by reversing the
row permutation and the row and column signings, all of which preserve total unimodularity. The hat
is not decoration here: $N$ denotes the unsigned companion of Lemma \ref{lem:permanent}, a different
matrix, which may carry an entry $2$.
\end{proof}

\begin{remark}[what the circuit does \emph{not} do]\label{rem:nocircuit}
The blocks of $M$ carry a graphic matroid, and explicitly so --- Proposition \ref{prop:graph} below
builds the graph --- so it is tempting to read Theorem
\ref{thm:cancel} as the classical cancellation of a determinant expansion along its first cycle
\cite{Zeilberger85,KasselLevy}. That reading does not apply, and we record the obstruction so that it is not tried
again: passing from a column to its difference vector is a row operation, which preserves the
determinant but not the population of nonzero terms, and in our blocks it leaves at most one term.
Neither the girth of the block nor the rank of its cycle space determines $|I_X|$ --- both fail on
our data. What the circuit controls is only whether the signed sum can carry an odd residue, and
that is exactly what Lemma \ref{lem:permanent} delivers by other means.
\end{remark}

\begin{corollary}[unimodularity]\label{cor:unimodular}
$M(\Lambda,X)$ is totally unimodular: every square submatrix has determinant in
$\{0,\pm1\}$. In particular $\det M\in\{0,\pm1\}$ and, by Proposition
\ref{prop:determinant} at a dominant index, $|c(\Lambda,\mu)|\le1$; that is \emph{(L1)}.
\end{corollary}

\begin{remark}[the neighbouring unimodularity, and why it is a different one]\label{rem:hi}
Hidaka--Itoh \cite{HI24} \stext{} study the Schur polynomials at \emph{all primitive} $n$th roots of
unity and reach a regular matroid there too. The route is theirs and we claim nothing for it: they
reduce the evaluation to unimodular vector systems and, in the key case, to a \emph{network matrix\nocorr};
at the matroid level they record that the relevant tensor product is cographic, of a complete
bipartite graph --- a reading they are explicit about being \emph{weaker} than their own theorem,
which asserts the unimodularity of the system itself. For the structural background they cite
Danilov--Grishukhin \cite{DG99}, who describe a maximal unimodular system as an amalgam of
\emph{components} --- the graphic root systems $A_n$ on one side and cographic systems attached to
non-planar $3$-connected cubic graphs on the other --- which is the classification the next sentence
appeals to. The two matroids are not the same one,
and the difference is structural rather than one of range: theirs is \emph{cographic}, of $K_{m,n}$,
and
ours is \emph{graphic}: block by block it is the cycle matroid of an explicit multigraph, obtained by
sending each interval column $\mathbf1_{[a,b]}$ to the oriented edge $e_a-e_{b+1}$ of Proposition
\ref{prop:graph}. A cographic matroid is graphic only when the
graph is planar, and $K_{m,n}$ with $m,n\ge3$ is not planar, so the two occupy different
graphic/cographic regimes of regular-matroid theory. (We do not claim either system is a
\emph{component} in \cite{DG99}'s technical sense; that is their word for the indecomposable pieces
of a maximal unimodular system, and we have not checked it for ours.) There is also a number that separates
them: their theorem assumes that $n$ has at most two distinct odd prime factors, and therefore does
not cover $n=105=3\cdot5\cdot7$ --- the smallest $n$ it excludes, not a point at which it stops,
since it still covers arbitrarily large $n$ with at most two --- whereas Lemma \ref{lem:c1p} asks
nothing of $t$ and holds at $t=105$ \stverif.
\end{remark}

\begin{proposition}[the block graph]\label{prop:graph}
Fix $\Lambda$ and $X$ with $t$ odd, and put $M(\Lambda,X)$ in the interval form of Lemma
\ref{lem:c1p}. Let $B$ be one of its blocks, with rows indexed by $1,\dots,n$, and let
\[
\partial:\ZZ^{n}\longrightarrow\ZZ^{n+1},\qquad
\partial(x_1,\dots,x_n)=\bigl(x_1,\;x_2-x_1,\;\dots,\;x_n-x_{n-1},\;-x_n\bigr).
\]
Then $\partial$ is injective and identifies $\ZZ^n$ with the root lattice
\[
A_n=\Bigl\{y\in\ZZ^{n+1}:\textstyle\sum_i y_i=0\Bigr\},
\]
the inverse being $x_k=y_1+\dots+y_k$; and on the indicator of an interval it gives
\begin{equation}\label{eq:boundary}
\partial\,\mathbf1_{[a,b]}=e_a-e_{b+1}.
\end{equation}
Write $\alpha:=1$, $\omega:=n+1$ and $v_2,\dots,v_n$ for the remaining coordinates. By Lemma
\ref{lem:laminar} every column of $B$ is a proper prefix $[1,k]$, a proper suffix $[k+1,n]$ or the
full interval $[1,n]$, so \eqref{eq:boundary} sends them to the edges $\alpha v_{k+1}$,
$v_{k+1}\omega$ and $\alpha\omega$ respectively, while a zero column becomes a loop. So the columns
of $B$ are the oriented edges of a multigraph $G_B$ on $\{\alpha,v_2,\dots,v_n,\omega\}$ whose
non-loop part satisfies
\[
G_B\setminus\{\text{loops}\}\ \subseteq\ K_{2,n-1}+\alpha\omega ,
\]
with parallel edges allowed; the loops are exactly the zero columns. The column matroid of $B$ is the
cycle matroid of $G_B$; in particular
\[
B\ \text{has linearly independent columns}\iff G_B\ \text{is a forest.}
\]
\end{proposition}

\begin{proof}
The coordinates of $\partial x$ telescope to $0$, so $\operatorname{im}\partial\subseteq A_n$;
$x_1=y_1$ and $x_k=x_{k-1}+y_k$ recover $x$ from $y$, which gives both injectivity and surjectivity
onto $A_n$. Identity \eqref{eq:boundary} is immediate. Since $\partial$ is injective,
$\ker B=\ker(\partial B)$, and $\partial B$ is the oriented incidence matrix of $G_B$ with one vertex
per coordinate; the kernel of an oriented incidence matrix is the cycle space of its graph, which is
the classical statement that a graphic matroid's circuits are the graph's cycles. Hence the columns
of $B$ are dependent exactly when they contain a cycle. \emph{We do not assume $B$ square}: the
blocks of $M$ need not be, and only $M$ itself is --- see Corollary \ref{cor:forest}.
\end{proof}

\begin{remark}[which circuits occur, and which ones matter]\label{rem:cycles}
The graph makes the classification of the singular blocks a single statement instead of a list. A
loop is a zero column; a pair of parallel edges is a repeated column; the triangle
$\alpha\,v_{k+1}\,\omega\,\alpha$ is the relation $\mathbf1=\mathbf1_{[1,k]}+\mathbf1_{[k+1,n]}$
between the full column and a complementary prefix--suffix pair. The remaining short cycle,
$\alpha\,v_{k+1}\,\omega\,v_{l+1}\,\alpha$ with $k\ne l$ --- two \emph{different} complementary
pairs, with no full column involved --- \emph{may} occur, and when it does it is a genuine circuit:
those four edges are a minimal dependence among themselves. What it is not is a new \emph{criterion},
and that is what closes the list. In a block of folded class $0$ there are no prefixes at all by
Lemma \ref{lem:laminar}(ii), so no complementary pair exists and the four-cycle cannot occur; in a
block of class $c\ge1$ the frozen column is present, so every two-edge path $\alpha\,v\,\omega$
already lies in a triangle, and the block is singular before the four-cycle is looked at. There is nothing
longer to check: the circuits of $G_B$ have sizes $1$, $2$, $3$ or $4$. Sizes $1$ and $2$ are the
loops and the parallel pairs, which the multigraph carries and the underlying simple graph does not;
and in the simple graph underlying the non-loop part, which sits inside $K_{2,n-1}+\alpha\omega$,
only triangles and four-cycles survive, because every path alternates between the two terminals
$\alpha,\omega$ and the interior vertices have degree $2$.
\end{remark}

\begin{corollary}[the coefficient is a forest condition]\label{cor:forest}
For every $\Lambda$ and every regular $X$,
\[
\tilde c(\Lambda,X)\neq0\iff\det M\neq0\iff\text{every }G_B\text{ is a forest},
\]
and if moreover $X$ is \emph{dominant}, $X_1>\dots>X_{r-1}>|X_r|$ --- so that
$\tilde c(\Lambda,X)=c(\Lambda,\mu)$ --- these are equivalent to $|I_X|=1$.
\end{corollary}

\begin{remark}[the dominance is not decoration]\label{rem:domneeded}
The last equivalence is the only one that needs it, and it needs it. At $t=3$, $r=2$, $m'=1$ with
$\Lambda=(4,2,0)$, so $V=2(\Lambda+\rho_{B_3})=(13,7,1)$, take the regular but \emph{not} dominant
index $X=(4,-10)$. Then
\[
M=\begin{pmatrix}1&1&0\\1&1&-1\\1&0&-1\end{pmatrix},\qquad \det M=-1,\qquad \operatorname{per}N=3 ,
\]
so every $G_B$ is a forest and $\tilde c\neq0$, while the fibre is hit three times: the sum over it is
$\pm1$ and not $0$. What fails first is not Theorem \ref{thm:cancel} but Lemma \ref{lem:laminar} ---
the columns with supports $\{1,2\}$ and $\{2,3\}$ overlap without nesting, which is exactly the
prefix--suffix overlap dominance forbids. Swept over three $\Lambda$ at $t=3,5,7$ there are ten such
$X$ off the dominant chamber and none on it \stverif.
\end{remark}

\begin{proof}
By \eqref{eq:detsign} and since $\kappa_{t,r}\epsilon_t$ is a unit, $\tilde c\ne0$ iff
$\det M\ne0$. As $M$
is block diagonal, $\det M\ne0$ forces every block to be square and nonsingular, so each has
independent columns and each $G_B$ is a forest by Proposition \ref{prop:graph}. Conversely, if every
$G_B$ is a forest then every block has independent columns, hence $\#\mathrm{cols}(B)\le
\#\mathrm{rows}(B)$ for each $B$; since $M$ is square, $\sum_B\#\mathrm{rows}(B)=
\sum_B\#\mathrm{cols}(B)$, so equality holds block by block and every $B$ is square and nonsingular,
giving $\det M\ne0$. For the last equivalence, Corollary \ref{cor:unimodular} gives
$\det M\in\{0,\pm1\}$, so $\det M\ne0$ iff $\det M$ is odd; by Lemma \ref{lem:permanent} and
$N\equiv M\pmod 2$ this holds iff $\operatorname{per}N=|I_X|$ is odd, and by Theorem
\ref{thm:cancel} an odd fibre count equals $1$.
\end{proof}

\begin{lemma}[the fibre is a permanent]\label{lem:permanent}
Let $N=N(\Lambda,X)$ be obtained from $M(\Lambda,X)$ by replacing each entry by the \emph{number} of
its elementary realisations rather than their signed sum: $N_{i,p}=[\,c_i=m'-p+1\,]$ for a frozen
slot, and $N_{i,m'+j}=\#\{\varepsilon\in\{\pm1\}:\varepsilon V_i=X_j+tk,\ k\ \text{odd}\ \ge1\}$ for
a free one. Then
\[
|I_X|\;=\;\operatorname{per}N(\Lambda,X).
\]
\end{lemma}

\begin{proof}
Call an \emph{elementary summand} a pair $(\sigma,\varphi)$ with $\sigma$ a bijection rows
$\to$ slots all of whose entries are nonzero and $\varphi$ a choice, at each free slot, of one of the
realisations of that entry. By definition of the permanent,
$\operatorname{per}N=\sum_\sigma\prod_iN_{i,\sigma(i)}$ counts exactly those pairs. The proof of
Proposition \ref{prop:determinant} already puts them in bijection with the data $(w,k)$; we record
the map and check it is bijective as a map of sets. Given $(\sigma,\varphi)$, let $i_j$ be the row of
the $j$-th free slot and $\varepsilon_j$ the chosen sign, and set
$k_j:=(\varepsilon_jV_{i_j}-X_j)/t$, odd and $\ge1$. Nonvanishing of the frozen entries says the
remaining rows meet each nonzero folded class once and avoid class $0$, i.e.\ they form a
transversal, so $\tilde\nu$ does not vanish at $X+tk$ by Proposition \ref{prop:transversal}(i), and
$(w,k)\in I_X$ with $w$ unique by Lemma \ref{lem:muinj}. The map is injective: the $V_i$ are strictly
positive and pairwise distinct, so the value $\varepsilon V_i$ determines $(i,\varepsilon)$, whence
$\sigma$ on the free slots and $\varphi$; and each remaining row folds to exactly one class, so it
can only go to the one frozen slot asking for it. It is surjective: if $\tilde\nu(X+tk)\neq0$ then
the coordinates of $X+tk$ have $r$ pairwise distinct absolute values among the $V_i$ --- an
alternant with a repeated absolute value vanishes --- carried by a transversal, and the transversal
admits exactly one bijective assignment to the frozen slots.
\end{proof}

\begin{lemma}[what a column is, and that the blocks are laminar]\label{lem:laminar}
In the notation of Lemma \ref{lem:c1p}, let $X$ be a \emph{dominant} regular coefficient index,
$X=2(\mu+\rho_{D_r})$ with $X_1>\dots>X_{r-1}>|X_r|$, and order the rows of the class-$c$ block by $w_i$
increasing.
\begin{enumerate}
\item[\rm(i)] For $c\ge1$ the frozen column of that class is the all-ones column, and a free column
is the suffix $\{w_i>u_j\}$ if $\tau_j=+1$ and the prefix $\{w_i<u_j\}$ if $\tau_j=-1$, where
$u_j:=\tau_jX_j$.
\item[\rm(ii)] For $c=0$ every nonzero column is a suffix.
\item[\rm(iii)] Every block is \emph{laminar}: any two columns are nested or disjoint.
\end{enumerate}
\end{lemma}

\begin{proof}
(i) All rows of the block fold to $c$, so the frozen entry is $\sigma_i$ in each of them and becomes
$1$ after signing. For a free column, Step 3 of Lemma \ref{lem:c1p} leaves the single sign
$\varepsilon=\sigma_i\tau_j$, and Step 0 makes the parity condition vacuous, so the entry is nonzero
exactly when $k=\tau_j(w_i-u_j)/t\ge1$, which is $w_i>u_j$ for $\tau_j=+1$ and $w_i<u_j$ for
$\tau_j=-1$. (ii) is Step 5 of Lemma \ref{lem:c1p}. For (iii), two prefixes are nested and so are two
suffixes, and the all-ones column contains everything; so the only possible crossing is a prefix
$\{w<u^-\}$ against a suffix $\{w>u^+\}$, and these are disjoint precisely when $u^-\le u^+$, that
is when $X_{j^-}+X_{j^+}\ge0$. Now $X$ is enumerated dominantly, $X_1>\dots>X_{r-1}>|X_r|$, so at
most one coordinate is negative and for $m\le r-1$ one has $X_m\ge X_{r-1}>|X_r|\ge-X_r$; hence
$X_a+X_b>0$ for any two distinct indices.
\end{proof}

\begin{lemma}[laminar systems of distinct representatives]\label{lem:sdr}
Let a laminar system consist of $a$ copies of the full interval, $p$ nested proper prefixes
$[1,k_1]\subseteq\dots\subseteq[1,k_p]$ and $s$ nested proper suffixes of lengths
$m_1\le\dots\le m_s$, on $n=a+p+s$ points, the prefixes and suffixes being disjoint. Then the number
of systems of distinct representatives is
\[
a!\;\prod_{i=1}^{p}(k_i-i+1)\;\prod_{l=1}^{s}(m_l-l+1),
\]
and it is odd only if it equals $1$. \emph{The product is not new. For a family satisfying Hall's
condition, the number of systems of distinct representatives is bounded below by $r!$ with $r$ the
size of the smallest set (M.~Hall \cite[Theorem 2]{Hall48}), and $\prod_k(|F_k|-k+1)$ for a family
ordered by increasing size is the standard sharpening of that bound. What the laminar hypothesis buys
is that here the bound is an \emph{equality}, and what the paper uses is the parity that follows from
it.}
\end{lemma}

\begin{proof}
The prefixes take points below the suffixes, so the two greedy counts are independent and the $a$
full columns take the $a$ remaining points in $a!$ orders; each nested chain contributes the
classical greedy product. For the parity, an odd value forces $a\le1$. Write $K=k_p$ and
$M_s=m_s$; then $p\le K$, $s\le M_s$ and $K+M_s\le n$. If $a=0$ then $p+s=n$ forces $p=K$ and
$s=M_s$; put $d_i=k_i-i\ge0$, so $d_p=K-p=0$, and monotonicity of $k$ gives $d_{i+1}\ge d_i-1$: the
sequence decreases by at most one at a time, so if some $d_{i_0}\ge1$ it must take the value $1$
somewhere between $i_0$ and $p$, and the corresponding factor is $2$. Hence all $d_i=0$ and the
product is $1$; likewise for the suffixes. If $a=1$ then $p+s=n-1\le K+M_s\le n$; when $K+M_s=n-1$
the same argument applies, and when $K+M_s=n$ either $p=K-1$ or $s=M_s-1$, and then $d_p=1$
respectively $m_s-s=1$, giving an even factor. Finally $a\ge2$ makes $a!$ even.
\end{proof}

\begin{lemma}[doubled entries force an even fibre]\label{lem:doubled}
If some entry of $N(\Lambda,X)$ equals $2$, then $|I_X|$ is even.
\end{lemma}

\begin{proof}
By Step 3 of Lemma \ref{lem:c1p} an entry admits both signs only in class $0$, and there
$N_{i,j}=[V_i>X_j]+[V_i<-X_j]$, whose second term needs $X_j<0$. Dominance leaves at most one
negative coordinate, so at most one column of the block carries doubled entries, and that column is
$\mathbf1+P$ with $P$ the prefix $\{V_i<|X_r|\}$. By multilinearity of the permanent in the columns,
$\operatorname{per}N_{\text{block}}=A+B$ with $A$ the block with that column replaced by $\mathbf1$
and $B$ with it replaced by $P$; both are $0/1$, and both are laminar because $X_j>|X_r|$ makes $P$
disjoint from every suffix. Let $p=|P|$. The other $n-1$ columns are suffixes contained in the $n-p$
rows above $|X_r|$. If $p\ge2$ they cannot be given distinct representatives there, so $A=B=0$. If
$p=1$ they occupy exactly the rows $2,\dots,n$, so in $A$ the all-ones column is forced onto row $1$
and in $B$ the column $P$ is; the two counts coincide, $A=B$, and
$\operatorname{per}N_{\text{block}}=2A$. In both cases the block permanent is even, hence so is
$|I_X|=\prod_{\text{blocks}}\operatorname{per}N_{\text{block}}$ by Lemma \ref{lem:permanent}.
\end{proof}

\noindent
Corollary \ref{cor:unimodular} implies (L1) at once, since $\det M$ is itself such a minor, and it
is the last ingredient of Theorem \ref{thm:cancel}. The order is worth naming because it is the
reverse of the expected one. Unimodularity alone does \emph{not} give the identity: a fibre with
three hits summing to $+1$ would satisfy the bound and violate it. What it gives is the vanishing
half on every fibre carrying an \emph{even} number of hits, where a sum of an even number of
$\pm1$'s is even and $|c|\le1$ forces $0$. So everything turns on the parity of $|I_X|$, and that
parity is not a property of $M$: Lemma \ref{lem:permanent} identifies $|I_X|$ with a permanent,
Lemmas \ref{lem:laminar} and \ref{lem:sdr} make it odd only when it is $1$, and what forces the
laminarity is the dominance with which the $X$ are enumerated. This
is the third form (L1) has taken in this paper --- a bound on branching multiplicities, then a
statement that a back-substitution never doubles, now a unimodularity statement about an explicit
$0/{\pm}1$ matrix --- and it is the first form for which there is a body of technique to appeal to:
total unimodularity has classical characterisations, Ghouila-Houri's among them.

It is also the first form that can be tested \emph{exhaustively}, for a reason worth recording
because it is what makes the evidence more than a sample. The entries of $M$ depend on $V$ and $X$
only through residues and orderings, so the same matrix recurs across enormous numbers of pairs
$(\Lambda,X)$; and total unimodularity is invariant under $M\mapsto PMQ$ with $P,Q$ signed
permutation matrices, since every minor of $PMQ$ is $\pm$ a minor of $M$. Quotienting by both --- by
equality of matrices, then by signed row and column permutations --- collapses the problem
drastically. Over the twelve configurations of \S\ref{sec:verif}, $37\,330$ live pairs
$(\Lambda,X)$ give $9\,561$ distinct matrices; and calling two matrices the same \emph{pattern}
when one is $PMQ$ for the other, those fall into between $16$ and $88$ patterns per configuration.
That last count is reported for ten of the twelve: the canonical form costs $R'!\,2^{R'}$ and is not
computed at $(t,r)=(9,2)$ and $(11,2)$, where $R'=6$ and $7$. Nothing below rests on it --- total
unimodularity is exhausted over the $9\,561$ \emph{distinct} matrices in all twelve, and the
patterns only say how few tests that reduces to.
(We say \emph{pattern} and not \emph{class} because \emph{class} already means a folded residue
class throughout this section.) Every one of the $9\,561$ is totally unimodular, checked over all of its minors and not sampled
\stverif. Scaling residues by a unit permutes the folded classes with signs, which is a signed
permutation of the frozen columns, so the Galois action of Proposition \ref{prop:galoissign} is of
the kind total unimodularity cannot see; we have not checked that it carries the set of matrices
occurring here to itself, and nothing above depends on its doing so.

This is the evidence the unimodularity conjecture was originally formulated from; Lemma
\ref{lem:c1p} now proves it, so the statement is a theorem and not a conjecture. The exhaustive
computation is kept as an independent check of the proof and of its implementation --- what it
verifies is that no matrix occurring in those twelve configurations fails, which is now a
consequence rather than the evidence for it.

\begin{remark}[the nonzero-determinant content of (L1) at $t=3$, $r=2$]\label{rem:threematrices}
The reduction is concrete enough to be worth exhibiting on the smallest case. At $t=3$ and $r=2$ the
matrices are $3\times3$ --- one frozen slot, two free ones --- and the $1\,191$ live pairs
$(\Lambda,X)$ with $\beta_i\le6$ fall into sixteen patterns. Thirteen of them are singular. The whole
content of (L1) there is carried by the other three, which up to signed row and column permutations
are
\[
M_1=\begin{pmatrix}-1&-1&-1\\-1&-1&0\\-1&0&0\end{pmatrix},\qquad
M_2=\begin{pmatrix}-1&-1&0\\-1&0&-1\\-1&0&0\end{pmatrix},\qquad
M_3=\begin{pmatrix}-1&-1&0\\-1&0&0\\0&0&-1\end{pmatrix},
\]
each with $|\det|=1$. (Only the absolute value is an invariant of the class: an odd signed
permutation reverses the sign.) They occur in $56$, $56$ and $351$ of the live pairs respectively,
so in $463$ altogether --- exactly the number of $X$ at which $c\neq0$ in that box, as it must be.
Witnesses: $M_1$ at $\Lambda=(1,1,0)$, $X=(4,-2)$; $M_2$ at $\Lambda=(1,1,0)$, $X=(4,2)$; $M_3$ at
$\Lambda=0$, $X=(2,0)$, all three with $c=+1$. $M_1$ is the full triangular staircase, the textbook
totally unimodular matrix; $M_3$, the commonest and the one containing the trivial weight, is a
staircase with one detached block.

One caveat about the thirteen we set aside. They carry no content for (L1), which asks only about
$\det M$ itself, but they are not irrelevant to Corollary \ref{cor:unimodular}: total unimodularity
is a statement about \emph{every} minor, and a singular matrix can perfectly well contain a
$2\times2$ minor of determinant $2$. That is why the sweep behind the conjecture tests all minors of
all $9\,561$ matrices and not only the nonsingular patterns.

Enlarging the box to $\beta_i\le12$ adds exactly one pattern and removes none, and the one it adds is
the constant matrix with every entry $-1$, of rank one and determinant $0$ \stverif. We do not read
more into the list than it says: sixteen patterns at one $(t,r)$ is not a proof, and the count grows
with $r$ --- $49$ at $(3,3)$ and $88$ at $(3,4)$. But it does say what a case analysis would have to
cover, and that is not something the earlier formulations of (L1) could say at all.
\end{remark}

\begin{figure}[htbp]
\centering
\includegraphics[width=\textwidth]{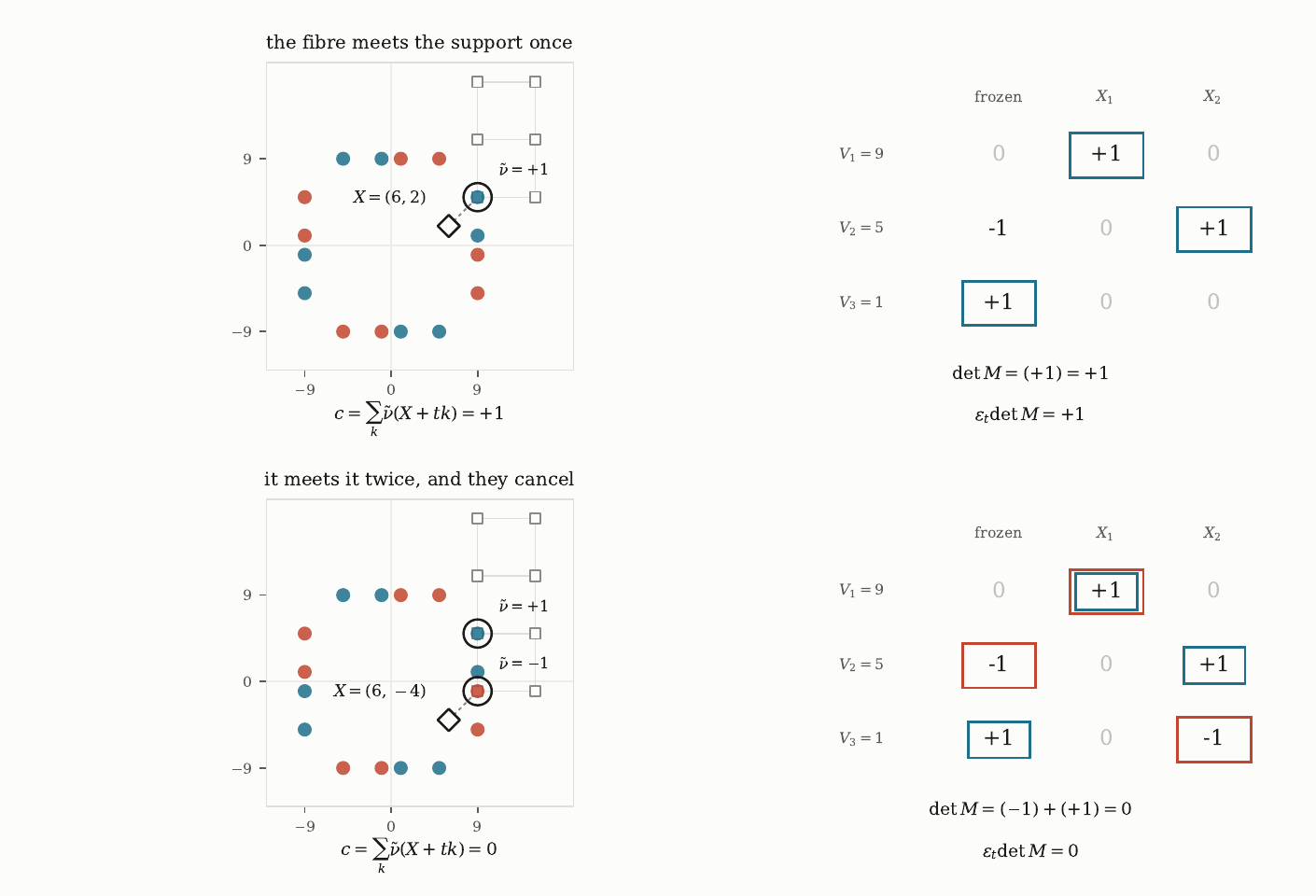}
\caption{Proposition \ref{prop:divided} and Proposition \ref{prop:determinant} are the same sum read
two ways, at $t=3$, $r=2$, $\Lambda=(2,1,0)$, so $V=(9,5,1)$ and the matrix is $3\times3$.
\textbf{Left}: the plane of doubled $D_r$ weights. Discs are $\mathrm{supp}\,\tilde\nu$, blue for
$+1$ and red for $-1$; open squares are the fibre $\{X+tk:k\text{ odd}\}$, a lattice of step $2t$
starting at $X+t(1,1)$; rings mark where the fibre meets the support. \textbf{Right}: the matrix
$M(\Lambda,X)$, with each nonzero term of its determinant expansion outlined in the colour of that
term's sign --- nested where two terms share a cell, as they do below. \textbf{Top}, $X=(6,2)$: the
fibre meets the support once, one permutation survives, and $c=\epsilon_t\det M=+1$. \textbf{Bottom},
$X=(6,-4)$: it meets the support twice, two permutations survive with opposite signs, and both
readings give $0$ --- which is Theorem \ref{thm:cancel} in the smallest case where it has
content. The script refuses to draw unless the two readings agree, the outlined terms are exactly the
nonzero ones and sum to $\det M$, the number of intersections is the one claimed, and no label
overlaps another label or a plotted point, all measured in pixels.}
\label{fig:determinant}
\end{figure}

\begin{figure}[htbp]
\centering
\includegraphics[width=\linewidth]{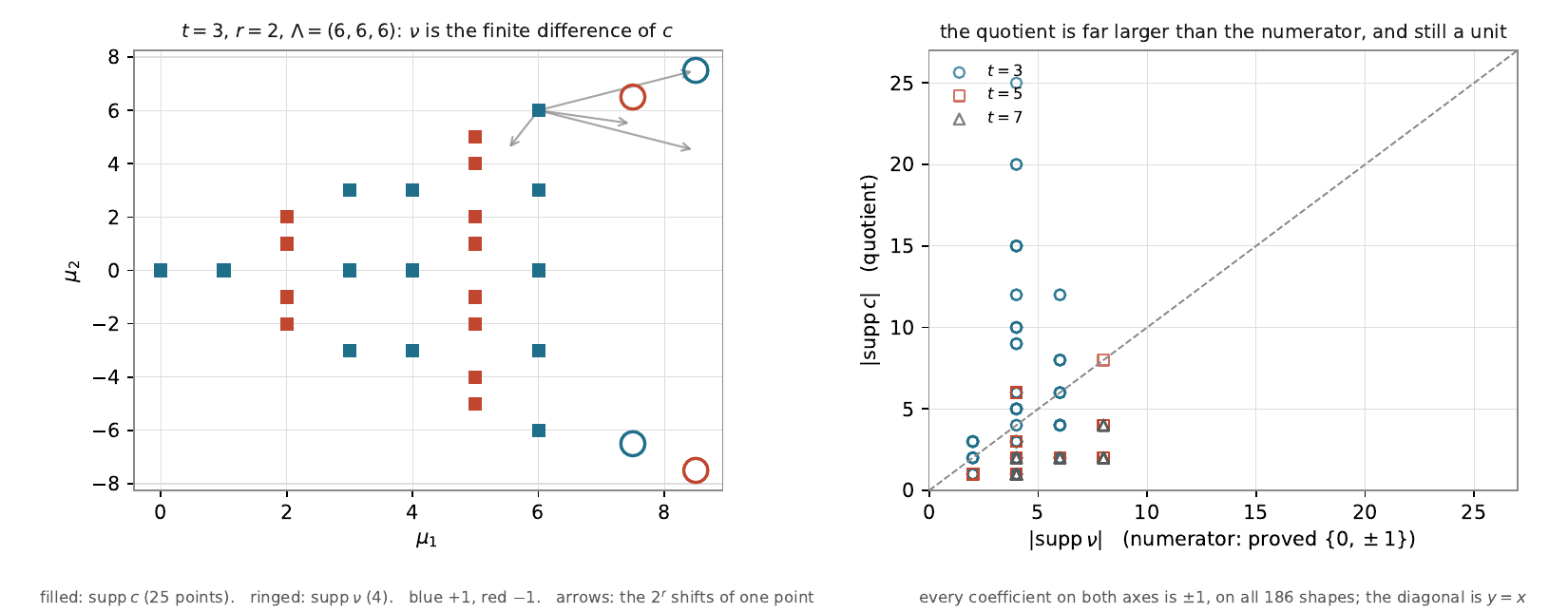}
\caption{The division of (L1), drawn --- the step that was the last difficulty until Corollary
\ref{cor:unimodular} closed it. \textbf{Left}: a single shape, $t=3$, $r=2$,
$\Lambda=(6,6,6)$, in the $D_2$ weight plane. Filled squares are $\mathrm{supp}\,c$ ($25$ points),
rings are $\mathrm{supp}\,\nu$ ($4$), blue is $+1$ and red is $-1$, and the arrows are the $2^r$
shifts of \eqref{eq:gkrsfilt} applied to one point. The numerator is the finite difference of the
quotient, and the quotient is the signed staircase one gets by summing back. \textbf{Right}: the
same comparison over $186$ shapes at $t=3,5,7$. The numerator never exceeds $8$ terms while the
quotient reaches $25$, and \emph{every} coefficient on both axes is $\pm1$; the dashed line is
$y=x$. By Lemma \ref{lem:muinj} the horizontal axis is $\{0,\pm1\}$ by theorem; the vertical axis is
(L1). Computed by \texttt{fig\_division.py} in integer arithmetic from the closed form of
$\tau^B_t$, with no branching and no character ring; it refuses to draw unless
$c\cdot\Delta_t=\nu$ exactly on every shape, every coefficient is a unit, and the drawn $c$ agrees
with the one Sage computes by branching.}
\label{fig:division}
\end{figure}

\begin{observation}[extremal primitivity, odd case]\label{obs:oddunit}
Over the full population of shapes with $\max\beta\le9$ at $(t,r)=(3,2)$ and $\max\beta\le10$ at
$(5,2)$ --- $129$ shapes, of which $6$ vanish --- every nonvanishing $\Phi_{t,r}$ has a
\emph{unique} maximal $\mu^+$ for the Weyl majorisation order of Conjecture \ref{conj:A}, and
$|A^D_{\mu^+}|=1$ in all $123$ \stverif. No shape has two maximal weights for that order. The order
matters and is not the root order: see the remark after Conjecture \ref{conj:A}.
\end{observation}

\noindent
That is evidence for the \emph{odd} half of Conjecture \ref{conj:H}, and the four populations
collected under \S\ref{sec:unit} are mostly even; it is also a far wider population than the seven
forms of \S\ref{sec:where}, which are all at $t=4$. We state it as an observation.

\section{Both filters are minimal-level fusion quotients}\label{sec:fusion}

The two filters were built separately, one proved here and one cited, and they act on different
groups. This section says what they have in common, and it is not an analogy: each is the quotient
map onto a fusion ring at the lowest level its type admits. The even fusion ring, and the
\emph{tensor sector} of the odd one, both reduce to the rank-one ring $\ZZ$ --- the qualification
matters, because the full odd ring does not. That is where the values $0,\pm1$ come from, and why
every surviving weight lands on the same point.

We use the presentation of Andersen and Stroppel \cite{AndersenStroppel} \stext. For $\mathfrak{sp}_{2n}$ at an
even order $\ell$ the alcove is $\{\sum m_i\omega_i:\sum m_i<\ell/2-n\}$, and the fusion ring is
$\ZZ[\chi(\omega_1),\dots,\chi(\omega_n)]$ modulo the ideal generated by the $\chi(k\omega_1+\omega_i)$
at level $k$; for $\mathfrak{so}_{2n+1}$ at an odd order the alcove is
$\{2m_1+\dots+2m_{n-1}+m_n\le\ell-2n\}$, with level-$1$ generators $\chi(\omega_i)$, $i<n$, and
$\chi(2\omega_n)$.

\begin{theorem}[minimal fusion realisation of the torsion filters]\label{thm:fusion}
Substituting our parameters puts both cases at the minimal level. The even alcove consists of the
single point $0$; the odd alcove is $\{0,\omega_{m'}\}$, but only $0$ lies in the tensor sector
relevant to $SO_{2m'+1}$. The parameters are:
\[
t=2m+2,\ C_m:\quad \tfrac{\ell}{2}-n=1,\ \text{level }0;
\qquad
t=2m'+1,\ B_{m'}:\quad \ell-2n=1,\ \text{level }1 .
\]
In the even case the alcove is $\{0\}$ and the fusion ideal is generated by
$\chi(\omega_1),\dots,\chi(\omega_m)$. In the odd case the alcove is $\{0,\omega_{m'}\}$, but
$\omega_{m'}=(\tfrac12,\dots,\tfrac12)$ is the spin weight and does not lie in the weight lattice of
$SO_{2m'+1}$, so the tensor character ring sees only the point $0$, and the ideal is generated
by the $\chi(\omega_i)$ with $i<m'$ together with $\chi(2\omega_{m'})$.

All of those generators vanish at the respective torsion element. Hence the evaluation homomorphisms
factor through the minimal fusion quotient, which is $\ZZ$ on a single class:
\[
\tau^C_t:\ R(\Sp_{2m})\twoheadrightarrow\mathcal F_0(C_m)\cong\ZZ,
\qquad
\tau^B_t:\ R(SO_{2m'+1})\twoheadrightarrow\mathcal F_1(B_{m'})^{\mathrm{tens}}\cong\ZZ .
\]
\end{theorem}

\begin{proof}
The generators are characters of exterior powers of the natural module, on which the torsion element
acts with eigenvalue set $\mu_t$ minus its fixed points in the even case and all of $\mu_t$ in the
odd one. The elementary symmetric functions of the spectrum are therefore explicit:
\[
\prod_{\alpha}(1+\alpha u)=\frac{1-u^{t}}{1-u^{2}}=1+u^2+\dots+u^{t-2}\ \ (t\ \text{even}),
\qquad
\prod_{\alpha}(1+\alpha u)=1+u^{t}\ \ (t\ \text{odd}),
\]
that is, $e_i=1$ for even $i$ and $0$ for odd $i$ in the even case, and $e_i=0$ for $0<i<t$ in the
odd one. In type $C_m$ the fundamental character is $\chi(\omega_i)=e_i-e_{i-2}$, which vanishes for
every $i$ because consecutive even indices carry the same value; in type $B_{m'}$ it is
$\chi(\omega_i)=e_i$ for $i<m'$ and $\chi(2\omega_{m'})=e_{m'}$, all of which vanish outright.

For the odd case we do not appeal to a minimality statement about generators --- Andersen and
Stroppel warn that at odd orders the natural generating sets need not be minimal --- and none is
needed. Spell it out instead. The monoid of dominant weights of $SO_{2m'+1}$ is generated by
$\omega_1,\dots,\omega_{m'-1}$ and $2\omega_{m'}$; hence the corresponding characters
$\chi(\omega_1),\dots,\chi(\omega_{m'-1}),\chi(2\omega_{m'})$ generate the tensor character ring, by
the usual triangularity of a product of characters against the dominance order on highest weights.
Each of them lies on a wall at $\ell=t$, so each has fusion image zero. A generating set with zero
image leaves only the constants, so the image of the tensor sector is $\ZZ[0]$.

Finally, evaluation at a group element is a ring homomorphism, so it descends to the quotient by the
ideal these generate. The values then lie in $\{0,\pm1\}$ not merely because the target is $\ZZ$ ---
a homomorphism to $\ZZ$ may take any value --- but because the folding of \cite{AndersenStroppel}
sends a singular weight to $0$ and a regular one to $(-1)^{\ell(w)}$ times the class of its alcove
representative, and here the relevant tensor sector contains exactly one alcove point.
\end{proof}

\noindent
Computed for $3\le t\le11$: every tensor generator vanishes, in both types --- and the spin weight is
excluded on the lattice, not on the value, which is the distinction the computation must respect.
Against a decoy evaluating the same characters at an element of order $t+1$, none of them
vanishes \stverif. Figure \ref{fig:alcove} draws both alcoves from their inequalities and shows the
odd one genuinely carrying two points, the second of them crossed out because it is spinorial.

\begin{figure}[!ht]
\centering
\includegraphics[width=0.94\textwidth]{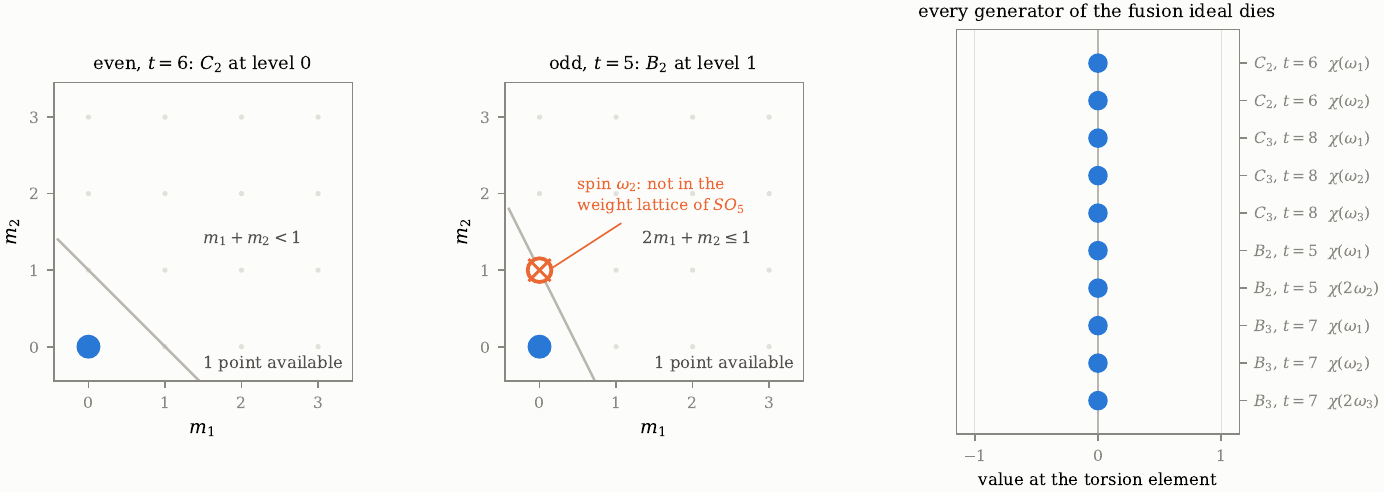}
\caption{The two minimal alcoves, drawn from the defining inequality and not from the conclusion, in
fundamental-weight coordinates at rank two. \emph{Left:} even $t$, where $\sum m_i<\ell/2-n=1$ leaves
only the origin. \emph{Centre:} odd $t$, where $2m_1+m_2\le\ell-2n=1$ leaves two points --- so the
odd fusion ring is \emph{not} rank one. The second is the spin weight $\omega_2$, whose coefficient
must be even for a weight of $SO_5$, so it is unavailable to tensor characters: drawn hollow and
crossed. That is exactly the qualification ``tensor sector'' in Theorem \ref{thm:fusion}.
\emph{Right:} the generators of the fusion ideal evaluated at the torsion element, computed from the
generating function of the spectrum; the script asserts they all vanish before drawing anything.
Computed by \texttt{fig\_alcove.py}.}
\label{fig:alcove}
\end{figure}

\begin{remark}[what this explains, and what is not ours]
The fusion ring, its presentation and the singular-dies/regular-folds mechanism are
\cite{AndersenStroppel}; we claim none of it. What is ours is the identification: that the torsion
evaluation appearing in \eqref{eq:A} and \eqref{eq:oddA} \emph{is} the quotient map at the minimal
level, for both parities. Given that, three facts recorded earlier stop being observations.
A singular weight dies and a regular one folds with a sign because that is what the quotient does;
the values lie in $\{0,\pm1\}$ because the tensor sector reaches a single alcove point and the
folding contributes only a sign; and all surviving weights fold to the same place because, for
tensor characters, there is only one place. It also
disposes of a question we had asked: in the odd case the alcove appears to have two points, and we
had proposed to measure which one each $\eta$ reaches. There is nothing to measure --- the second
point is the spin weight, and our characters are tensorial.
\end{remark}

\section{The highest surviving weight}\label{sec:top}

Let $N_\beta$ denote the numerator of \eqref{eq:def} and $\Newt(N_\beta)$ its Newton polytope.
Newton polytopes are additive under products \cite{Ostrowski} and $\Phi=N_\beta/N_\delta$, so
$\Newt(N_\beta)=\Newt(\Phi)\oplus\Newt(N_\delta)$. \emph{Whenever the three are orbit polytopes for
$W(C_r)$} --- which for $\Newt(N_\delta)$ is Proposition \ref{prop:newtden} below and for the other
two is Conjecture \ref{conj:A}, not a theorem --- a Minkowski sum of orbit polytopes with dominant
vertices is the orbit polytope of the sum, and the dominant vertex of $\Newt(\Phi)$ is then the
difference of the two dominant vertices. The
subtracted one is not a datum: it can be computed once and for all.

\begin{proposition}[the denominator's polytope, entirely]\label{prop:newtden}
Let $t\ge2$, $r\ge1$ and $N=t+2r$. Then, as a polytope in $\mathbb R^r$:
\begin{enumerate}
\item[\rm(i)] $\Newt(N_\delta)$ is the \emph{zonotope} generated by the positive root \emph{directions} of $C_r$,
those of the long roots with radius $t+1$ and those of the short roots with radius $1$ (the radius is
attached to the direction $e_k$, not to the root $2e_k$):
\[
\Newt(N_\delta)=\bigoplus_{k=1}^{r}\bigl[-(t+1)e_k,\,(t+1)e_k\bigr]
\ \oplus\bigoplus_{k<l}\Bigl(\bigl[-(e_k-e_l),\,e_k-e_l\bigr]\oplus
\bigl[-(e_k+e_l),\,e_k+e_l\bigr]\Bigr).
\]
\item[\rm(ii)] Equivalently it is the orbit polytope
$\mathrm{conv}\bigl(W(C_r)\cdot(N-1,N-3,\dots,N-2r+1)\bigr)$, and it has exactly $2^r r!$ vertices,
one for each element of $W(C_r)$.
\item[\rm(iii)] In particular $\mathrm{top}\,\Newt(N_\delta)=(N-1,\,N-3,\,\dots,\,N-2r+1)$.
\end{enumerate}
\end{proposition}

\begin{proof}
$N_\delta$ is the Vandermonde determinant of the alphabet, so $N_\delta=\pm\prod_{i<j}(x_i-x_j)$
and, by additivity \cite{Ostrowski}, $\Newt(N_\delta)$ is the Minkowski sum of the segments
$\Newt(x_i-x_j)$. Each factor is a binomial, so each segment is read off directly. The
$\binom t2$ factors $(\zeta^a-\zeta^b)$ have $z$-degree $0$ on both monomials and contribute a
point. For a fixed $k$, the $t$ factors $(\zeta^a-z_k)$ give $[0,e_k]$, the $t$ factors
$(\zeta^a-z_k^{-1})$ give $[-e_k,0]$, and $(z_k-z_k^{-1})$ gives $[-e_k,e_k]$; summing,
$[-(t+1)e_k,(t+1)e_k]$. For a pair $k<l$ the four factors $(z_k^{\pm1}-z_l^{\pm1})$ give the
segments $\mathrm{conv}\{e_l,e_k\}$, $\mathrm{conv}\{-e_l,e_k\}$, $\mathrm{conv}\{-e_k,e_l\}$ and
$\mathrm{conv}\{-e_k,-e_l\}$, whose half-vectors are $\tfrac12(e_k-e_l)$ twice and
$\tfrac12(e_k+e_l)$ twice; their sum is the stated pair of segments. That is (i).

For (ii), the generator directions are exactly the positive roots of $C_r$, so the hyperplane
arrangement dual to them is the Coxeter arrangement of $C_r$. The vertices of a zonotope are in
bijection with the chambers of that arrangement, and that bijection is $W(C_r)$-equivariant because
$W(C_r)$ permutes the generators up to sign. A Coxeter group acts \emph{simply transitively} on the
chambers of its own arrangement, so the vertex set is a single free orbit, of size
$|W(C_r)|=2^rr!$. (Stability alone would not give this: a stable set of size $|W|$ need not be one
orbit.)
It remains to name the dominant vertex, and that is a count in the direction $w=(r,r-1,\dots,1)$:
each factor contributes the exponent maximising $w$. The $t$ factors $(\zeta^a-z_k)$ contribute
$+1$ to coordinate $k$ and the $t$ factors $(\zeta^a-z_k^{-1})$ contribute $0$, since $w_k>0>-w_k$;
$(z_k-z_k^{-1})$ contributes $+1$; and for each $l\neq k$ the four cross factors contribute $+2$ to
the larger of $w_k,w_l$ and $0$ to the smaller. Summing, coordinate $k$ receives
$t+1+2(r-k)=N-2k+1$, which is (iii), and (ii) follows.
\end{proof}

\noindent
Three named sources for the dominant vertex, and none of them a $\rho$: the orbit contributes $t$,
the pair $z_k$ contributes $1$ against itself, and the other free pairs contribute $2$ each. All
three parts were checked in the $44$ configurations $2\le t\le12$, $1\le r\le4$: the zonotope equals
the orbit polytope in $44$ of $44$, the vertex count is $2^rr!$ in $44$ of $44$, and the dominant
vertex is the stated one in $44$ of $44$. Against the \emph{expanded} $N_\delta$ wherever expanding
it terminates ($N\le8$) the polytope agrees in $9$ of $9$ --- so no cancellation over the cyclotomic
field shrinks it, which is what \cite{Ostrowski} predicts and the place one would least believe it.
Three decoys: giving the long generators radius $t$ fails in $44$ of $44$, dropping them
altogether --- the roots of $D_r$ --- fails in $44$ of $44$, and giving the short generators radius
$2$ fails in all $33$ configurations with $r\ge2$, the $r=1$ cases having no short roots and so
testing nothing \stverif.

\begin{remark}[what this says about Conjecture \ref{conj:A}]\label{rem:orbitpolytope}
Conjecture \ref{conj:A} asks that $\Newt(\Phi_{t,r})$ be an orbit polytope for $W(C_r)$. Part (ii)
says that its \emph{denominator} provably is one. That does not prove the conjecture --- the
numerator is where the cancellation lives, and a quotient of orbit polytopes need not be one --- but
it does say that the conjecture asks of $\Phi$ only what the neighbouring object satisfies outright,
and it supplies the shape of an answer: a zonotope over the root system.
\end{remark}

Hence the law, whose subtracted vector depends only on $(t,r)$ and simplifies:

\begin{equation}\label{eq:mumax}
\mu_{\max}(\beta)=\mathrm{top}\,\Newt(N_\beta)-2\rho_{C_r}-(t-1)(1,\dots,1),
\end{equation}
where $2\rho_{C_r}=(2r,2r-2,\dots,2)$ --- but see Remark \ref{rem:coincidence} before reading that
decomposition structurally.

\keybox{\emph{The shift is the denominator, and nothing else.} What has to be subtracted is
$\mathrm{top}\,\Newt(N_\delta)$: the raw vector $(N-1,N-3,\dots,N-2r+1)$, which depends on $N$ and
$r$ alone. Setting $\beta=\delta$ makes $\Phi\equiv1$ and exhibits it directly.}

\noindent
Indeed, since $N=t+2r$, the $k$-th entry is $t+2(r-k)+1=(t-1)+2(r-k+1)$, which is why it can be
written as $2\rho_{C_r}+(t-1)(1,\dots,1)$. Verified in $332$ of $332$ forms, against a decoy
subtracting the other natural vector $(N-1,N-2,\dots,N-r)$, which fails \stverif.

\begin{remark}[the name is a coincidence of coordinates, and we say so because we believed otherwise]\label{rem:coincidence}
It is tempting --- we did it --- to read $2\rho_{C_r}+(t-1)(1,\dots,1)$ structurally: the sum of the
positive roots of the free \emph{symplectic} factor, plus an offset contributed by the frozen block.
That reading is wrong, and the odd case refutes it. There the free factor is $D_r$, so the
structural reading predicts $2\rho_{D_r}+(t-1)(1,\dots,1)$, which differs from
$2\rho_{C_r}+(t-1)(1,\dots,1)$ by the uniform vector $(2,\dots,2)$ and is therefore easy to separate.
Measured on the full populations: the $C_r$ shift is right in $122$ of $122$ and $303$ of $303$ even
forms \emph{and} in $82$ of $82$ and $41$ of $41$ odd ones, while the $D_r$ shift is right in none of
them \stverif. So \eqref{eq:mumax} is a statement about the Vandermonde denominator in exponent
coordinates, valid at both parities, and the coincidence with a sum of positive roots is arithmetic.
The gain is that the law transfers to the odd case unchanged, and that what remains to be proved is
only the claim that no cancellation reaches the top vertex.
\end{remark}

\begin{figure}[t]
\centering
\includegraphics[width=\textwidth]{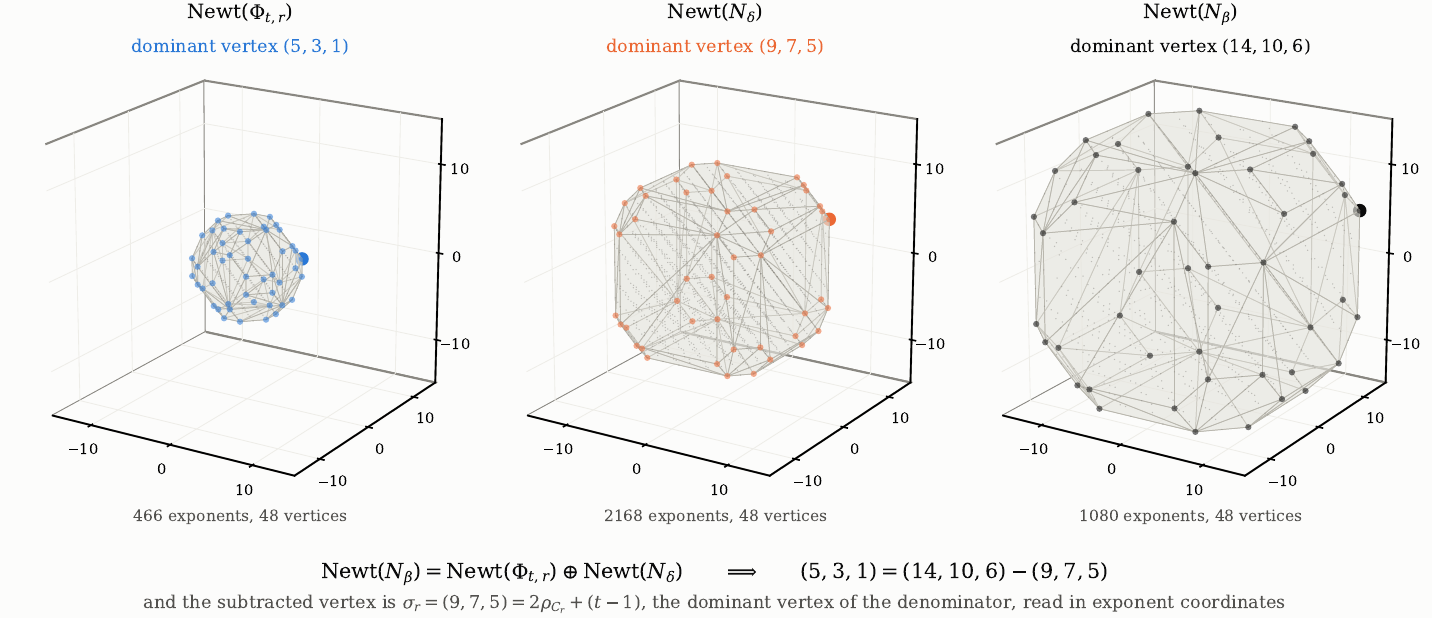}
\caption{The law \eqref{eq:mumax} as three solids and one subtraction of vertices, at $r=3$. By
Ostrowski's additivity \cite{Ostrowski} $\Newt(N_\beta)=\Newt(\Phi_{t,r})\oplus\Newt(N_\delta)$, and all three are
orbit polytopes of $W(C_r)$ --- which the figure shows rather than assumes, and shows by the right
test: for each of the three hulls the script checks that the vertex set \emph{equals}
$W(C_3)\cdot v_{\mathrm{dom}}$ as a set, not merely that it has the same cardinality $48=2^3\cdot3!$,
and it does in $3$ of $3$. Being orbit polytopes, each is determined by
its dominant vertex, and the Minkowski identity descends to the subtraction printed underneath. The
hulls are computed from the exponents that actually occur, not from the law; drawing
$\mathrm{conv}(W(C_r)\mu_{\max})$ instead would be drawing the statement rather than the data.}
\label{fig:law3d}
\end{figure}

\begin{remark}
Equation \eqref{eq:mumax} is not a closed form in $\beta$: evaluating it still requires knowing which
vertex of $\Newt(N_\beta)$ survives cancellation. What it buys is that the denominator is out of the
way and the shift has a name. It is also not yet a theorem --- it rests on Conjecture \ref{conj:A}
below, and Remark \ref{rem:AB} says exactly how. Figure \ref{fig:law3d} shows the three polytopes
involved on one shape, and the subtraction of vertices that the shift performs; the vertex set of
each is exactly one $W(C_3)$-orbit, of size $48=2^3\cdot3!$, which is what Conjecture \ref{conj:A}
asserts in general.
\end{remark}

\section{The extremal primitivity conjecture}\label{sec:unit}

Everything above is a factorisation into parts that are separately understood. One statement is left
over, and it is the reason for this paper.

The two halves of what we measure are usually stated together, and they should not be: one of them
may be much more accessible than the other, and only the first is what \eqref{eq:mumax} rests on.

\noindent
One convention, because the next statement is the only one written for both parities at once and the
two branches expand in different bases. \emph{From here on $A_\mu$ denotes the folded coefficient
support of $\Phi_{t,r}$}: on the even branch it is the $A_\mu$ of \eqref{eq:A} in the $C_r$ basis,
and on the odd branch the $A^D_\mu$ of \eqref{eq:oddA} with the chirality folded, which is legitimate
because $A^D_\mu=A^D_{\mu^*}$ by Remark \ref{rem:chirality}. Both then live on $W(C_r)$-dominant
$\mu$, which is where $\Newt(\Phi_{t,r})$ lives too.

\begin{conjecture}[a single orbit polytope]\label{conj:A}
If $\Phi_{t,r}\neq0$ then $\{\mu:A_\mu\neq0\}$ has a unique maximal element $\mu_{\max}$ for the
\emph{Weyl majorisation} order $\mu\preceq\mu'\iff\mu\in\mathrm{conv}\bigl(W(C_r)\,\mu'\bigr)$;
equivalently, $\Newt(\Phi_{t,r})$ is the orbit polytope $\mathrm{conv}\bigl(W(C_r)\,\mu_{\max}\bigr)$.
\end{conjecture}

\noindent
Which order is meant has to be said, because with the root order the statement would be false on our
own data. The root lattice of $C_r$ is $\{x\in\ZZ^r:\sum_ix_i\ \text{even}\}$, so weights with
$|\mu|$ of different parities lie in different cosets and are incomparable; and at odd $t$ the
support meets both parities --- in $55$ of $82$ shapes at $t=3$ and $8$ of $41$ at $t=5$
(\S\ref{sec:comb}) --- so it has at least two maximal elements there. Weyl majorisation has no such
obstruction, and it is the order the polytope formulation uses. We are grateful to a correspondent
for insisting on the distinction; elsewhere in this paper ``dominance'' means the root order, and
this is the one place it does not.

\begin{conjecture}[extremal primitivity]\label{conj:H}
With $\mu_{\max}$ as in Conjecture \ref{conj:A}: the class $M_{\mu_{\max}}$ at the highest surviving
weight --- virtual on the even branch and a genuine multiplicity space on the odd one, as in
\S\ref{sec:comp} --- is \emph{primitive} in the minimal fusion quotient of Theorem
\ref{thm:fusion}, that is $q_t\bigl([M_{\mu_{\max}}]\bigr)=\pm[0]$. Explicitly, in the two parities,
\[
\sum_{\eta} B_{\eta,\mu_{\max}}\,\tau^C_t(\eta)\;=\;\pm1\quad(t\ \text{even}),
\qquad
\sum_{\Lambda,\eta} a^B_\Lambda\,B^{\mathrm{odd}}_{\Lambda;\eta,\mu^+_{\max}}\,
\tau^B_t(\eta)\;=\;\pm1\quad(t\ \text{odd}),
\]
the sign being unconstrained: the conjecture is primitivity and not positivity, and both signs occur
in our data.
\end{conjecture}

\noindent
We say \emph{primitive} rather than \emph{unit} deliberately. After Theorem \ref{thm:fusion} the
target is a rank-one ring, so ``unit'' would invite the reading that $M_{\mu_{\max}}$ is an
invertible object --- a simple current --- which is not what is claimed and not what is proved. What
is claimed is that its class is $\pm$ the generator.

The two parities are not equally far from it, and the reader who wants the shortest route should go
back to \S\ref{sec:odd}: for odd $t$ the conjecture localises to (L1)--(L3), and (L1) is now
settled entire --- Proposition \ref{prop:transversal} makes its numerator a signed count of
transversals with values in $\{0,\pm1\}$, Proposition \ref{prop:divided} performs the division in
closed form, and Corollary \ref{cor:unimodular} bounds the quotient --- so that what remains towards
Conjecture \ref{conj:H} is (L2) and (L3). Nothing of the kind is available for even $t$.

\subsection{A map of the four conjectures}\label{sec:map}
This paper states four conjectures, and they are not four problems: three of them are upstream of
one. Two further statements are listed with them and are no longer conjectures --- they are the two
shapes (L1) took, the identity that explains the bound and the sufficient condition that gives it
outright, and both are proved above. They stay on the map because what closed them is what makes the
remaining four legible: the difficulty was never in $M$, and the four that are left are the ones
stated about representation-theoretic quantities rather than about an explicit finite object. A
reader deciding where to spend effort should have that laid out rather than inferred, so here it
is.

\begingroup\small
\rowcolors{2}{hband}{white}
\renewcommand{\arraystretch}{1.25}
\begin{center}
\begin{tabular}{@{}p{0.24\linewidth}p{0.40\linewidth}p{0.28\linewidth}@{}}
\toprule
\textbf{Statement} & \textbf{What it says} & \textbf{Status}\\
\midrule
Conj.~\ref{conj:H}, extremal primitivity & the class at $\mu_{\max}$ is $\pm$ the generator of the
fusion quotient & \emph{the reason for this paper}; open at both parities\\
Conj.~\ref{conj:A}, single orbit polytope & $\Newt(\Phi_{t,r})$ is one $W(C_r)$-orbit polytope &
open; \eqref{eq:mumax} rests on it, and it is what makes $\mu_{\max}$ well defined\\
Conj.~\ref{conj:tie}, the tie cancels & when the top is attained twice, the contributions cancel &
open; with Prop.~\ref{prop:transversal}(v) it gives (L2) \emph{only} where the top is attained
once\\
Conj.~\ref{conj:L3}, extremal multiplicity one & the unique extremal $\Lambda$ has
$a^B_\Lambda=1$ & open; it is (L3) \emph{only} where the top is attained once\\
Thm.~\ref{thm:cancel}, unfolded cancellation & at a dominant index, a fibre meeting
$\mathrm{supp}\,\tilde\nu$ more than once sums to zero & \emph{proved}; it gives (L1)\\
Cor.~\ref{cor:unimodular}, unimodularity & $M(\Lambda,X)$ is totally unimodular &
\emph{proved}, Lemma \ref{lem:c1p}; it gives (L1) outright\\
\bottomrule
\end{tabular}
\end{center}
\endgroup

\noindent
The implications run one way. At odd $t$, Conjecture \ref{conj:H} localises into (L1)--(L3)
(\S\ref{sec:oddlocal}), and
\[
\text{\ref{cor:unimodular}}\ \Longrightarrow\ \text{(L1)},
\qquad
\text{\ref{thm:cancel}}\ \Longrightarrow\ \text{(L1)},
\qquad
\text{Prop.~\ref{prop:transversal}(v)}+\text{\ref{conj:tie}}\ \Longrightarrow\ \text{(L2)}^{\,\circ},
\qquad
\text{\ref{conj:L3}}=\text{(L3)}^{\,\circ},
\]
and (L1)$\,\wedge\,$(L2)$\,\wedge\,$(L3) gives Conjecture \ref{conj:H} at odd $t$. \emph{The circle
is not decoration}: $(\text{L2})^{\circ}$ and $(\text{L3})^{\circ}$ are (L2) and (L3) restricted to
the shapes where $M(\beta)$ is attained by a single $\Lambda$, and on the rest --- $27$ of our $132$
--- the two last arrows do not close. When the tied top cancels, the true $\mu^+_{\max}$ drops and
neither conjecture speaks about the layer it drops to; the counter-model is in the paragraph after
Conjecture \ref{conj:tie}. So the honest reading of the last two arrows is that they dispose of the
non-tie case and leave the tie case exactly where (L2) and (L3) already were. Conjecture
\ref{conj:A} sits to one side: it is what \eqref{eq:mumax} needs, and \ref{conj:H} refers to the
$\mu_{\max}$ it produces, but neither implies the other. At even $t$ there is no localisation at all
(Remark \ref{rem:eventransversal}), and \ref{conj:H} stands alone.

The two that are now settled were settled for the reason the map predicted: they were the only ones
stated about an explicit finite object rather than about a representation-theoretic quantity, the
only ones with an existing body of technique behind them, and the only ones we could test
exhaustively rather than sample. (L1) appeared in three shapes in this paper and they were never
three problems: (L1) is the bound, Theorem \ref{thm:cancel} is the identity that explains it, and
Corollary \ref{cor:unimodular} is the sufficient condition that arrives with a toolkit. The route
that closed them runs the other way round from the one we expected --- not from a cancellation
mechanism on the fibre, but from a permanent, its parity, and the dominance of the enumeration.

\begin{remark}[the conjecture is extremal, and that is measured]\label{rem:extremal}
Conjecture \ref{conj:H} is not the top instance of a global bound, because the global bound is false.
The tempting route to one runs as follows. Yacobi shows that in the rank-one branching
$\Sp_{2n}\downarrow\Sp_{2n-2}$ each multiplicity space is canonically an irreducible module for a
product of copies of $SL_2$ \cite{Yacobi}; for a \emph{single} irreducible $\Lambda$ that makes the
character at our element a product of values in $\{0,\pm1\}$, so the individual contribution
$\epsilon(\Lambda,\mu)$ lies in $\{0,\pm1\}$. The invalid strengthening is to expect the same bound
after summing, since our object is virtual: $A_\mu=\sum_\Lambda a_\Lambda\,\epsilon(\Lambda,\mu)$
with $a_\Lambda$ of both signs at even $t$ --- a point made precisely in \S\ref{sec:comb}. Nothing in
\cite{Yacobi} bounds that sum, and the data rule it out. We tested exactly that, in exactly the case
where the theorem applies --- $t=4$, so $m=1$ --- and it fails: over seven forms and $57$ weights
with nonzero coefficient, nine carry $|A_\mu|=2$ \stverif. The reconstruction control passed on all
seven, so the count is of the coefficients themselves and not of an artefact.

We record this because it is the difference between a conjecture and a corollary. What is bounded is
the coefficient at the highest surviving weight, and the bound already fails one step below it; so
whatever explains Conjecture \ref{conj:H} has to see the extremal weight in particular, and cannot be
a statement about multiplicity spaces in general.
\end{remark}

\begin{remark}[what rests on which]\label{rem:AB}
Equation \eqref{eq:mumax} needs Conjecture \ref{conj:A} and nothing more. Ostrowski's additivity
gives $\Newt(N_\beta)=\Newt(\Phi_{t,r})\oplus\Newt(N_\delta)$ for any $\beta$; the passage from that
to a subtraction of dominant vertices is exactly the statement that the three polytopes are orbit
polytopes with a single dominant vertex each, which for the middle one is Conjecture \ref{conj:A}.
So \eqref{eq:mumax} is not a theorem as it stands: it is a consequence of \ref{conj:A}, and the
verification table records it as measured for that reason. Conjecture \ref{conj:H} is independent of
it and is where the difficulty is concentrated.
\end{remark}

Conjecture~\ref{conj:H} is not a restatement of a computation. The individual $B_{\eta,\mu_{\max}}$
are not small. Over the sixteen forms of \S\ref{sec:top16} the largest single summand
$|B_{\eta,\mu_{\max}}\tau^C_t(\eta)|$ is $798$, and the running sums travel as far as $455$ from zero
before returning; the filter supplies only signs, and every one of the sixteen lands on $\pm1$
(Figure~\ref{fig:collapse}). Nor does it follow from
Lemma~\ref{lem:T}: $A_\mu$ is an integer combination of values in $\{0,\pm1\}$ and is a priori an
arbitrary integer. Neither the branching nor the filter, taken alone, constrains it.

\medskip\noindent
\textbf{Evidence.} Uniqueness of $\mu_{\max}$ and $|A_{\mu_{\max}}|=1$ hold in every case we have
computed \stverif: $396$ of $396$ in the symplectic expansion computed in the type-$C$ bialternant;
$332$ of $332$ in the population where \eqref{eq:mumax} was checked; $13$ of $13$ in the branching
verification of \S\ref{sec:comp}; and $16$ of $16$ in the stratum of \S\ref{sec:top16} below, where
$\mu_{\max}$ was obtained twice by unrelated machinery and agreed.

\medskip\noindent
\textbf{A remark on what a unit means here.} Every weight that survives the filter has, by the proof,
folded classes exhausting $\{1,\dots,\mathrm{rk}\}$, hence folds to the same alcove point --- and by
Theorem \ref{thm:fusion} that point is the only one there is, in either parity, on the tensor
sector. So the conjecture is not that the coefficient is small; it is that the extremal class ---
virtual for even $t$, a genuine multiplicity space for odd $t$ --- is \emph{primitive} in the
rank-one quotient, that is $q_t([M_{\mu_{\max}}])=\pm[0]$. An earlier version of this remark took the
two points of the odd alcove for an open question and proposed to measure which one each $\eta$
reaches. Theorem \ref{thm:fusion} settles
it: the second point is the spin weight, our characters are tensorial, and there is nothing to
measure.

\begin{figure}[t]
\centering
\includegraphics[width=\textwidth]{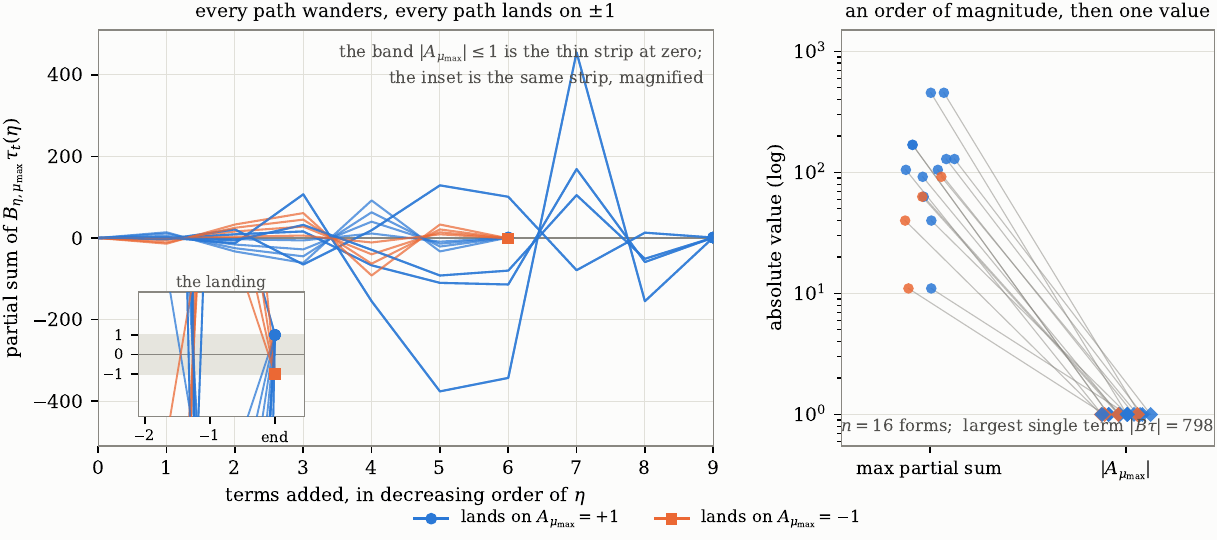}
\caption{Conjecture \ref{conj:H} seen from inside. \emph{Left:} for each of the sixteen forms of
\S\ref{sec:top16}, the running partial sum of $B_{\eta,\mu_{\max}}\tau^C_t(\eta)$ as the surviving
$\eta$ are added in decreasing order. The paths travel as far as $455$ from zero and the largest
single summand is $798$; the filter contributes only signs; and every path lands inside the band
$|A|\le1$, magnified in the inset. \emph{Right:} the same fact as two scales --- how far each sum
travels against where it ends, on a logarithmic axis. The left column spans two orders of magnitude
and the right column is a single value. Nothing in the two factors separately forces this: the
branching knows nothing of the root of unity, and the filter knows nothing of the branching.}
\label{fig:collapse}
\end{figure}

\section{The stratum of the failing candidates}\label{sec:top16}

The expansion \eqref{eq:A} also settles a question left by the naive approach. Expanding
\eqref{eq:def} by Laplace over the frozen rows produces a candidate for the top weight; the candidate
is correct except on a small set of forms, on which the whole top orbit cancels and one must descend.
At $t=6$, $r=2$ and $\beta$ bounded by $13$ there are exactly $16$ such forms, and on all $16$ the
descent is by the same vector $\Delta=(1,1)$, whereas at $t=4$ the descent takes five different
values.

Computing the branching for all $16$ shows what causes it \stverif. In every one of the $16$, the
highest $\eta$ in the block of the failing candidate \emph{survives} the filter -- no wall kills it --
and there are between $5$ and $9$ surviving $\eta$ in that block which cancel among themselves to give
$A=0$. No form has zero survivors.

\keybox{The descent $\Delta=(1,1)$ is not caused by the torsion filter. It is caused by cancellation
inside the branching, among weights that the filter lets through.}

\noindent
Two further facts from the same computation. The weight $\mu_{\max}$ obtained here from the branching
agrees in $16/16$ with the first nonvanishing level of the unrelated Laplace machinery, which is an
independent check of both. And all $16$ forms share the same pair of weights, candidate $(5,2)$ and
maximum $(4,1)$: the interesting statement is not that $\Delta=(1,1)$ but that the failures
concentrate on a single pair.

\section{The residue}\label{sec:residue}

\subsection{One constraint the locus satisfies for free}
Before measuring what is left, we record a property of the vanishing locus that costs nothing and
that we did not put there. Every step of the factorisation is a map of representation rings ---
Littlewood restriction, branching to a subgroup, and, by Theorem \ref{thm:fusion}, a fusion quotient
--- so the composite $s_\lambda\mapsto\Phi_{t,r}(\lambda)$ is a ring homomorphism, and its kernel is
an ideal.

\begin{proposition}[zeros come with their neighbours]\label{prop:ideal}
If $\Phi_{t,r}(\lambda)=0$ then $\sum_{\rho=\lambda+\square}\Phi_{t,r}(\rho)=0$. In particular
$\lambda$ cannot have exactly one non-vanishing neighbour $\lambda+\square$.
\end{proposition}

\begin{proof}
Pieri gives $s_\lambda s_{(1)}=\sum_{\rho}s_\rho$, the sum over $\rho=\lambda+\square$, and
$\Phi_{t,r}$ is multiplicative, so
$\sum_\rho\Phi_{t,r}(\rho)=\Phi_{t,r}(\lambda)\,\Phi_{t,r}((1))=0$. If exactly one summand were
nonzero the sum could not vanish.
\end{proof}

\noindent
Measured over the shapes of size at most $11$ or $12$ in four configurations: the identity holds in
all $21$, $32$, $5$ and $19$ vanishing shapes, and the count of non-vanishing neighbours is never
$1$ --- the distributions run over $\{0,2,3,4\}$ --- whereas among non-vanishing shapes the value $1$
does occur, $6$, $8$, $2$ and $3$ times respectively \stverif. The constraint is therefore about the
locus and not about the count. In two configurations there are also vanishing shapes all of whose
neighbours vanish, $7$ at $t=6$ and $3$ at $t=5$.

\subsection{How much is left}
Over the population of all strictly decreasing $\beta$ with entries at most $13$, the classical
criterion accounts for most of the vanishing.

\begin{center}
\rowcolors{2}{hband}{white}
\begin{tabular}{l r r l}
\toprule
& forms & vanishing & accounted for by \\
\midrule
$\beta$ with an empty residue class mod $t$ & $275$ & $275$ & the classical core criterion \\
occupied, vanishing already at $t=2$ & --- & $9$ & the reduction of \S\ref{sec:red} \\
occupied, vanishing by cancellation & --- & $\mathbf{12}$ & \textbf{nothing} \\
\midrule
total ($t=4$, $r=2$, $N=8$, $1716$ forms) & $1716$ & $296$ & \\
\bottomrule
\end{tabular}
\end{center}

\noindent
At $t=6$, $r=2$ ($N=10$, $715$ forms) the corresponding residue is $12$ forms out of $491$ occupied.
So the open problem, on the range we can compute, is a set of about a dozen forms per configuration:
the classical criterion sees $93\%$ of the zeros and the reduction a few more.

\subsection{What is true of the residue}
Two statements survive testing against the full non-vanishing population \stverif.

First, a parity condition: $|\lambda|\equiv N(N-1)/2 \pmod 2$, which holds in $21$ of $21$ at $N=8$
(where the residue is even) and $12$ of $12$ at $N=10$ (where it is odd). It is necessary and far
from sufficient: $718$ and $216$ non-vanishing forms satisfy it.

Second, a concentration: at $t=6$ the occupied population carries $36$ distinct $t$-cores and
\emph{$33$ of them contain no vanishing form at all}. The three that do give
$P(\text{vanish}\mid\text{core})=28.6\%$, against $10.1\%$ for the same conditioning on $|\lambda|$,
so the signal is real; but it does not decide.

\begin{figure}[t]
\centering
\includegraphics[width=\textwidth]{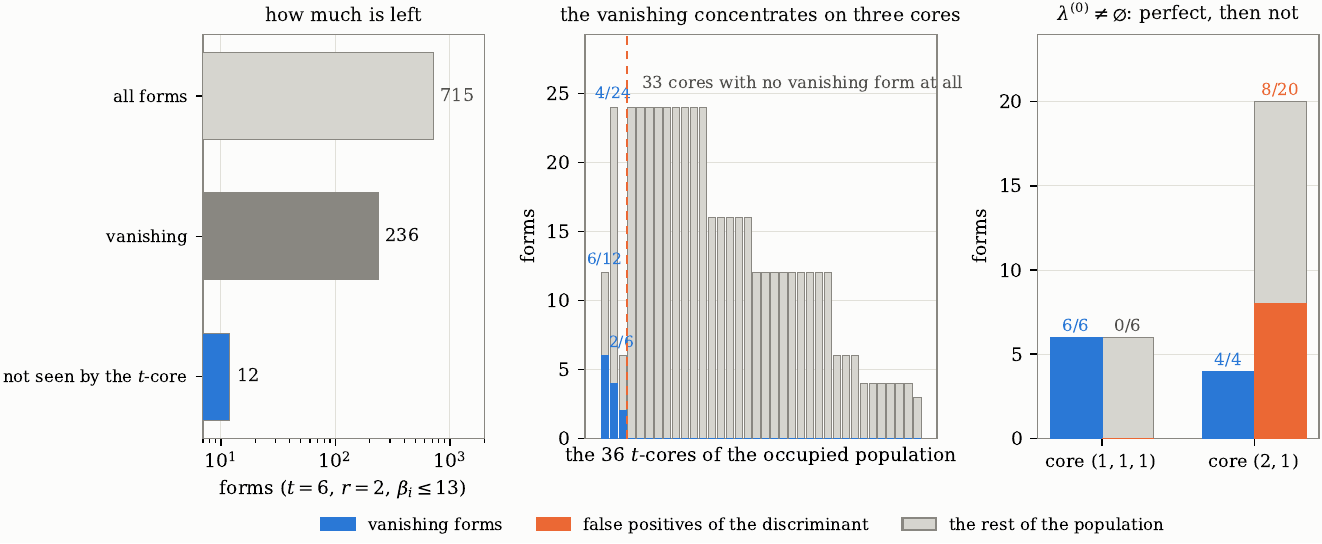}
\caption{Where the open problem is, and what does not find it. \emph{Left:} the funnel, on a
logarithmic scale. Of $715$ forms, $236$ vanish, and all but twelve of those are seen by the
classical core criterion. \emph{Centre:} the occupied population carries $36$ distinct $t$-cores and
$33$ of them contain no vanishing form at all --- a sharp restriction --- but the three that do are
far from pure, which is why the core is necessary and not sufficient. \emph{Right:} the best
discriminant we found, and the stratum where it dies. Inside the core $(1,1,1)$ the six vanishing
forms are exactly the six whose quotient component at the self-paired residue is non-empty; inside
$(2,1)$ the same condition holds for all four vanishing forms and for eight of the twenty others.
Both strata are drawn at the same scale, so the failure is as visible as the success.}
\label{fig:residue}
\end{figure}

\subsection{What is not true of it}
We list the discriminants that fail, with the false-positive count that kills each, because a
vanishing criterion is exactly the kind of statement that is easy to fit on twelve objects. Figure
\ref{fig:residue} draws the funnel: how much each layer of the criterion removes, the three cores of
thirty-six that survive it, and the discriminant that is perfect on one stratum and wrong on the
next.

\begin{center}
\rowcolors{2}{hband}{white}
\begin{tabular}{l c r l}
\toprule
candidate & covers & false $+$ & verdict \\
\midrule
$t$-core empty & $6/21$ & $156$ & no \\
$\ell(\lambda)>r+1$ & $20/21$ & $1337$ & no \\
$\lambda$ an odd full-height rectangle & $\mathbf{0/21}$ & $3$ & covers none \\
$|\lambda|$ even & $21/21$ & $718$ & necessary, not sufficient \\
balanced residue profile & $1/21$ & $51$ & no \\
$\lambda^{(0)}\neq\varnothing$ (see below) & $10/12$ & $118$ & no \\
core among the three \emph{and} $\lambda^{(0)}\neq\varnothing$ & $10/12$ & $8$ & best, still no \\
\bottomrule
\end{tabular}
\end{center}

\noindent
The last two deserve a word, since within one stratum they look decisive. Fix the core $(1,1,1)$ at
$t=6$: there are $12$ occupied forms, $6$ vanishing and $6$ not, all with the same residue profile.
The component of the $t$-quotient at the self-paired residue $q=0$ is nonempty in all six vanishing
forms and empty in all six others -- a split that has probability $1/\binom{12}{6}\approx0.1\%$ at
random. It nevertheless fails in the next stratum: at the core $(2,1)$ the same condition holds for
all four vanishing forms but also for eight of the twenty non-vanishing ones. We record it as a dead
candidate.

\begin{remark}[what would settle this]
With a residue of a dozen forms, no conjunction of two features can be resolved: there are more
degrees of freedom in the candidate rules than there are data. What is needed is a larger population,
not more features.
\end{remark}

\section{The same coefficient, computed the other way}\label{sec:square}

The composite \eqref{eq:A} processes the free pairs first and the root of unity last. The companion
paper does the opposite: it processes the orbit first, leaving a Schur function in the free alphabet.
Both compute the same $A_\mu$, so the two orders of processing must agree, and the agreement is an
identity between objects of different provenance.

\subsection{The other order}
Adjoining a set of variables splits a Schur function as $s_\lambda(X\cup Y)=\sum_\nu
s_{\lambda/\nu}(X)\,s_\nu(Y)$. With $X=\zt$ and $Y$ the free reciprocal alphabet this gives
\begin{equation}\label{eq:eps}
\Phi_{t,r}\;=\;\sum_\nu \varepsilon^{(t)}_{\lambda,\nu}\,s_\nu(z_1^{\pm1},\dots,z_r^{\pm1}),
\qquad
\varepsilon^{(t)}_{\lambda,\nu}=s_{\lambda/\nu}(\zt),
\end{equation}
and then, restricting each $GL_{2r}$-character to $\Sp_{2r}$ with multiplicities $b_{\nu\mu}$,
\begin{equation}\label{eq:square}
\boxed{\ \sum_\nu \varepsilon^{(t)}_{\lambda,\nu}\,b_{\nu\mu}
\;=\;\sum_\eta B_{\eta,\mu}\,\tau^C_t(\eta)\ }
\qquad\text{for every }\mu.
\end{equation}
On the left, a root-of-unity evaluation followed by a Littlewood restriction; on the right, a
symplectic branching followed by a torsion folding. Neither side knows about the other.

\begin{proposition}[the two factorisations agree]\label{prop:square}
Let $t$ be even. Then \eqref{eq:square} holds for every $\mu$.
\end{proposition}

\begin{proof}
Both sides are the coefficient of $\spc_\mu$ in the same element $\Phi_{t,r}$ of the representation
ring of $\Sp_{2r}$: the left from \eqref{eq:eps} followed by the Littlewood restriction, the right
from \eqref{eq:A}. The irreducible characters $\spc_\mu$ are linearly independent, so the expansion
is unique and the two coefficients coincide.
\end{proof}

\noindent
The identity is therefore not experimental, and we record the computation only as an implementation
control: five forms at $t=4$, $r=2$, each side independently reconstructing $\Phi_{t,r}$ monomial by
monomial, with a decoy using the $\varepsilon$ of a different order disagreeing on all five
\stverif. What the control tests is our two pipelines, not the statement.

This is what makes the composite of \S\ref{sec:comp} more than a bookkeeping device. It is not that
we may compose two known steps; it is that the two natural ways of splitting the same evaluation
commute.

\subsection{The signs on the left are a determinant of zeros and ones}
The coefficients $\varepsilon^{(t)}_{\lambda,\nu}$ take the values $0,\pm1$ --- a fact the companion
paper records --- and there is a reason, which also makes them instant to compute.

\begin{lemma}\label{lem:epsdet}
On the full orbit, $\sum_{k\ge0} h_k(\zt)\,z^k=\prod_{j}(1-\zeta^j z)^{-1}=(1-z^t)^{-1}$, so
\[
h_k(\zt)=1\ \text{ for }\ k\in t\ZZ_{\ge0},
\qquad
h_k(\zt)=0\ \text{ otherwise},
\]
the second clause including every $k<0$, where $h_k=0$ by convention. Hence, by the Jacobi--Trudi
formula for skew shapes,
\[
\varepsilon^{(t)}_{\lambda,\nu}
\;=\;\det\Bigl(h_{\lambda_i-\nu_j-i+j}(\zt)\Bigr)
\;=\;\det\Bigl(\mathbf 1\bigl[\ \lambda_i-\nu_j-i+j\in t\ZZ_{\ge0}\ \bigr]\Bigr),
\]
a determinant of a matrix of zeros and ones.
\end{lemma}

\noindent
The non-negativity is not decoration, and dropping it makes the identity false. Take $t=2$ and
$\lambda=\nu=(1,1,1)$, so that $\lambda/\nu=\varnothing$ and $\varepsilon^{(2)}_{\lambda,\nu}=1$. The
Jacobi--Trudi indices $\lambda_i-\nu_j-i+j=j-i$ are
$\bigl(\begin{smallmatrix}0&1&2\\-1&0&1\\-2&-1&0\end{smallmatrix}\bigr)$; with $h_k=0$ for $k<0$ the
matrix is upper triangular with unit diagonal and the determinant is $1$, whereas reading
``$t$ divides $k$'' at $k=-2$ as well makes the first and third rows equal and the determinant $0$.
We are grateful to a correspondent for the example.

The bound $\varepsilon^{(t)}_{\lambda,\nu}\in\{0,\pm1\}$ does \emph{not} follow from the shape of that
determinant: a determinant of zeros and ones is an arbitrary integer in general, and already at size
$3$ it takes the value $2$. It follows instead from the skew form of Littlewood's theorem
\cite{LR34}: $\varepsilon^{(t)}_{\lambda,\nu}=s_{\lambda/\nu}(\zt)$ is the sign of a tiling of
$\lambda/\nu$ by $t$-ribbons, and is zero when no tiling exists. Lemma~\ref{lem:epsdet} is therefore
a computational device --- it makes the coefficient instant to evaluate --- and not the source of the
bound. What the ribbon reading adds, and what we use below, is that it fixes the \emph{size} of any
move between surviving $\nu$: the step has exactly $t$ cells.

\subsection{A selection rule, and why only even \texorpdfstring{$t$}{t} has one}
Composing the two constraints carried by the factors of \eqref{eq:square} costs nothing and gives a
vanishing criterion that needs no computation at all.

\begin{corollary}[parity selection rule]\label{cor:selection}
Let $t$ be even. Then $A_\mu=0$ for every $\mu$ with $|\mu|\not\equiv|\lambda| \pmod 2$.
\end{corollary}

\begin{proof}
By the ribbon reading above, $\varepsilon^{(t)}_{\lambda,\nu}\neq0$ forces $t$ to divide
$|\lambda|-|\nu|$; and $b_{\nu\mu}\neq0$ forces $|\nu|-|\mu|$ to be even, since the Littlewood
restriction $GL_{2r}\downarrow\Sp_{2r}$ removes a partition with even columns. Hence every
contribution to $A_\mu=\sum_\nu\varepsilon^{(t)}_{\lambda,\nu}b_{\nu\mu}$ has
$|\mu|=|\lambda|-tk-2j$, and for even $t$ the right-hand side has the parity of $|\lambda|$.
\end{proof}

\noindent
For odd $t$ the same computation gives nothing, because $tk$ can be odd. That is a third way in which
the parity of $t$ decides, independent of the two of Proposition \ref{prop:parity}, and the data
separate the two cases sharply: over the full populations every even form has all its surviving
weights in one parity class, and that class is $|\lambda|$'s --- $122$ of $122$ at $t=2$, $303$ of
$303$ at $t=4$, $50$ of $50$ at $t=6$ --- whereas odd forms carry both classes at once in $55$ of
$82$ shapes at $t=3$ and $8$ of $41$ at $t=5$ \stverif.

\begin{remark}[the same rule seen from the other side]
\emph{At $t=4$}, Corollary \ref{cor:selection} and the checkerboard of Lemma \ref{lem:checker} are one
statement read through the two sides of \eqref{eq:square} --- the corollary holds for every even $t$,
the checkerboard only where $E^{(4)}$ is defined. The checkerboard says $E^{(4)}_{\Lambda,\mu}\neq0$ forces
$|\Lambda|+|\mu|$ even, which is the branching-side form; the corollary derives the same parity
constraint from the ribbon and Littlewood steps on the orbit side. That they agree is a consistency
check on the commuting square at the level of parity, and it is the only part of \eqref{eq:square}
we can verify without computing either side.
\end{remark}

\section{The branching side made combinatorial}\label{sec:comb}

\emph{For the whole of \S\ref{sec:comb} and \S\ref{sec:where} we specialise to $t=4$, equivalently
$m=1$.} The frozen factor is then $\Sp_2$, which reduces the branching step to the rank-one
restriction $\Sp_{2R}\downarrow\Sp_{2R-2}$, and lets Lemma \ref{lem:T} be applied one factor at a time. Nothing
below is claimed for $t\ge6$: there the frozen factor is $\Sp_{t-2}$ and the rank-one argument does
not reach it. The notation $E^{(4)}$ carries the specialisation in its name.

With that fixed, the right-hand side of \eqref{eq:square} can be pushed further, and the push removes
the characters altogether.

\subsection{From weights to a parity}
Write $\Phi_{2,R}=\sum_\Lambda a_\Lambda\,\spc_\Lambda$ and let $M_{\Lambda,\mu}$ be the multiplicity
space of $\mu$ in $\Lambda$ for $\Sp_{2R}\downarrow\Sp_{2R-2}$. Yacobi shows that $M_{\Lambda,\mu}$
carries a canonical irreducible action of $(SL_2)^R$ under which it is
$F_{r_1}\boxtimes\dots\boxtimes F_{r_R}$, and --- this is the point for us --- that restricting that
action to the \emph{diagonal} $SL_2\subset(SL_2)^R$ recovers the natural action of the $Sp_2$ that
centralises $\Sp_{2R-2}$ \cite{Yacobi} \stext. Our torsion element lives in that diagonal. The same
multiplicity space carries a second irreducible structure --- a twisted Yangian action, from which
Molev builds a Gelfand--Tsetlin-type basis \cite{Molev} \stext{} --- and we use neither as a basis: only
the weight structure of the first, which is what the character evaluation sees. Since the character of a
tensor product on a diagonal is the product of the characters, and since each factor is evaluated by
Lemma \ref{lem:T} at $m=1$,
\begin{equation}\label{eq:parity}
A_\mu\;=\;\sum_{\Lambda}a_\Lambda\,\epsilon(\Lambda,\mu),
\qquad
\epsilon(\Lambda,\mu)=
\begin{cases}
0 & \text{if some } r_i \text{ is odd},\\
(-1)^{\frac12\sum_i r_i} & \text{if all } r_i \text{ are even},
\end{cases}
\end{equation}
where the $r_i$ are Yacobi's, and $\mu$ must doubly interlace $\Lambda$ or the multiplicity space is
zero. Verified against the character route on seven forms \stverif, with the decoy that omits the
rearrangement disagreeing on all seven.

\begin{remark}[a precondition that is not decoration]
The double interlacing $\Lambda^+_i\ge\mu_i\ge\Lambda^+_{i+2}$ is not a side condition one may carry
loosely: dropping it makes \eqref{eq:parity} accumulate contributions from multiplicity spaces that
are zero, and the identity fails outright. We record this because it cost us a run, and because the
local form below has exactly the same domain.
\end{remark}

\subsection{The parity is local}
Yacobi's $r_i$ are defined by a global rearrangement: sort the $2R$ numbers
$\{\Lambda_1,\dots,\Lambda_R,0\}\cup\{\mu_1,\dots,\mu_{R-1}\}$ into a weakly decreasing sequence
$s_0\ge s_1\ge\dots\ge s_{2R-1}$ and put $r_i=s_{2i-2}-s_{2i-1}$ for $1\le i\le R$, so that the
first is $r_1=s_0-s_1$. Under the interlacing the
rearrangement is not needed, and the reason is four inequalities.

\begin{lemma}[the parity is local]\label{lem:local}
Suppose $\Lambda_i\ge\mu_i\ge\Lambda_{i+2}$ for $1\le i\le R-1$, with $\Lambda_{R+1}=0$. Then
\begin{equation}\label{eq:local}
r_i=\min(\Lambda_i,\mu_{i-1})-\max(\Lambda_{i+1},\mu_i),
\qquad \mu_0=+\infty,\ \mu_R=0,\ \Lambda_{R+1}=0 .
\end{equation}
\end{lemma}

\begin{proof}
It suffices to identify the sorted sequence. We claim that for $1\le i\le R-1$,
\[
\min(\Lambda_i,\mu_{i-1})\ \ge\ \max(\Lambda_{i+1},\mu_i),
\]
which unpacks into exactly four comparisons: $\Lambda_i\ge\Lambda_{i+1}$ and $\mu_{i-1}\ge\mu_i$
because both are partitions, and $\Lambda_i\ge\mu_i$ and $\mu_{i-1}\ge\Lambda_{i+1}$ because those
are the two halves of the interlacing hypothesis. Hence
\[
\Lambda_1\ \ge\ \{\Lambda_2,\mu_1\}\ \ge\ \{\Lambda_3,\mu_2\}\ \ge\ \dots\ \ge\
\{\Lambda_R,\mu_{R-1}\}\ \ge\ 0,
\]
each brace being a pair whose two members both lie between the maximum of the previous brace and the
minimum of the next. A weakly decreasing sequence is determined by its values, so
$s_0=\Lambda_1$, the pair $\{s_{2i-1},s_{2i}\}$ is $\{\max(\Lambda_{i+1},\mu_i),
\min(\Lambda_{i+1},\mu_i)\}$ for $1\le i\le R-1$, and $s_{2R-1}=0$. Reading off
$s_{2i-2}-s_{2i-1}$ gives \eqref{eq:local}, the first and last cases being
$r_1=s_0-s_1=\Lambda_1-\max(\Lambda_2,\mu_1)$ and $r_R=s_{2R-2}-s_{2R-1}=\min(\Lambda_R,\mu_{R-1})-0$.
\end{proof}

\noindent
So survival is $R$ local parity tests and the sign is a local product. The hypothesis is not a
convenience: on all $5368$ interlacing pairs in a box the two definitions agree, and on the $6732$
pairs outside they disagree \emph{every time} \stverif{} --- the interlacing is the exact boundary of
validity, which is what one expects of a statement whose proof uses it four times.

\subsection{A checkerboard}

\noindent
One convention first, so that the second statement of the lemma has a domain. The entry
$E^{(4)}_{\Lambda,\mu}$ is defined by \eqref{eq:parity} when $\mu$ doubly interlaces $\Lambda$; we
extend it by
\[
E^{(4)}_{\Lambda,\mu}:=0 \qquad\text{whenever $\Lambda$ is not a dominant weight or $\mu$ does not
doubly interlace it,}
\]
which is the usual zero convention outside the domain and is consistent with the boundary measurement
just reported. With it, $\Lambda+e_k$ may be written freely.

\begin{lemma}\label{lem:checker}
$\sum_i r_i\equiv|\Lambda|+|\mu| \pmod 2$. Consequently
$E^{(4)}_{\Lambda,\mu}:=\epsilon(\Lambda,\mu)\neq0$ forces $|\Lambda|+|\mu|$ even, and
$E^{(4)}_{\Lambda,\mu}\,E^{(4)}_{\Lambda+e_k,\mu}=0$ for every $k$, with the zero convention above
covering the $k$ for which $\Lambda+e_k$ leaves the domain.
\end{lemma}

\begin{proof}
Write $x_i=s_{2i-2}$ and $y_i=s_{2i-1}$ for $1\le i\le R$, so that $r_i=x_i-y_i$ and the pairs
$(x_i,y_i)$ exhaust $s_0,\dots,s_{2R-1}$. Those are a rearrangement of the multiset
$\{\Lambda_1,\dots,\Lambda_R,0\}\cup\{\mu_1,\dots,\mu_{R-1}\}$, so
$\sum x+\sum y=|\mu|+|\Lambda|$.
Since $\sum r_i=\sum x-\sum y$ and $\sum x-\sum y\equiv\sum x+\sum y \pmod 2$, the congruence follows;
if all $r_i$ are even the left side is even, and adding a box to $\Lambda$ changes the right side by
one.
\end{proof}

Adding a single box therefore switches the entry off. Any pairing of surviving states on this side
must move by an even number of boxes --- which is consistent with the ribbon reading of
Lemma~\ref{lem:epsdet}, where the step size is $t$. Figure \ref{fig:cone3d} draws both facts at once:
the cone cut out by the double interlacing, and inside it the surviving entries occupying alternate
planes, $252$ of $819$ lattice points.

\begin{figure}[t]
\centering
\includegraphics[width=\textwidth]{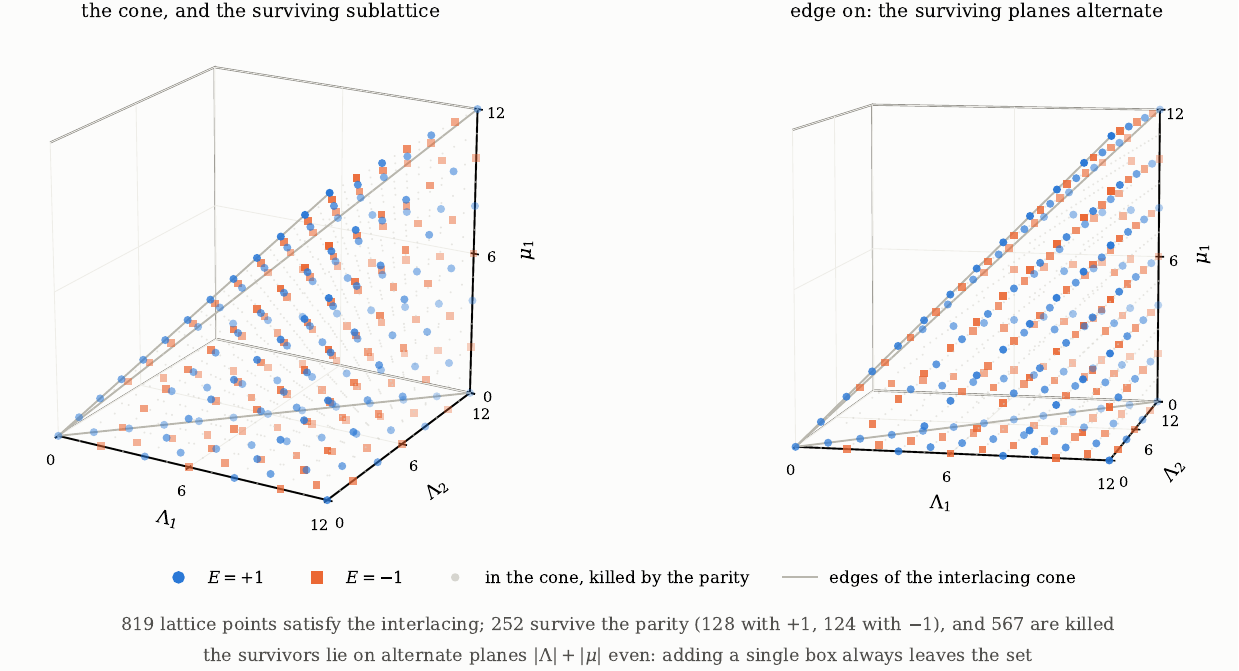}
\caption{The domain of the operator, drawn whole. At rank two a pair $(\Lambda,\mu)$ is a point of
$(\Lambda_1,\Lambda_2,\mu_1)$, so the whole domain is three-dimensional and both statements proved in
this section are visible at once: the double interlacing carves out a cone, and the parity of
Lemma~\ref{lem:checker} selects a sublattice inside it. Of the $819$ lattice points that interlace,
$252$ survive --- $128$ with $E=+1$ and $124$ with $-1$ --- and $567$ are killed by the parity; the
survivors lie on alternate planes $|\Lambda|+|\mu|$ even, which is why adding a single box always
leaves the set. \emph{Right:} the same solid seen edge on, where the alternation of the planes is
what one sees. Everything is computed from the operator.}
\label{fig:cone3d}
\end{figure}

\subsection{The operator}
Formula \eqref{eq:parity} reads $A=E^{(4)}a$, and $E^{(4)}$ can be studied once, independently of
$\beta$: the vector $a$ carries what comes from the companion paper, the matrix carries what the
torsion does. In a box of rank $R=3$ with entries at most $8$ --- a $165\times45$ matrix --- we find
\stverif: entries exactly in $\{0,\pm1\}$ ($385$ negative, $406$ positive, density $10.65\%$); the
weight is a product of local factors, so $E^{(4)}$ is a \emph{transfer matrix} along the chain and
not a generic matrix; it is triangular, with no violation of $|\mu|\le|\Lambda|$ and with a minimal
$\Lambda$ of $|E|=1$ for each of the $45$ weights $\mu$; and it is \emph{not} of M\"obius type ---
the signed row and column sums range over $-4\dots+5$ rather than concentrating on $\{0,\pm1\}$.

\section{Where the cancellation is, and where it is not}\label{sec:where}

\subsection{It is not inside a cell}
Yacobi's construction stratifies the branching semigroup into order types --- which branch the
$\min$ and the $\max$ of \eqref{eq:local} take at each position. Inside a fixed order type each $r_i$
becomes a linear form, so a cell is a much more structured place to look for a pairing than the whole
set. We segmented the survivors at $\mu_{\max}$ by cell on all seven forms, with the cell sums
reproducing $A$ in each \stverif. The cancellation is \emph{not} internal: in four of the seven forms
every cell contains a single survivor, so there is nothing to cancel within one; and the form with
eight cells has sums $+1,-2,+1,-2,-1,+2,+1,+1$, none distinguished, cancelling only when combined.
Across the seven forms, two of the twenty-six cells sum to zero --- neither of them in that one ---
against none of the eighteen cells of a decoy partition.

\subsection{The coordinates were the problem}
On the branching side the coefficient at the top weight is a genuine collective cancellation: in the
largest form, twenty-five terms with individual summands as large as $798$ and running sums reaching
$455$ from zero, all landing on $\pm1$ (Figure \ref{fig:collapse}). We then built the same number on
the other side of \eqref{eq:square}, as a signed set of atoms $(\nu,Q)$ where $\nu$ ranges over the
shapes with $\varepsilon^{(t)}_{\lambda,\nu}\neq0$ and $Q$ indexes the branching multiplicity
$b_{\nu\mu}$ --- for which a complete tableau rule is Sundaram's \cite{Sundaram} \stext{} and an
explicit bijective one, valid outside the stable range, is Watanabe's \cite{Watanabe} \stext, the two
being connected combinatorially by Azenhas \cite{Azenhas} \stext{} --- each atom weighing $\pm1$. Since
$\varepsilon$ depends on $\nu$ alone, atoms sharing a $\nu$ share a sign. The last of those
references matters for what one would do next: it moves the indexing set out of the quantum
symmetric-pair formalism and into Littlewood--Richardson--Sundaram tableaux, where a sign-reversing
involution would be a classical object rather than one inside a black box.

\keybox{\emph{The change of coordinates does not simplify the cancellation; it removes it.} At the
highest surviving weight, the seven forms carry \textbf{one} or \textbf{three} atoms. Where there is
one, $A_{\mu_{\max}}=\pm1$ is a single surviving term and nothing cancels at all. Where there are
three, the signs are two against one. The same coefficient, on the branching side, needed
twenty-five terms of size up to $798$.}

We state this as an observation and not as a theorem: seven forms is not a population, and the
pattern ``one or three, with signs $2$ against $1$'' has to survive a wider range before it is worth
a name. What it already does is relocate Conjecture \ref{conj:H}: in the coordinates of
\eqref{eq:eps} the statement to prove is about \emph{how many shapes $\nu$ reach the top weight with
a nonzero ribbon sign}, not about why large integers cancel.

\begin{remark}[the step size is not fitted]
Lemma \ref{lem:checker} says a move on the branching side cannot be a single box. On the other side
the sign is a $t$-ribbon sign, so the natural move has exactly $t$ cells. In two of the seven forms
an opposite-sign pair of shapes differs by exactly one $4$-ribbon, and a decoy testing ribbons of
other sizes fires in none of the seven \stverif. The numbers are small; the alignment of the two step
sizes is not.
\end{remark}

\section{What this settles of the companion paper}\label{sec:answers}

The companion paper closes with six problems. This one bears on three of them, and we say in each
case how far.

\subsection*{The instrument for the second pair}
That problem asks not for a proof of the vanishing at $r\ge2$ -- which the companion paper settles at $t=2$ by
an extremal argument -- but for an \emph{instrument}: a machine that computes the coefficients whose
signed sums cancel, in terms one can reason with. Equation \eqref{eq:A} is such a machine. The
coefficients are $A_\mu=\sum_\eta B_{\eta,\mu}\tau^C_t(\eta)$, with $B$ the symplectic branching
multiplicities and $\tau^C_t$ given in closed form by Lemma~\ref{lem:T}; both factors are computable
without expanding a single monomial, and the branching side has a standard monomial basis
\stext. This does not explain the cancellation, and we do not claim it does. What it does is convert
``why do these signed sums vanish'' into the single sharp statement of Conjecture~\ref{conj:H}, and
localise the difficulty in the interaction between $B$ and a function taking only the values
$0,\pm1$. On the odd side the instrument is now sharper than that: Proposition
\ref{prop:transversal} removes $B$ from the numerator altogether and leaves a signed count of
transversals of residue classes, so there the interaction to understand is not between a branching
matrix and a filter but between that count and one division.

\subsection*{An operator that sees both halves}
That problem asks whether there is an operator acting on alphabets of the form (union of orbits)
$\cup$ (reciprocal pairs) -- the two halves having, at present, two different formalisms that do not
compose. The composite of \eqref{eq:A},
\[
R_{\mathrm{virt}}(\Sp_{2R})\ \xrightarrow{\ \text{branch}\ }\
R(\Sp_{2r})\otimes R(\Sp_{t-2})\ \xrightarrow{\ 1\otimes\mathcal T_t\ }\ R(\Sp_{2r}),
\]
is one such operator for \emph{even} $t$; for odd $t$ its counterpart is \eqref{eq:oddA}, through
$B_{R'}\to B_{m'}\times D_r$. A single construction producing both remains open, and is Problem
\ref{prob:idealodd}. What both do is take the object with all pairs free and return the object with
the orbit frozen, and after Theorem \ref{thm:fusion} the second arrow has the same meaning in the two
branches: it is the minimal fusion projection on the frozen factor. We record two limits on what
that means. It is not a Verschiebung and we do not claim it factors one --- the factorisation
formalism for these evaluations is Albion's \cite{Alb23,Albion} \stext{} --- the Verschiebung
$\varphi_t$ is his --- and our composite is not one of
its operators; and it is built after the fact from a branching plus an evaluation, rather than from a
symmetric-function operation, so it answers the question in the weak sense of exhibiting a map and
not in the strong sense of making the two formalisms compose.

\subsection*{A sieving phenomenon that keeps a parameter}
Here we must be careful, because a neighbouring question has a negative answer and it is not this
one. What we have tested, and what fails, is cyclic sieving in the \emph{free reciprocal} parameter:
there the object is, up to a global sign and a balancing phase, a product of $q$-integers, and the
dressing makes the evaluations irrational away from a short list of orders and negative at $q=1$ for
a substantial fraction of shapes. The companion paper's question is the other one -- a cyclic action
in the frozen direction, refined by the free weight -- and nothing here bears on it. We state the
negative only so that the two are not confused.

\section{Verification}\label{sec:verif}

Every statement above was checked, and the ranges and counts are collected here once. Where a
statement is proved, the check is a control on the proof and on the implementation; where it is only
measured, the entry says so. Each measured statement is accompanied by the \emph{decoy} that was run
against it: a deliberately wrong version which had to fail, and whose failure count is printed. A
statement whose decoy also passes is recorded as untested, not as confirmed.

\begingroup\footnotesize\setlength{\tabcolsep}{4.0pt}
\rowcolors{2}{hband}{white}
\begin{longtable}{@{}>{\raggedright\arraybackslash}p{0.335\linewidth}%
                  >{\raggedright\arraybackslash}p{0.200\linewidth}%
                  >{\raggedleft\arraybackslash}p{0.295\linewidth}@{\hspace{5pt}}l@{}}
\toprule
statement & range & result & status\\
\midrule
\endfirsthead
\multicolumn{4}{@{}l}{\footnotesize\itshape Verification, continued}\\[1pt]
\toprule
statement & range & result & status\\
\midrule
\endhead
\midrule
\multicolumn{4}{r@{}}{\footnotesize\itshape continues on the next page}\\
\endfoot
\bottomrule
\endlastfoot

Lemma \ref{lem:T}, even $t$, both directions & $t\le12$, $\eta_1\le14,\dots,4$ &
$453$ weights, $0+0$ fail & \stproved\\
\quad the two routes agree (character vs bialternant) & same range &
$453/453$, no discrepancy & \stproved\\
\quad the sign, not only $|\tau|=1$ & same range &
$66/66$; random sign $27/66$ & \stproved\\
\quad decoy: no folding, only the $c_j$ distinct & same range &
errs on $136$ weights & \stproved\\
\quad decoy: wall at $\{0\}$ instead of $\{0,t/2\}$ & same range &
errs on $101$ weights & \stproved\\
Cor.~\ref{cor:evensize}: $\tau^C_t(\eta)\neq0\Rightarrow|\eta|$ even & $t\le12$, full support &
proved; $13,36,64,81,48$ of each; odd-branch decoy $\approx$ chance & \stproved\\
Proposition \ref{prop:oddfilter}, odd $t$, both directions & $t\le7$ &
$10/10$, $72/72$, $406/406$ & \stverif\\
\quad decoy: the same test at modulus $t+1$ & same range &
agrees only $7$, $29$, $208$ times & \stverif\\
\quad the odd filter is \emph{not} the symplectic one & $t\le9$ &
the two differ; least witness $t=5$, $\eta=(1,0)$ & \stverif\\
Proposition \ref{prop:red}, even $t$ & $t\le6$, $\beta_i\le15$ &
proved; control on $6435$ forms, $0$ counterexamples & \stproved\\
Proposition \ref{prop:red}, odd $t$ & $t=3$, $r=2$, $\beta_i\le11$ &
proved; control $260/260$, cross-$\beta$ decoy $258/259$ & \stproved\\
Equation \eqref{eq:branch}: the peeling closes & $t=4,6$; $r=2$ &
remainder $0$ in $13/13$ & \stverif\\
\quad $B_{\eta,\mu}\in\ZZ$ & same &
$13/13$ & \stverif\\
Equation \eqref{eq:A} against $\Phi_{t,r}$ apart & same &
identical monomial by monomial, $13/13$ & \stverif\\
\quad decoy: the filter removed ($\tau\equiv1$) & same &
disagrees $13/13$ & \stverif\\
\quad decoy: the $\tau$ of the wrong order & same &
disagrees $10/13$, \emph{ties in} $3$ & \stverif\\
\quad Prob.~\ref{prob:threshold}: the tie is \emph{not} a level threshold & same &
both alcoves are $\eta_1<1$: same, minimal level & \stproved\\
\quad \quad it is pointwise agreement on $\mathrm{supp}\,B$ & same &
$13/13$; ties $\Leftrightarrow$ $0$ disagreements there & \stverif\\
\quad \quad and the same three tie at order $t+4$ & same &
$3/13$, the same three & \stverif\\
Equation \eqref{eq:mumax} & $t\le8$, $r\le3$ &
$332/332$ & \stverif\\
\quad decoy: the other natural shift $(N-1,\dots,N-r)$ & same &
fails & \stverif\\
\quad the Minkowski identity on the vertices & $t=4$, $r=3$ &
$(14,10,6)=(5,3,1)+(9,7,5)$; $\mathrm{Vert}=W(C_3)v_{\mathrm{dom}}$, $3/3$ & \stverif\\
Prop.~\ref{prop:newtden}(i)--(ii): $\Newt(N_\delta)$ is the $C_r$ zonotope
& $t\le12$, $r\le4$ &
proved; $=$ orbit polytope $44/44$, $2^rr!$ vertices $44/44$ & \stproved\\
\quad against the expanded $N_\delta$ & $N\le8$ &
$9/9$: no cancellation shrinks it & \stverif\\
\quad decoys: radius $t$; no long roots; short radius $2$ & same &
$0/44$, $0/44$, $0/33$ (at $r\ge2$) & \stverif\\
Conj.~\ref{conj:A} ($\mu_{\max}$ unique) and Conj.~\ref{conj:H} ($|A|=1$) & four populations &
$396/396$, $332/332$, $16/16$, $13/13$ & \stverif\\
$|A_\mu|\le1$ away from the top: \emph{false} & $t=4$, $m=1$ &
$9$ of $57$ weights carry $|A_\mu|=2$ & \stverif\\
\S\ref{sec:top16}: the descent is $\Delta=(1,1)$ & $t=6$, $r=2$, $\beta_i\le13$ &
$16/16$ & \stverif\\
\quad $\mu_{\max}$ by branching $=$ by the Laplace route & same &
$16/16$, two unrelated machineries & \stverif\\
\quad the highest $\eta$ of the failing block survives & same &
$16/16$; $5$ to $9$ survivors cancel & \stverif\\
The residue & $t=4$, $\beta_i\le13$ &
$1716$ forms, $296$ zeros, $275+9+12$ & \stverif\\
\quad the same & $t=6$, $\beta_i\le13$ &
$715$ forms, $236$ zeros, $224+12$ & \stverif\\
How the residue vanishes: inherited / by cancellation & $t=4$; $t=6$ &
$9+12$ of $21$; $4+8$ of $12$; never by support & \stverif\\
\quad decoy: the $\tau$ of the wrong order & $t=4$ &
changes the split & \stverif\\
\quad the same decoy & $t=6$ &
\emph{ties}: same split, so untested & \stverif\\
$|\lambda|\equiv N(N-1)/2 \pmod 2$ on the residue & $t=4$ and $t=6$ &
$21/21$ and $12/12$; $718$, $216$ false $+$ & \stverif\\
The vanishing concentrates on three $t$-cores & $t=6$, $\beta_i\le13$ &
$3$ of $36$; $28.6\%$ against a decoy's $10.1\%$ & \stverif\\
$\lambda^{(0)}\neq\varnothing$, inside one core & core $(1,1,1)$ &
$6/6$ against $0/6$; $p\approx0.1\%$ by chance & \stverif\\
\quad the same, in the next core & core $(2,1)$ &
$4/4$ but $8/20$ false $+$: \emph{dead} & \stverif\\
\quad best combination found & $t=6$ &
$10/12$ with $8$ false $+$: not a criterion & \stverif\\
Prop. \ref{prop:square}: the two factorisations agree & $t=4$, $r=2$ &
proved; control $5/5$, each side rebuilds $\Phi$ & \stproved\\
\quad decoy: the $\varepsilon$ of a different order & same &
disagrees $5/5$ & \stverif\\
Lemma \ref{lem:epsdet}: $\varepsilon$ as a $0/1$ determinant & --- &
proved; a computational device only & \stproved\\
\quad the bound $\varepsilon\in\{0,\pm1\}$, which does \emph{not} follow from it & --- &
Littlewood's ribbon sign \cite{LR34} & \stext\\
\eqref{eq:parity}: the parity route against the character route & $t=4$, $r=2$ &
$7/7$ & \stverif\\
\quad decoy: the same without the rearrangement & same &
disagrees $7/7$ & \stverif\\
Lemma \ref{lem:local}: the local form of the $r_i$ & box $\le8$, $R=3$ &
proved; $5368/5368$ under interlacing & \stproved\\
\quad and outside the interlacing & same &
disagrees on $6732$ of $6732$ & \stverif\\
(L1)--(L3): the odd conjecture localises & $t\le5$, $r\le3$ &
$1735$ pairs, $132$ forms, no exception & \stverif\\
\quad but on that box every $a^B_\Lambda=1$, so (L3) is untested there & same &
$476/476$; both selection decoys $132/132$ & \stverif\\
\quad multiplicity-freeness fails one box up & $t=3$, $r=2$, $\beta_i\le13$ &
only $47\%$ of $70\,685$ are $1$; up to $23$ & \stverif\\
\quad (L3) re-measured where it can fail, i.e.\ Conj.~\ref{conj:L3} & $t=3$, $r=2$, $\beta_i\le12$ &
$678/678$; $78\%$ of shapes carry some $a^B>1$, up to $13$ & \stverif\\
\quad and again in a second configuration & $t=5$, $r=2$, $\beta_i\le12$ &
$323/323$; $52\%$ carry some $a^B>1$, up to $4$ & \stverif\\
Conj.~\ref{conj:tie}: the tie cancels & $t\le5$, $r\le3$ &
$27/27$ tied shapes; the biconditional $132/132$ & \stverif\\
\quad the subset formula that would prove (L1) & $t=3$, $r=2$ &
over-predicts the support in $66$ of $67$ & \stverif\\
Rem.~\ref{rem:gkrs}: \eqref{eq:gkrsfilt} up to the global sign $s_r$
& $(t,r)$ as below$^{\dagger}$ &
\textbf{proved}; $245/245$, decoys $0/221$ & \stproved\\
\quad the cyclotomic half of $s_r$: $\prod_j(1-z_j^t)=(-1)^r(\prod_jz_j^{t/2})\Delta_t$ & --- &
\textbf{proved}, Prop.~\ref{prop:crossden} & \stproved\\
\quad and the residual convention factor, so that $s_r=(-1)^r$ & same as above &
$+1$ throughout; measured, not derived & \stverif\\
Lem.~\ref{lem:muinj}: $w\mapsto\mu_w$ injective, hence $\nu\in\{0,\pm1\}$ & $5$ pairs $(t,r)$ &
proved; $147/147$, decoy $w\mapsto\eta_w$ $0/147$ & \stproved\\
Prop.~\ref{prop:transversal}: $\nu$ is a signed transversal count & $5$ pairs $(t,r)$ &
proved; $301/301$ support and values; decoys $68$, $24$ & \stproved\\
\quad part (v): the top is the transversal of smallest $V$ & same &
proved; $277/277$ and unique; decoys $32$, $0$ & \stproved\\
Rem.~\ref{rem:tie}: $\mu^+_{\max}$ is the largest of those tops
& $t\le5$, $r\le3$ &
right in $105$ of $132$, and \emph{iff} the top is attained once: $132/132$ & \stverif\\
Prop.~\ref{prop:divided}: the inverse formula \eqref{eq:divided} & $4$ pairs $(t,r)$ &
\textbf{proved}; $256/256$, back-multiplication $256/256$ & \stproved\\
\quad and $\max|c|=1$ on that population --- this is (L1), \emph{not} proved & same &
$256/256$; open in general & \stverif\\
\quad decoys: even progression; no straightening sign; only $k=(1,\dots,1)$ & same &
$0/256$, $57/256$, $102/256$ & \stverif\\
\quad the dichotomy of Rem.~\ref{rem:divided}: $1$ term, or sum $0$ & same &
$1016$ single; $322+11$ multiple, all summing $0$ & \stverif\\
\quad three mechanisms for it: box; exchange involution; $W(D_r)$ & same &
$282/333$, $166/266$ and $197/333$, $43/333$ & \stverif\\
\quad the affine folding of the frozen block is not it & same &
$\det$ or $\delta$ flips, never both: $313/313$; but $\delta$ only $62$ & \stverif\\
Prop.~\ref{prop:determinant}: $c=\pm\epsilon_t\det M$ & $12$ configs, $t\le11$, $r\le4$ &
proved; $37\,330/37\,330$ & \stproved\\
\quad the global sign $\kappa_{t,r}$ of \eqref{eq:detsign} & same &
$=+1$ in $12/12$, so $c=\epsilon_t\det M$ on the tested range; not proved. \emph{Not} the sign of
\eqref{eq:gkrsfilt}, which is $(-1)^r$ & \stverif\\
\quad the discarded columns really are zero & same &
$526/526$ have $\det M=0=c$ & \stverif\\
\quad decoys: even $k$; folding signs dropped & same &
$0/11012$ and $11012/11819$, scored only where scorable & \stverif\\
Cor.~\ref{cor:unimodular}: $M$ is totally unimodular & same &
$9561/9561$ distinct matrices, all minors, not sampled & \stproved\\
\quad the reduction that makes that exhaustive & same &
$37\,330$ pairs $\to$ $9561$ matrices $\to$ $16$--$88$ patterns & \stverif\\
\quad Rem.~\ref{rem:threematrices}: the three nonsingular patterns & $t=3$, $r=2$, $\beta_i\le6$ &
$16$ patterns, $13$ singular; the three carry $463$ pairs & \stverif\\
\quad \quad and the one class the larger box adds & $\beta_i\le12$ &
$17$ patterns, none lost; the new one has rank $1$ & \stverif\\
Rem.~\ref{rem:divided}: the narrow toggle, chirality coupled to the folding sign
& $4$ pairs $(t,r)$, $256$ shapes &
$313/313$; the literal $2t$ toggle alone $175/313$ & \stverif\\
\quad and there the $-1$ is the Weyl determinant, not the folding & same &
$167$ of $175$ against $8$ & \stverif\\
\quad Rem.~\ref{rem:tie} decoys: the opposite transversal; the dominance-largest $\Lambda$
& $t\le5$, $r\le3$ &
$39/132$ and $104/132$ & \stverif\\
Cor.~\ref{cor:galoisquad}: $D_t^2=(-t)^n$ and $\gamma_t$ trivial iff $n$ even
& $3\le t\le23$, both branches &
proved; $21/21$ moduli, $170/170$ pairs (unit, modulus) & \stproved\\
\quad reading $c$ off one corner of the $2^r$ shifts & $245$ GKRS weights &
\emph{fails}: works in $121$ of $245$; supports collide in $131$ & \stverif\\
\quad the back-substitution, from GKRS data only & same &
$245/245$, remainder $0$, never $|c|>1$ & \stverif\\
Prop.~\ref{prop:eventransversal}: the same reading for even $t$ & $4$ pairs $(t,r)$ &
proved; $112/112$ support and values, top $101/101$, decoy $72$ & \stproved\\
\quad and the even identity, with denominator $\frac{1-z^t}{1-z^2}$ & $5$ pairs $(t,r)$ &
$266/266$, sign $+1$ throughout; decoys $0$, $0$, $5$ of $222$ & \stverif\\
Prop.~\ref{prop:galois}(i): units preserve both loci & $t\le12$, $n\le3$ &
proved; $118/118$ units, $0/58$ non-units & \stproved\\
Prop.~\ref{prop:galois}(ii): the $B\leftrightarrow C$ translation of \emph{supports}
& odd $t\le11$, $n\le3$ &
proved; $12/12$, and no other shift works & \stproved\\
\quad and it fails for even $t$ from rank $2$ & $t\le14$, $n\le3$ &
witness at $(6,2)$: $16$ vs $8$; sizes differ in $9$ of $9$ & \stverif\\
\quad but not at rank $1$, where $t\equiv2\ (4)$ works & $t\le14$ &
$3$ even cases are translates anyway & \stverif\\
Prop.~\ref{prop:galoissign}: $\tau_t(ka)=\gamma_t(k)\tau_t(a)$ & $t=3,5,7$ &
proved; $8/8$, $40/40$, $120/120$; $\gamma_t$ a character & \stproved\\
Cor.~\ref{cor:oddsign}: the odd filter in closed form & $t=3,5,7,9$ &
proved; $\epsilon_t=+1,-1,+1,+1$, controls $4/4$, $10/10$, $20/20$, $15/15$ & \stproved\\
\quad the two classical inputs: Gauss's lemma and \cite{Pan06} & odd $t\le61$ &
$788/788$ and $788/788$ pairs (unit, modulus) & \stverif\\
\quad $A_\rho\equiv(-2)(1,\dots,m')$, hence $\epsilon_t=\gamma_t(-2)$ & odd $t\le61$ &
$30/30$, and the closed value $30/30$ & \stverif\\
\quad the value is \emph{not} constant on an orbit & $t=4$, $m=1$ &
witness: $\tau(1)=+1$, $\tau(3)=-1$ & \stverif\\
the Galois saving: one evaluation per orbit, plus $\gamma_t$ & $t\le12$, $n\le3$ &
factor $=\varphi(t)$ once the action is free & \stverif\\
Lemma \ref{lem:checker}: the checkerboard & box $\le8$ &
proved; $3333$ pairs, $0$ failures & \stproved\\
$E^{(4)}$: transfer, triangular, $\{0,\pm1\}$ & $165\times45$ &
density $10.65\%$; $0$ triangularity violations & \stverif\\
\quad M\"obius structure: \emph{absent} & same &
row and column sums span $-4\dots+5$ & \stverif\\
\S\ref{sec:where}: cancellation inside a cell: \emph{no} & $7$ forms &
$4$ of $7$ have singleton cells only; $2$ of the $26$ cells sum to $0$ & \stverif\\
The atoms at the top weight & $7$ forms &
\textbf{one} or \textbf{three}, signs $2$ vs $1$ & \stverif\\
\quad opposite pair differing by one $t$-ribbon & same &
$2$ of $7$; decoy on other sizes fires in $0$ & \stverif\\
Values in $\{0,\pm1\}$ at the \emph{principal} torsion element & --- &
Nadimpalli--Pattanayak--Prasad \cite{NPP}, Thm.~4.1; Polo \cite{Polo};
Kostant \cite{Kostant} at $d=h$ & \stext\\
Our element is principal ($t$ odd) / is not ($t$ even, $m\ge2$) & all $t$ &
proved both ways: $\rho^\vee(\zeta)$ has spectrum $\zt$; and $2b_m\equiv k$ forces a parity
contradiction & \stproved\\
\quad and the multiset $\{\alpha(g)\}$ agrees, as an implementation control & $t\le11$ &
conjugacy decided the same way in every case & \stverif\\
The centraliser dichotomy decides vanishing & $t\le8$ &
$t$ odd: $488$ of $488$; $t$ even: $562$ of $688$ & \stverif\\
Singular dies, regular folds with a sign & --- &
Andersen--Stroppel \cite{AndersenStroppel} & \stext\\
Branching algebra with a standard monomial basis & --- &
Kim--Yacobi \cite{KimYacobi} & \stext\\
Vanishing when a residue class is empty & --- &
Littlewood \cite{LR34}; Ayyer--Kumari \cite{AK22} & \stext\\
\end{longtable}
\endgroup

\noindent
$^{\dagger}$ The four GKRS rows run over all dominant $\Lambda$ with $\Lambda_1$ bounded at
$(t,r)=(3,2),(5,2),(3,3)$ and $(7,2)$, $245$ weights in all.

\noindent
Two entries deserve to be read twice. The decoy for the order of the filter \emph{ties} on three
forms of thirteen: on those the wrong filter reproduces the object, so they are untested by it, not
confirmed. And the best discriminant on the residue is perfect inside one stratum and wrong inside
the next; it is listed because a vanishing criterion is exactly the kind of statement that is easy to
fit on twelve objects, and the second stratum is what stopped us from fitting it.

\subsection{The ancillary bundle}\label{sec:anc}
Everything the tables above appeal to travels with the paper, in three directories.

\begin{itemize}
\item \texttt{gates/} --- the computations behind the claims, $269$ scripts, of which $242$ carry the
archived run that produced their numbers as \texttt{*\_OUT.txt} and \texttt{*\_DUMP.json}. The
remaining $27$ are helpers and exploratory probes; none of them is the sole support of a number
quoted here.
\item \texttt{figures/} --- the $15$ scripts that draw the figures: the thirteen of the figures
themselves, and the two that redraw them with Spanish labels for the translated edition. Each
recomputes its data from the closed forms rather than reading a saved file, and each refuses to draw
if its controls fail, so no figure in this paper can disagree with its own data.
\item \texttt{audits/} --- the $25$ checks run on the manuscript itself rather than on the
mathematics: that every reference resolves, that the verification and attribution tables agree on the
status of each result, that no displayed formula has an unbalanced delimiter or a symbol used in two
senses, that every number in the text still matches the dump it came from, and that the two editions
cite the same sources the same number of times --- comparing the sets of keys is not enough, and a
citation present in one edition and missing in the other slipped past exactly that way once. One
of them refuses outright on a control character anywhere in either source: a shell heredoc once
wrote a literal backspace inside a macro name, \LaTeX{} compiled it without a murmur because the
byte is invisible, and the printed formula had eaten the macro.
\end{itemize}

\noindent
The header of each script states what it tests and, where the claim is a measurement, the
\emph{decoy} run against it. Scripts ending in \texttt{.sage} need SageMath; the rest are plain
Python~3, and the figure scripts additionally need matplotlib. A manifest lists all of them with a
one-line description and says which carry archived output.

We say plainly what the bundle is not. It is not a per-row map from the tables above to the scripts:
the tables are organised by statement and the scripts by computation, and the two groupings do not
coincide. What it is, is complete --- every number quoted has a script in the bundle, and $215$ of
them have the run archived beside the script, so that a reader who does not wish to execute anything
can still see what was executed.

\section{Attribution}\label{sec:attr}

We separate what we use from what we claim, in two lists rather than one, because the two are read
for different reasons.

\subsection*{What we use}
\begingroup
\rowcolors{2}{hband}{white}
\small
\begin{longtable}{@{}>{\raggedright\arraybackslash}p{0.58\linewidth}%
                   >{\raggedright\arraybackslash}p{0.34\linewidth}@{}}
\toprule
statement & due to \\
\midrule
\endhead
$t$-quotient factorisation at a root-of-unity orbit; vanishing when a residue class is empty
& Littlewood; Ayyer--Kumari \stext \\
reciprocal-pair factorisation into symplectic and orthogonal characters
& Ciucu--Krattenthaler; Ayyer--Behrend \stext \\
character at a \emph{principal} torsion element: the centraliser dichotomy, and its constant
& Nadimpalli--Pattanayak--Prasad; Polo \stext \\
the same at the Coxeter element
& Kostant \stext \\
Young-diagrammatic framework for restrictions to reductive subgroups of maximal rank
& Koike--Terada \stext \\
singular weight dies, regular weight folds with the sign of the folding element
& Andersen--Stroppel \stext \\
character on a reflection coset equals a character of the orbit algebra
& Jantzen; Kumar--Lusztig--Prasad \stext \\
multiplication by a unit as a permutation of $\ZZ/t$, and its sign
& Zolotarev; Frobenius; Lerch; Gauss's lemma in Jacobi's form \stext \\
the same on the \emph{folded} half system: $\sgn(\sigma_k)=\left(\frac kt\right)^{(t+1)/2}$, and that
it is a real even character
& Pan \cite{Pan06} \stext \\
signed Galois covariance of Weyl data: the Galois image differs by the sign of the affine Weyl
element that returns it to the alcove
& Fuchs--Schellekens--Schweigert \cite{FSSGalois} \stext \\
the character of an equal-rank pair as an alternating multiplet over a spinor denominator
& Gross--Kostant--Ramond--Sternberg; the cubic Dirac operator behind it, Kostant; the loop-group
extension, Landweber \stext \\
isobaric divided-difference operators, and that the one attached to $w_0$ is antisymmetrisation
followed by division by the Weyl denominator
& Demazure, in the form of Harada--Landweber--Sjamaar \cite{HLS} \stext \\
the same for a subgroup of maximal rank: the relative antisymmetriser, the operator identity
\eqref{eq:LS44}, and the vanishing of the alternating sum of a GKRS multiplet under the forgetful
map & Landweber--Sjamaar \cite{LS13} \stext \\
branching algebra of $\Sp_{2n}\downarrow\Sp_{2n-2}$ with a standard monomial basis
& Kim--Yacobi \stext \\
additivity of Newton polytopes under products, used in Proposition~\ref{prop:newtden}
& Ostrowski \stext \\
tableau models for the branching multiplicities $b_{\nu\mu}$, and the bridge between them
& Sundaram; Watanabe; Azenhas \stext \\
the multiplicity space at $t=4$ as a module for $(SL_2)^R$, and its diagonal
& Yacobi \stext \\
\bottomrule
\end{longtable}
\endgroup

\subsection*{What is ours}

\noindent
One thing has to be said before the table rather than inside it, because a table of rows hides it.
The odd branch of this paper --- \S\S\ref{sec:oddgkrs}--\ref{sec:det}, which is where the paper gets
furthest --- rests on an observation that is \emph{not ours}: that the obstruction we had hit is the
relative denominator of an equal-rank pair, and that the Gross--Kostant--Ramond--Sternberg formula
therefore applies. A correspondent made that observation. Everything from Proposition
\ref{prop:crossden} onwards --- the identity \eqref{eq:gkrsfilt}, the transversal count, the
inversion of the division, the determinant --- is downstream of it. We proved those statements and
we take responsibility for them, but we would not have been looking there. The same correspondent
supplied Lemma \ref{lem:muinj}, the divided-difference route, the two-line refutation of the affine
guess, the narrow toggle that closes $313$ of $313$, the form of Theorem \ref{thm:cancel} as an
identity, the Weyl-majorisation order in Conjecture \ref{conj:A}, and the counterexample in
Lemma \ref{lem:epsdet}. The route by which Theorem \ref{thm:cancel} is finally proved is the same
correspondent's: that the parity of a fibre should be read through
$\operatorname{per}\equiv\det\pmod2$ and the total unimodularity of $M$, rather than through a
cancellation mechanism on the fibre itself. Ours is the identification of the fibre with a permanent
(Lemma \ref{lem:permanent}), the laminarity (Lemma \ref{lem:laminar}) and the two counting lemmas
that follow it. The rows below repeat each of these in place; this paragraph is here so that
the sum of them is visible in one go.

\begingroup
\rowcolors{2}{hband}{white}
\small
\begin{longtable}{@{}>{\raggedright\arraybackslash}p{0.58\linewidth}%
                   >{\raggedright\arraybackslash}p{0.34\linewidth}@{}}
\toprule
statement & standing \\
\midrule
\endhead
the reduction of Proposition~\ref{prop:red}, both parities & proved \\
the explicit sign in Lemma~\ref{lem:T}, in type-$C$ coordinates & proved \\
Lemma~\ref{lem:regular}: the filter is regularity in $G$ itself
& proved (even); after \cite{NPP} (odd) \\
Theorem~\ref{thm:fusion}: both filters are the minimal fusion projections
& proved here; the fusion machinery is \cite{AndersenStroppel} \\
Proposition~\ref{prop:square}: the two factorisations commute & proved \\
Proposition~\ref{prop:newtden}: $\Newt(N_\delta)$ entirely --- the $C_r$ zonotope, hence an orbit
polytope with $2^rr!$ vertices --- and the dominant vertex, hence the shift
& proved; the additivity used is \cite{Ostrowski}, and the \emph{shape} of the answer is not a
surprise and not ours: for a plain Vandermonde $\prod_{i<j}(x_i-x_j)$ the same argument gives the
zonotope of the positive roots of type $A$, which is the permutohedron, and that is standard. What
is ours is the computation for \emph{our} alphabet --- that the long directions get radius $t+1$ and
the short ones radius $1$, and the resulting dominant vertex \\
Proposition~\ref{prop:ideal}: the vanishing locus is cut by an ideal & proved \\
Corollary~\ref{cor:selection}: the parity selection rule, even $t$ & proved \\
Corollary~\ref{cor:evensize}: the even filter is supported on $|\eta|$ even & proved \\
Lemma~\ref{lem:checker}: the checkerboard of signs & proved \\
Lemma~\ref{lem:epsdet}: $\varepsilon^{(t)}_{\lambda,\nu}$ as a determinant of zeros and ones
& proved, and a computational device only. The bound $\varepsilon\in\{0,\pm1\}$ does \emph{not}
follow from it: that is Littlewood's ribbon sign \cite{LR34}. The counterexample showing the
convention $h_k=0$ for $k<0$ cannot be dropped is a correspondent's \\
Lemma~\ref{lem:local}: the local form \eqref{eq:local} of the $r_i$
& proved; the $r_i$ themselves are Yacobi's \cite{Yacobi}, and what is ours is that under the
interlacing the global rearrangement defining them is unnecessary \\
Proposition~\ref{prop:galois}: the two \emph{nonvanishing loci} differ by $\sigma_2$
& proved; the Galois action itself is classical \\
Proposition~\ref{prop:galoissign}: at minimal level the Galois sign is a \emph{weight-independent}
character $\gamma_t$
& proved; the general signed covariance is \cite{FSSGalois}, and the two factors are Gauss's lemma
and \cite{Pan06} \\
Corollary~\ref{cor:galoisquad}: $\gamma_t$ is trivial or the quadratic character of
$\mathbb Q(\sqrt{-t})$ according to the parity of $n$, via $D_t^2=(-t)^n$
& proved, but the \emph{mechanism} is not ours and is old: extracting a quadratic character from a
matrix of roots of unity whose square is a scalar --- so that a square root of $\pm t$ appears and
the Galois group acts on it by a sign --- is Schur's linear-algebra route to the sign of the
quadratic Gauss sum \cite{Schur21,Murty}, and thence to quadratic reciprocity. Schur uses the \emph{trace} of the full
matrix $(\zeta^{rs})$; ours is the \emph{determinant} of its antisymmetric part, and the evaluation
$D_t^2=(-t)^n$ is elementary from the orthogonality of the discrete sine transform. What is ours is
the identification of $\gamma_t$ with that character, not the idea of looking there \\
Corollary~\ref{cor:oddsign}: the odd filter in closed form with its sign, and that the frozen vector
is the half system times the unit $-2$
& proved; the value of $\gamma_t(-2)$ is \cite{Pan06} + Gauss \\
Lemma~\ref{lem:auxC}: the auxiliary type-$C$ reading is regularity in $\mathcal R^C$ & proved \\
Lemma~\ref{lem:muinj}: $w\mapsto\mu_w$ is injective, hence $\nu\in\{0,\pm1\}$
& proved; \emph{the argument is a correspondent's} \\
Proposition~\ref{prop:divided}: the division by $\Delta_t=\psi^t(\Delta_1)$ inverts as a sum over an
arithmetic progression & proved, but little of it is new in itself: inverting a centred difference
operator on functions of finite support is elementary telescoping, and the divided-difference
formalism is \cite{HLS} with its maximal-rank sequel \cite{LS13}. What is ours is the identification
of the specialised denominator with that operator, and the consequence for (L1). The suggestion to
attack the division this way, the clean operator form $D_{t}=\prod_j(T_{-te_j}-T_{+te_j})$ we adopt,
and the caution about the character lattice are all a correspondent's \\
Proposition~\ref{prop:determinant}: the quotient is $\pm\epsilon_t\det M$ for an explicit $0/{\pm}1$
matrix & proved; that a signed sum over perfect matchings is a determinant is of course not ours ---
what is ours is the matrix \\
Theorem~\ref{thm:cancel}: a multi-hit fibre at a dominant regular index sums to zero
& proved; \emph{stating it as an identity rather than as the inequality (L1) is a
correspondent's recommendation}, and \emph{so is the route through
$\operatorname{per}\equiv\det\pmod2$} \\
Lemma~\ref{lem:permanent}: the fibre count is a permanent & proved; the bijection is already inside
the proof of Proposition \ref{prop:determinant} --- what is new is reading it without the signs \\
Lemma~\ref{lem:laminar}: the blocks are laminar & proved; ours, and the point is that it is not a
statement about $M$ at all but about the dominance with which the $X$ are enumerated \\
Lemma~\ref{lem:sdr}: the laminar count, and that it is odd only when it is $1$ & proved; \emph{the
product itself is classical} --- the lower bound on the number of systems of distinct representatives
is M.~Hall's \cite[Theorem 2]{Hall48}, and $\prod_k(|F_k|-k+1)$ its standard sharpening for a family
ordered by size; we claim nothing for either. Ours is that the laminar hypothesis makes the bound an
equality, and the parity consequence, which is the only part the paper uses \\
Lemma~\ref{lem:doubled}: a doubled entry forces an even fibre & proved; ours, and elementary \\
Corollary~\ref{cor:periodic}: the filter is a function on a finite set & proved; an immediate
consequence of Lemma \ref{lem:regular}, recorded because the rest of the paper computes with it \\
Conjecture~\ref{conj:A}: one $W(C_r)$-orbit polytope
& measured, not proved; that the order has to be Weyl majorisation and not the root order --- with
which the statement would be false on our own data --- we owe to a correspondent \\
Corollary~\ref{cor:unimodular}: $M$ is totally unimodular & proved, Lemma \ref{lem:c1p}; the
invariance of total unimodularity under signed row and column permutations, and the total
unimodularity of interval matrices, are elementary and standard --- the consecutive-ones theory is
Fulkerson--Gross \cite{FG65} \\
Proposition~\ref{prop:graph}: each block is a multigraph whose non-loop part lies inside
$K_{2,n-1}+\alpha\omega$, and its dependencies are that graph's cycles & proved; the boundary map and the suggestion to build the graph explicitly rather
than to record that a graphic matroid is present are a correspondent's; that the kernel of an
oriented incidence matrix is the cycle space is textbook. \emph{What is ours is the identification of
the graph} --- two terminals, one interior vertex per row boundary --- and Remark \ref{rem:cycles},
which is where the classification of the singular blocks closes \\
Corollary~\ref{cor:forest}: $c\ne0\iff|I_X|=1\iff$ every block graph is a forest & proved; ours,
and it is only the three preceding statements read in one line \\
Theorem~\ref{thm:odd}: the odd branch in one statement & proved, and \emph{nothing in it is new};
assembling the five clauses into a single theorem is a correspondent's recommendation \\
Proposition~\ref{prop:transversal}: the numerator is a signed transversal count, with a
computation-free vanishing criterion and an explicit top weight
& proved; but it is Lemma~\ref{lem:muinj} combined with Corollary~\ref{cor:oddsign}, and the first
of those is a correspondent's argument, so the proposition is ours only in the assembly \\
Proposition~\ref{prop:eventransversal}: the same reading for even $t$ --- two forbidden classes, no
chirality, denominator $\prod_j(1-z_j^t)/(1-z_j^2)$ --- and the parity dichotomy inside the numerator
& proved \\
Remark~\ref{rem:eventransversal}: the even identity, and the contrast between the two divisors
& verified \\
Proposition~\ref{prop:crossden}: the cross factor \eqref{eq:crossden}, and that it is the
specialisation of the \cite{GKRS} spinor denominator, collapsing cyclotomically to
$\prod_j(1-z_j^t)$
& proved; the general spinor denominator and its interpretation are \cite{GKRS,Kostant99}, and
\emph{the observation that GKRS applies here at all is a correspondent's} --- it is the hinge of the
whole odd branch \\
the identity \eqref{eq:gkrsfilt}: \cite{GKRS} composed with Proposition~\ref{prop:crossden} and the
odd filter & proved up to the global sign $s_r$, whose cyclotomic half $(-1)^r$ is now proved in
Prop.~\ref{prop:crossden} and whose residual convention factor we fixed by computation \\
Remark~\ref{rem:adams}: the specialised denominator is the $t$-th Adams image of the
$B_r\supset D_r$ spinor denominator & observed; \emph{the observation is a correspondent's} \\
the reduction of (L1) to the division by $\Delta_t$, and that the division reproduces $c$
& the reduction is proved, and so is that the division never makes a $2$ --- Theorem
\ref{thm:cancel} with Corollary \ref{cor:unimodular} \\
\S\ref{sec:whatcancels}: that affine folding cannot be the mechanism, since a shift by $2t$ is a
product of two reflections and carries sign $+1$
& proved; \emph{the two-line argument is a correspondent's} \\
\S\ref{sec:whatcancels}: the narrow toggle, with the chirality coupled to the folding sign, which
accounts for $313$ of $313$ cancelling pairs
& measured; \emph{the move is a correspondent's suggestion}, as is the prediction --- confirmed,
$167$ of $175$ --- that the sign comes from the representative and not from the jump \\
the $B/C$ discrepancy at $a_i=t/2$, and its witness in $\Sp_4$ & computed \\
the odd pipeline \eqref{eq:oddbranch}, and Proposition~\ref{prop:parity} & verified \\
the composite \eqref{eq:A} as a computational object & verified \\
the full formula \eqref{eq:mumax} for $\mu_{\max}$ & conditional on Conj.~\ref{conj:A} \\
Conjecture~\ref{conj:H} and its evidence, both parities & open \\
its localisation (L1)--(L3) on the odd side, and the Laplace route
that fails at (L1) (Remark~\ref{rem:L1route}) & measured; not proved \\
(L2) and (L3), sharpened into Conjectures~\ref{conj:tie} and \ref{conj:L3}
& measured; not proved \\
the residue of \S\ref{sec:residue} and the failed discriminants & computed \\
\bottomrule
\end{longtable}
\endgroup

\section{Open problems}\label{sec:open}

Each of the following is stated with the line of work it belongs to, because in every case the
lineage is what makes the question precise rather than merely unanswered. Figure \ref{fig:problems}
places the six against what the paper proves and what it borrows; they are the boxes it leaves open.

\begin{figure}[!ht]
\centering
\includegraphics[width=\textwidth]{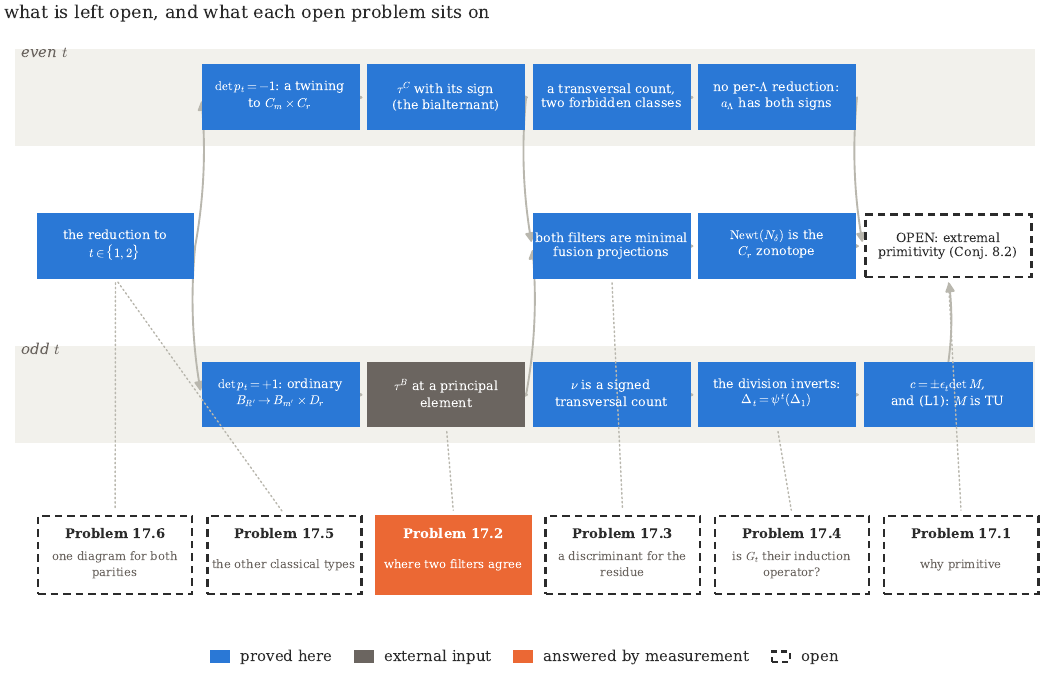}
\caption{The two routes to Conjecture \ref{conj:H}, and where each open problem attaches. Boxes are
statements and arrows are dependencies; the colour is the status column of \S\ref{sec:verif} ---
filled for what is proved here, grey for the one external input, dashed for what is open. The
\emph{upper} band is the even branch and the \emph{lower} one the odd, and the asymmetry between them
is the thesis of the paper rather than an accident of drawing: the odd branch reaches four boxes
further, because for a single $\Lambda$ its character is genuine and the even one is virtual. What
the two share is the spine in the middle, and Theorem \ref{thm:fusion} is where the fork closes.
Problem \ref{prob:threshold} carries the third colour because it is the one the paper answers rather
than asks: the level-threshold explanation it proposed is refuted here, and what replaces it is a
measured biconditional. Statements are named rather than numbered, and the problem numbers are read
from the \texttt{.aux}, so the figure cannot drift when the numbering does; the script refuses to
draw unless the set of problems it shows is exactly the set the section states --- which is how the
merge of two duplicate problems into Problem \ref{prob:idealodd} was kept from leaving a stale
figure behind --- and unless every label measures inside its own box, in pixels. Computed by
\texttt{fig\_problems.py}.}
\label{fig:problems}
\end{figure}

\begin{problem}[why primitive]\label{prob:unit}
Prove Conjecture \ref{conj:H}, or find the shape it really has. The difficulty is easy to locate and
hard to remove: $A_\mu$ is the value at a torsion element of the character of a \emph{virtual}
module, whereas the theory available --- Lemma \ref{lem:regular} here, and \cite{NPP} for a principal
element --- constrains the value of an \emph{irreducible} character, and at a torsion element whose
relation to $\eta$ is fixed. An integer combination of values in $\{0,\pm1\}$ is a priori an
arbitrary integer, and on our data the individual terms are not small.

The parity split of Proposition \ref{prop:parity} suggests attacking the two halves differently. On
the odd side the input to the filter is a genuine character, the torsion element is principal, and
the localisation of \S\ref{sec:oddlocal} has the top weight reached by very few surviving $\eta$
--- in every form tested, exactly one $\Lambda$ contributes there, which is (L2) and is measured, not
proved; an
extremal-branching argument for the pair $SO_{2R'+1}\downarrow SO_{2m'+1}\times SO_{2r}$ may
therefore suffice there. On the even side it will not, and the promising coordinates are those of
\S\ref{sec:where}, where the same integer is a signed count of one or three atoms.

We can say more precisely where the odd side now stands, because we tried it. The conjecture
localises into (L1)--(L3) of \S\ref{sec:oddlocal}, of which (L1) --- that a single $\Lambda$
contributes $0$ or $\pm1$ --- is the representation-theoretic half and the one that looked reachable.
The Laplace expansion of the $B_{R'}$ Weyl numerator along the frozen rows does give an injective
correspondence between regular subsets and weights, and would give (L1) with a closed formula, but
the normalisation defeats it: the denominator does not reduce to one term, and the trivial
representation is already a counterexample (Remark \ref{rem:L1route}). What a proof needs is that
argument carried out with the denominator in product form. We record this so that the next attempt
starts after the obstruction rather than before it.

That is where the obstruction stopped being anonymous, and the rest of this paragraph is what
happened when it was followed. Proposition \ref{prop:crossden} identifies the cross factor as the
relative denominator of the equal-rank pair $B_{R'}\supset B_{m'}\times D_r$, which puts the question
inside the Gross--Kostant--Ramond--Sternberg formula \cite{GKRS}; specialising that formula at the
torsion point gives \eqref{eq:gkrsfilt}, Lemma \ref{lem:muinj} makes its numerator
multiplicity-free after the projection, and Proposition \ref{prop:transversal} turns the numerator
into a signed count of transversals with values in $\{0,\pm1\}$ and an explicit top weight. So (L1)
is no longer a question about branching multiplicities but about one division --- and that division
has since been carried out. The divisor is the $t$-th Adams image of the $B_r\supset D_r$ spinor
denominator (Remark \ref{rem:adams}), so it is a product of commuting centred difference operators
and inverts in closed form, as a sum over a fibre (Proposition \ref{prop:divided}); regrouped over
signed permutations instead of subsets, that sum is $\pm\epsilon_t\det M$ for an explicit $0/{\pm}1$
matrix (Proposition \ref{prop:determinant}). (L1) then follows in either of two ways,
and they are not equivalent: the identity that a fibre meeting the support more than once contributes
nothing (Theorem \ref{thm:cancel}), or the total unimodularity of $M$ (Corollary
\ref{cor:unimodular}), which gives the bound outright. Both are now available, and the second is
what proves the first; \S\ref{sec:map} says why that was the form to attack. That is the sharpest thing we can say about Conjecture \ref{conj:H} on either side.

One reading we have deliberately not pursued: primitivity invites the language of simple currents
\cite{FSS}. Two of the three things that would have to hold are now available --- Theorem
\ref{thm:fusion} shows that the object does define a class in a fusion quotient, and that this class
is what the torsion evaluation computes. The third is not, and it is the decisive one: that $\pm[0]$
comes from a genuinely invertible object rather than from a virtual combination that cancels. Since
$M_{\mu_{\max}}$ is virtual on the even side by Remark \ref{rem:virtual}, that is exactly where the
question bites. We raise it as speculative for that reason and no other.
\end{problem}

\begin{problem}[where two filters agree]\label{prob:threshold}
On three forms out of thirteen the filter of the wrong order reproduces the object exactly, so those
forms are untested by that decoy. An earlier version of this problem proposed the explanation and
asked for a proof; we then measured it, and the explanation is wrong. We leave the whole route
visible, because what replaces the question is sharper than the question was.

\emph{What was asked.} In type $A$ the analogous phenomenon has a clean answer: at
sufficiently high level no wall is crossed and fusion coefficients coincide with tensor product
multiplicities, with the threshold given explicitly \cite{Feingold}. For type $C$ the alcove is also
explicit --- for even order it is $\sum_i m_i < \ell/2 - n$ in the fundamental-weight coordinates
\cite{AndersenStroppel} --- so the question seemed to pose itself: do two filters of different order
agree on $\Phi_{t,r}$ exactly when the support of $B_{\eta,\mu}$ lies inside the alcove of the
smaller one?

\emph{Why that cannot be it.} The two alcoves are the same one. Our decoy is not the same group at a
higher level: it is $\tau$ at order $t+2$ on $C_{m+1}$, with $\eta$ padded by a zero. In
fundamental-weight coordinates $\sum_i m_i=\eta_1$, so the alcove of $C_n$ at order $\ell$ is
$\{\eta_1<\ell/2-n\}$, and
\[
\text{true filter: }\ \tfrac t2-m=1,
\qquad
\text{decoy: }\ \tfrac{t+2}{2}-(m+1)=1 .
\]
Both sit at the minimal level, for every $t$; there is no smaller alcove for a support to be inside,
and no threshold to cross \stverif. The reason is structural rather than numerical: raising the order
by two raises the rank of the frozen group by one, and the two moves cancel in $\ell/2-n$ exactly.

\emph{What the measurement gives instead}, and it is a biconditional with no exception. The wrong
filter reproduces $\Phi_{t,r}$ precisely when it agrees with the right one \emph{pointwise on}
$\mathrm{supp}\,B$: on the three tying forms the two filters agree at every one of the $4$, $14$ and
$10$ weights of the support, and on each of the other ten they differ somewhere, from one weight to
eight. Thirteen of thirteen \stverif. So the tie is not a collective cancellation in
$\sum_\eta B_{\eta,\mu}\tau(\eta)$ --- which is what a level threshold would have had to be --- but
the filter itself, evaluated where the two rules happen to coincide. The decoy at order $t+4$ ties on
exactly the same three forms, which is what an agreement-region explanation predicts and a threshold
one does not.

\emph{The problem, restated.} Describe
$\mathcal A_t=\{\eta:\tau^C_t(\eta)=\tau^C_{t+2}(\eta,0)\}$. It is not simply an alcove condition,
and the two parities of the rank already differ: at $t=6$, rank two, the two rules agree on every
$\eta$ with $\eta_1\le4$ and first differ at $\eta_1=5$, so there the region \emph{is} an initial
segment; at $t=4$, rank one, they agree at $\eta_1\in\{0,1,2,3,5,7,9,11,12\}$ and differ at
$\{4,6,8,10\}$, which is no initial segment at all \stverif. A description of $\mathcal A_t$ would
turn the three ties from a gap in a decoy into a statement, and would say which decoy to run instead:
one that moves the support off $\mathcal A_t$. Consistently with all of this, the maximum of $\eta_1$
over $\mathrm{supp}\,B$ does separate the ties inside each configuration --- $3$ against
$5,6,7,7,11,13,17$ at $t=4$, and $3$ and $4$ against $6$ and $8$ at $t=6$ --- but the cut is not the
same number in the two, so it is a symptom of $\mathcal A_t$ and not a threshold.

A second measurement, from an unrelated test, is left as it was found. Classifying how the residue
vanishes, the decoy that swaps the filter for the one of the next order changes the answer at $t=4$
and \emph{does not change it at all} at $t=6$: the same four inherited, the same eight by
cancellation. We had read that as a threshold too. It is now one more instance to explain from
$\mathcal A_t$, and we have not done it.
\end{problem}

\begin{problem}[a discriminant for the residue]\label{prob:residue}
Find a criterion for the residual vanishing, or show that none of a given shape exists. What we can
say is where it is not: not the emptiness of the $t$-core, not the stable-range rectangles that
settle the inherited half, not the parity, and not the quotient component at the self-paired residue,
which is perfect inside one stratum and wrong inside the next. What we can say is a necessary
condition --- a parity, and membership of one of three $t$-cores out of thirty-six.

We add a methodological remark rather than a hint. With a residue of a dozen forms no conjunction of
two features can be resolved, since the candidate rules have more freedom than the data. The next
step is a larger population, not more features, and any rule proposed before that is fitting noise.
\end{problem}

\begin{problem}[is $G_t$ their induction operator?]\label{prob:LS}
Landweber and Sjamaar \cite{LS13} construct, for a closed subgroup of maximal rank, a relative
twisted $\mathrm{Spin}^c$-induction and a character formula for it, together with GKRS-type
multiplets, in equivariant K-theory. Our $G_t=D_t^{-1}$ is a coefficient-level object: an inverse of
a centred difference operator on functions of finite support, obtained after specialising at the
torsion element and applying the fusion projection. Is $G_t$ the specialisation of their relative
induction operator for $B_{R'}\supset B_{m'}\times D_r$, and if so does their character formula
already contain Proposition \ref{prop:divided}? We can state the question precisely and cannot
answer it; an affirmative answer would place \eqref{eq:divided} inside an existing theory rather than
beside it.

The sharper form of the question, and the one we would actually like answered, is the one Remark
\ref{rem:LSgap} isolates. Their identity \eqref{eq:LS44} yields the vanishing of the alternating sum
of a multiplet \emph{after} the forgetful map, and we have checked that our $\nu$ satisfies it. Our
Theorem \ref{thm:cancel} asks for the vanishing on each fibre of the coefficient extraction of
the $\psi^t$-dilated operator, separately, and the fibrewise version of their weighting is false. So:
can \eqref{eq:LS44} be refined coefficientwise --- can the single vanishing
$\prod_{\alpha\in R^+_M}(1-1)=0$ be resolved into one vanishing per fibre? That refinement would turn
their global multiplet identity into our local one. Theorem \ref{thm:cancel} and (L1) are now proved
by elementary means, so the question is no longer whether their theory can \emph{prove} them but
whether it \emph{explains} them: why should a relative induction identity in $K$-theory produce,
after specialisation, exactly this graphic-matroid and permanent structure?
\end{problem}

\begin{problem}[the other classical types]\label{prob:types}
The companion paper asks what the same alphabet does to $sp_\lambda$, $so_\lambda$ and $o_\lambda$,
reports measurements across the types and finds no law covering the whole table. The instrument of
this paper was not available then, and it is built for the question: the composite of \eqref{eq:A} is
a branching followed by an evaluation, and neither step is specific to the Schur case. What we do not
know is whether the two steps survive the change of type in the same form, and in particular whether
the analogue of Conjecture \ref{conj:H} holds --- extremal primitivity is the part of the structure with the
least reason to be stable, and therefore the part worth testing first.
\end{problem}

\begin{problem}[one diagram for both parities]\label{prob:idealodd}
An earlier version of this problem asked for a minimal generating set of the fusion ideal at odd
order in type $C$, on the grounds that our odd filter lived in that regime. It does not: the odd
filter is of type $B$ (\S\ref{sec:odd}), and Theorem \ref{thm:fusion} identifies its quotient
without needing minimality. The question that replaces it is the one Theorem \ref{thm:fusion} makes
natural. Both pipelines now read
\[
\text{branching}\ \longrightarrow\ \text{minimal fusion projection on the frozen factor},
\]
with the parity changing only the first arrow --- twined $C_R\rightsquigarrow C_m\times C_r$ for even
$t$, ordinary $B_{R'}\to B_{m'}\times D_r$ for odd. Construct a single functorial diagram producing
both, so that Proposition \ref{prop:parity} becomes a corollary of one statement rather than a table
of two.

Half of that is now available, and saying which half is the useful part of the problem. Both
inclusions have \emph{equal rank}, both therefore obey the same equal-rank character formula, and
after the frozen specialisation the two differ only in which positive roots are complementary:
$e_i\pm f_j$ together with $f_j$ in the odd case, giving the divisor $\prod_j(1-z_j^{\,t})$ of
Proposition \ref{prop:crossden}; and $e_i\pm f_j$ alone in the even one, giving
$\prod_j(1-z_j^{\,t})/(1-z_j^{2})$, by Remark \ref{rem:eventransversal}. The numerator is a signed
transversal count on both sides, with one forbidden residue class instead of two and a chirality
factor instead of none. So the second and third arrows are one construction already.

What is not uniform is the \emph{first} arrow, and it has shrunk to a single sign: Proposition
\ref{prop:red} is stated uniformly in the parity, and what decides whether the object then lives on
$C_R$ or on $B_{R'}$ is $\det p_t=(-1)^{t+1}$ of \eqref{eq:det}. The problem is therefore no longer
``unify a table of two'' but ``exhibit a construction over $O(N)$ whose restriction to each component
gives the corresponding branching'', with the component read off that determinant. We do not have it.

Two things about that first arrow are worth stating so that they are not looked for twice. There is
no orbit algebra of the \emph{odd} object to find: by \eqref{eq:det} the odd point lies in the
identity component, the first step is an ordinary restriction to $SO_{2R'+1}$, and no twining occurs.
Only the even side has one, and there it is classical --- the frozen block $\{1,-1\}$ with the
reciprocal pairs makes the alphabet a reflection coset, and the algebra whose characters compute the
value is the \emph{orbit} Lie algebra of that coset, the Langlands dual of the fixed-point
subalgebra, which for our coset is $C_R$; that is Jantzen's twining theorem \cite{Jantzen} in the
form of Kumar--Lusztig--Prasad \cite{KLP}. And one attempt failed, so that it is not repeated: the
question cannot be settled by asking which basis the object expands in. $B_r$ and $C_r$ share a Weyl
group, so their irreducible characters are two bases of the \emph{same} ring of invariants and both
expansions close with integer coefficients, always. The expansion is indeed consistently shorter in
the symplectic basis for even $t$ and in the odd orthogonal basis for odd $t$ --- three of three and
five of five in our range --- but economy of an expansion is a heuristic, and the determinant is not.
\end{problem}

\section{Closing: one parity, seen three times}\label{sec:closing}

It is worth saying plainly what the paper found, because the same elementary fact surfaces three
times in three different formalisms and it would be easy to read the three as unrelated.

The fact is the parity of $t$. It appears first as a \emph{determinant}: by \eqref{eq:det} the
evaluation point $p_t$ has $\det p_t=(-1)^{t+1}$, so for odd $t$ it lies in the identity component
and for even $t$ it does not. That single sign decides which group the object lives on, hence
whether the first arrow of \eqref{eq:A} is an ordinary restriction or a twining, and hence the
entire dichotomy of \S\ref{sec:odd}: $B_{R'}\to B_{m'}\times D_r$ against
$C_R\rightsquigarrow C_m\times C_r$.

It appears a second time as an \emph{invertibility}. The two conditions $a_i\equiv0$ and
$2a_i\equiv0$ are different at an even modulus and equivalent as soon as $2$ is invertible, which is
Remark \ref{rem:BC}; so the extra wall exists for even $t$ and not for odd. Proposition
\ref{prop:galois} makes this an identification rather than an inspection: for odd $t$ the two
\emph{nonvanishing loci} are exchanged by the Galois element $\sigma_2$ and are translates of one
another (Figure \ref{fig:galois}), while the values differ by the character $\gamma_t$ of
Proposition \ref{prop:galoissign}. This is not the determinant restated: one is a statement about the
component group, the other about root lengths, and they happen to be governed by the same integer.

It appears a third time as a \emph{ribbon-size parity}, in Corollary \ref{cor:selection}. A surviving
orbit-side term differs from $\lambda$ by $kt$ cells, while the subsequent symplectic restriction
changes the size by an even number; hence for even $t$ every surviving coefficient lies in the parity
class of $|\lambda|$; that is Corollary \ref{cor:selection}, and at $t=4$ its branching-side
incarnation is exactly the checkerboard of Lemma \ref{lem:checker}. Note that no \emph{sign} enters here: the argument is about sizes modulo $2$, not
about the sign of a ribbon tiling, which depends on the heights of the ribbons and not on how many
were removed. Here the formalism is combinatorial and owes nothing to either of the previous two.

We would call this a third \emph{manifestation} rather than a third independent mechanism. Three
formalisms --- algebraic groups, root systems, ribbon combinatorics --- each produce a visible
difference between odd and even order, and the paper's structure is largely the record of following
each of them to where it leads.

One caution, because the temptation to make everything the parity is strong and would be wrong. The
\emph{signs} do not see the parity of $t$; they see a finer invariant. By Corollary
\ref{cor:galoisquad} the Galois character $\gamma_t$ is trivial exactly when $t\equiv1,2\pmod4$ ---
uniformly in the two branches --- and by
Corollary \ref{cor:oddsign} the normalisation of the odd filter is $-1$ exactly when
$t\equiv5\pmod8$ --- in the reading order of Lemma \ref{lem:T}, and Remark \ref{rem:convention} says
what the other order does to that sentence. Those are refinements of the parity, not further
instances of it, and their
mechanism is classical rather than ours --- Gauss's lemma and \cite{Pan06}. The honest summary is
that the \emph{support} of everything in this paper is governed by $t$ modulo $2$, and the
\emph{global Galois and normalisation signs} $\gamma_t$ and $\epsilon_t$ by $t$ modulo $4$ and $8$.
Not every sign in the paper: the individual ones carry straightening, shuffle, folding and ribbon
heights, and none of those is a function of $t$ alone.

The conjecture that remains open, Conjecture \ref{conj:H}, is the one place where the formalisms do
not yet meet, and the two parities are no longer equally far from it. The odd side now reaches the
numerator: by Proposition \ref{prop:transversal} it is a signed transversal count with values in
$\{0,\pm1\}$, and the division by the specialised spinor denominator is carried out in closed form
in \S\ref{sec:invert}, so on the odd side \emph{(L1)} is settled and what remains is \emph{(L2)}
and \emph{(L3)}. The even side has
the atoms of \S\ref{sec:where} and no such reduction, because the equal-rank formula there acts on a
virtual character. No argument we have covers both.

To close with the same precision the abstract opens with, and the distinction matters because
Proposition \ref{prop:divided} is one of the things here we are fondest of: of the three localised
statements (L1)--(L3), (L1) is now proved entire: its numerator, \emph{and the division} --- explicit
in closed form, as a sum along an arithmetic progression --- \emph{and} the bound on the fibre sum
the division produces, in the stronger form of Theorem \ref{thm:cancel}, that every multi-hit fibre
at a dominant regular coefficient index sums to zero. What closed it was not a cancellation mechanism on the fibre but the observation that
the fibre count is a permanent, together with the dominance under which the $X$ are enumerated.
(L2) and (L3) are
\emph{measured and not proved}; Conjectures \ref{conj:tie} and \ref{conj:L3} settle the case where
the extremal weight is attained once and leave the tied case open. Every
statement in this paper marked \stverif\ in the tables of \S\ref{sec:verif} is a computation and not
a theorem, and we would rather leave the frontier visible than let a table of agreements read as a
proof.

\subsection*{Code, data and disclosure}
Every computation quoted in this paper is performed by a named script distributed as an ancillary
file, and every number quoted is reproduced by an archived run of one: $269$ scripts for the claims,
of which $242$ carry their archived output, $15$ that draw the figures, and $25$ that check the
manuscript itself. Section \ref{sec:anc} describes the bundle and a manifest names every script and
what it does; where a statement is a measurement rather than a theorem, the entry in
\S\ref{sec:verif} also records the decoy that was run against it. The same scripts and the same
archived output are kept at
\url{https://github.com/karlesmarin/schur-orbit-and-reciprocal-pair}, which is a convenience: the
ancillary files carry everything the paper appeals to. Generative AI (Claude, Anthropic) was used
throughout as a research assistant --- for literature search, for writing and checking the ancillary
code, and on the prose. The mathematics, and the responsibility for it, are the author's.

\end{document}